%% file: monadic2.tex
\documentclass{amsart}
\usepackage[textsize=scriptsize,colorinlistoftodos]{todonotes}

\usepackage[margin=1in]{geometry}
\usepackage{etoolbox}

\usepackage{longtable}

\newtoggle{final}
\toggletrue{final}

\usepackage{tikz}
\usetikzlibrary{matrix,arrows}
\usepackage{tikz-cd}

\input{preamble2.tex}
\usepackage[capitalise]{cleveref}
\usepackage{mathtools}
\theoremstyle{plain}
\newtheorem{introtheorem}{Theorem}

\crefname{theorem}{Theorem}{Theorems}
\crefname{introtheorem}{Theorem}{Theorems}
\crefname{lemma}{Lemma}{Lemmas}
\crefname{construction}{Construction}{Constructions}
\crefname{proposition}{Proposition}{Propositions}
\newcommand{\Sp}{\operatorname{Sp}}
\newcommand{\Pro}{\operatorname{Pro}}
\def\Grp{\mathrm{Grp}}
\def\Grpd{\mathrm{Grpd}}
\def\cart{\mathrm{cart}}
\def\epi{\mathrm{epi}}
\def\sld{\mathrm{sld}}
\newcommand{\comp}{{{\kern -.5pt}\wedge}}
\newcommand{\Tot}{\operatorname{Tot}}
\newcommand{\B}{\mathrm{B}}
\DeclareMathOperator{\mat}{mat}
\DeclareMathOperator{\CB}{CB}

\iftoggle{final}{
\newcommand{\NB}[1]{}
\newcommand{\todolater}[1]{}

\renewcommand{\todo}[1]{}

\newcommand{\todoklaus}[1]{}
\newcommand{\todotom}[1]{}
}{ 
\newcommand{\NB}[1]{\todo[color=gray!40]{#1}}

\newcommand{\todolater}[1]{\todo[color=brown!40]{#1}}

\newcommand{\todoklaus}[1]{\todo[color=blue]{#1}}
\newcommand{\todotom}[1]{\todo[color=green]{#1}}
}

\newcommand{\axiom}[3]{\item[(#1)] \hypertarget{axiom:#2}{#3}}
\newcommand{\axiomref}[1]{\hyperlink{axiom:#1}{(#1)}}

\def\MAP{\mathrm{MAP}}
\def\ADJ{\scr{ADJ}}

\def\inj{\mathrm{inj}}
\def\surj{\mathrm{surj}}

\def\Shv{\mathrm{Shv}}

\def\Cube{\mathrm{Cube}}

\newcommand\noloc{%
  \nobreak
  \mspace{6mu plus 1mu}
  {:}
  \nonscript\mkern-\thinmuskip
  \mathpunct{}
  \mspace{2mu}
}

\title{Monadic resolutions for generalized spaces}
\date{\today}

\author{Tom Bachmann}
\email{\tomemail}
\urladdr{\href{http://www.tom-bachmann.com}{www.tom-bachmann.com}}

\author{Anton Engelmann}
\email{\href{mailto:anton.engelmann@uni-mainz.de}{anton.engelmann@uni-mainz.de}}
\urladdr{\href{https://anton-engelmann.com}{www.anton-engelmann.com}}

\author{Klaus Mattis}
\email{\href{mailt:klaus.mattis@uni-mainz.de}{klaus.mattis@uni-mainz.de}}
\urladdr{\href{http://www.klaus-mattis.com}{www.klaus-mattis.com}}

\begin{document}

\begin{abstract}
  We extend the work of Bousfield and Kan \cite{BousfieldKan:HomotopyLimitsCompLoc} on monadic resolutions of spaces to $\infty$-topoi, with applications to genuine $G$-equivariant spaces ($G$ a finite group) and motivic spaces over a perfect field.
  In particular, we give a proof of the principal fibration lemma in this context.
  We apply the principal fibration lemma to prove convergence of several kinds of monadic resolutions in unstable equivariant and motivic homotopy theory.
  For example, we show that, over an algebraically closed field, the unstable Adams--Novikov spectral sequence (i.e., the monadic resolution corresponding to the algebraic cobordism spectrum $\mathrm{MGL}$) converges for all nilpotent, connected, $2$-effective motivic spaces.
\end{abstract}
\maketitle
\tableofcontents
\newpage

\iftoggle{final}{}{
\listoftodos\relax
\newpage
}

\section{Introduction}
Unstable homotopy theory is concerned with the study of spaces and their maps.
Because spaces are nonlinear and wildly complicated, a standard strategy is to pass to linear invariants that are easier to compute and compare, such as homology.
An evident question arises, which is \emph{how much information is lost in this process?}, or \emph{how much of a space can we recover from its homology?}
When taking into account enough structure on the homology, there is a natural ``best possible approximation'' we can extract from the homology, known as the \emph{Bousfield--Kan completion} (of a space at a homology theory).
In their remarkable work \cite{BousfieldKan:HomotopyLimitsCompLoc}, Bousfield and Kan identified essentially two situations in which the Bousfield--Kan completion can be described somewhat explicitly:
\begin{theorem*}[Bousfield--Kan]
  Let $X \in \Spc_*$ be a pointed nilpotent space. 
  \begin{itemize}
    \item If $p$ is a prime and the homology theory is $H\F_p$, then the corresponding Bousfield--Kan completion coincides with the unstable $p$-completion.
    \item If $S \subset \Z$ is a multiplicative subset and the homology theory is $H\Z[S^{-1}]$, then the Bousfield--Kan completion coincides with the unstable $S$-localization.
  \end{itemize}
For example, the Bousfield--Kan completion of any nilpotent space at $H\Z$ recovers the space itself.
\end{theorem*}

The aim of this article is to provide a treatment of the ideas of Bousfield--Kan in modern language, and apply them to other kinds of spaces.
In particular, we will prove a Bousfield--Kan completion theorem for categories of sheaves as well as genuine equivariant and motivic spaces.

\subsection{Digression: categorical formulation}
Let \[ L \colon \scr C \adj \scr D \noloc R \] be an adjunction between presentable $\infty$-categories.
Suppose we attempt to study the $\infty$-category $\scr C$ through the lens of the functor $L$.
A first observation is that if $X \in \scr C$, then the object $LX \in \scr D$ carries additional structure, e.g., the (adjunction) map $L(RL)X \wequi (LR)LX \to LX$.
In a more structured way, we can observe that $C \coloneqq LR \colon \scr D \to \scr D$ is a comonad, and $LX$ is a coalgebra for the comonad $C$ (see \cite[the dual of Proposition 4.7.3.3]{higheralgebra} and the discussion thereafter).
A natural question would be to ask if $X$ can be recovered from $LX$ with this additional structure.
It is clear that this is not possible in general, as we now explain.
An $L$-equivalence in $\scr C$ is a map inverted by $L$.
These form a strongly saturated class of small generation, and so, there is an associated Bousfield localization functor $\hat L\colon \scr C \to \scr C$ \cite[§5.5.4]{highertopoi}.
We can thus factor our starting adjunction as \[ \scr C \adj \hat L \scr C \adj \mathrm{Coalg}_C(\scr D) \adj \scr D, \] where the first adjunction is a Bousfield localization and the third is a forgetful/cofree adjunction.
Clearly the best we can hope for is that the middle adjunction is fully faithful, perhaps on some suitable subcategory of $\hat L \scr C$.

\begin{remark}
Note that the composite \[ \scr C \xrightarrow{\hat L} \mathrm{Coalg}_C(\scr D) \to \scr C \] (where the second functor is right adjoint to the first) is equivalent to the \emph{cobar construction} (also called the \emph{Bousfield--Kan completion}) \[ X \mapsto \CB(X) := \lim_{[n] \in \Delta} (RL)^{n+1} X. \]
By construction, this object is $L$-local and receives a map from $X$.
So we can rephrase our initial question as follows: for which $X \in \scr C$ is the canonical map \[ \hat L X \to \CB(X) \] an equivalence?
\end{remark}

\subsection{Results}
Our results fit into the categorical framework of the previous section.
Typically, in addition to $\scr C \adj \scr D$ we pick a ring $A \in \CAlg(\scr D)$ and consider the composite adjunction \[ \scr C \adj \Mod_A(\scr D). \]
We denote the resulting Bousfield localization by $\hat L_A$.
Our goal is to exhibit (ideally many) examples of $X \in \scr C$ such that $\hat L_A$ coincides with the Bousfield--Kan completion of $X$ (for the above adjunction).
Moreover, we seek to identify $\hat L_A$ in more familiar terms.
In order to do this, we prove results about Bousfield--Kan completions in a relatively general context.
For a summary, see \cref{subsec:proof-ingredients}.
We then apply these general results in settings of interest, and combine them with additional assumptions and techniques specific to the settings, to obtain useful criteria in these specific situations.
Here we concentrate on explaining the latter.

\subsubsection{Sheaves}
Our first application is the case where $\scr C=\scr X$ is an appropriate $\infty$-topos and $\scr D = \Sp(\scr X)$ is its stabilization.
Recall that in this case there are notions of \emph{unstable $n$-completion} and \emph{unstable $S$-localization}, see \cref{sec:unstable-n-comp}, \cite[Section 3]{mattis2024unstablepcompletionmotivichomotopy} and \cite[Section 3]{mattis2024unstablearithmeticfracturesquares}.
Generalizing the situation for classical spaces, these are given by Bousfield localization endofunctors \[ L_S, L_n\colon \scr X \to \scr X. \]
We can only understand these functors to a reasonable degree on \emph{nilpotent sheaves}, i.e., sheaves built out of infinite loop sheaves in a controlled manner \cite[Appendix A.2]{mattis2024unstablepcompletionmotivichomotopy}.

\begin{introtheorem}[Sheaves, {\cref{thm:sheaves:main-thm}}] \label{thm:main:sheaves}
Let $\scr X$ be an $\infty$-topos admitting a locally finite-dimensional cover \cite[Definition 5.1]{mattis2024unstablearithmeticfracturesquares} (see 
\cref{rmk:locally-finite-dimensional-examples} for examples),
let $A \in \CAlg(\Sp(\scr X)_{\ge 0})$, and let $X \in \scr X_*$ be nilpotent.
Then, depending on $\ul\pi_0(A)$ we can identify $\hat{L}_A X$ as:
\begin{center}
\renewcommand{\arraystretch}{1.2}
\begin{tabular}{l|c}
$\ul\pi_0(A)$ & Localization functor $\hat L_A$ \\
\hline
$S^{-1}\ul\Z \subset \ul\Q$ & $L_S X$ \\
$\ul{\Z}/n$ & $L_n X$
\end{tabular}
\end{center}
In these cases, we can moreover compute the localization of $X$ with respect to $A$ by the totalization formula
\begin{equation*}
  \hat L_A X \wequi \tau_{\ge 1} \lim_{\Delta} (\Omega^\infty A \otimes \Sigma^\infty)^{\bullet+1}(X).
\end{equation*}

\end{introtheorem}

\subsubsection{Equivariant spaces}
Our second application is to genuine equivariant homotopy theory.
For a finite group $G$, we denote by $\Spc(\B G) \wequi \PSh_\Sigma(\Fin_G)$ the $\infty$-category of genuine $G$-spaces, and by $\SH(\B G)$ its stabilization at representation spheres.
See \cref{subsec:equivariant} for more details.
\begin{introtheorem}[Equivariant homotopy theory, {\cref{thm:equivariant:main-thm,thm:equivariant:Z:loc:main-thm,thm:equivariant:Z:completion:main-thm}}] \label{thm:main:equivariant}
Let $G$ be a finite group, $A \in \CAlg(\SH(\B G)_{\ge 0})$, and $X \in \Spc(\B G)_*$ a nilpotent and $(1,G)$-connective (see \cref{def:equivariant:n-G-connective}) $G$-space.
Then, depending on $\ul\pi_0(A)$ we can identify $\hat{L}_A X$ as:
  \begin{center}
  \renewcommand{\arraystretch}{1.2}
  \begin{tabular}{l|c|c}
  $\ul\pi_0(A)$ & Additional assumptions & Localization functor $\hat L_A$ \\
  \hline
  $S^{-1}\ul\pi_0(\1)$ & & $L_S X$ \\
  $\ul\pi_0(\1)/n$ & & $L_n X$ \\
  $S^{-1}\ul\Z$ & $|G| \in S$ & $i_\ast L_S i^\ast X$ \\
  $\ul{\Z}/p$ & $G$ a finite $p$-group & $L_p X$
  \end{tabular}
  \end{center}
In these cases, we can moreover compute the localization of $X$ with respect to $A$ by the totalization formula
\begin{equation*}
  \hat L_A X \wequi \lim_{\Delta} (\Omega^\infty A \otimes \Sigma^\infty)^{\bullet+1}(X).
\end{equation*}
\end{introtheorem}

\subsubsection{Motivic spaces}
Our final application is to motivic homotopy theory.
We fix a perfect field $k$ and denote by $\Spc(k)$ the $\infty$-category of motivic spaces over $k$ \cite{Morel:Theoriehomotopique,MorelVoevodsky:a1ofschemes}.
Denote its stabilization at $\P^1$ by $\SH(k)$.
See \cref{subsec:motivic} for more details.

\begin{introtheorem}[Motivic homotopy theory, {\cref{thm:motivic:milnor-witt-localization,thm:motivic:milnor-witt-completion,thm:motivic:milnor-localization,thm:motivic:milnor-completion}}]
Let $k$ be a field, $A \in \CAlg(\SH(k)^\veff)$, and $X \in \Spc(k)_*$ a resolvable motivic space (see \cref{def:motivic:resolvable},
e.g.\ any nilpotent $2$-effective motivic space, cf.\ \cref{rmk:motivic:resolvable}).
Let $L$ be given as follows (depending on $\ul\pi_0(A)_*$):
  \begin{center}
  \renewcommand{\arraystretch}{1.2}
  \begin{tabular}{l|c|c}
  $\ul\pi_0(A)$ & Additional assumptions & Localization functor $L$ \\
  \hline
  $S^{-1}\ul\pi_0(\1)_*$ & & $L_S$ \\
  $\ul\pi_0(\1)_*/n$ & & $L_n$ \\
  $S^{-1}(\ul\pi_0(\1)_*/\eta)$ & $\mathrm{cd}_2(k)<\infty$ & $L_S$ \\
  $\ul\pi_0(\1)_*/(n,\eta)$ & $\mathrm{cd}_2(k)<\infty$ & $L_n$
  \end{tabular}
  \end{center}
Then in these cases we can identify the Bousfield--Kan completion of $X$ as
\begin{equation*}
  LX \wequi \tau_{\ge 1} \lim_{\Delta} (\Omega^\infty A \otimes \Sigma^\infty)^{\bullet+1}(X).
\end{equation*}
Moreover, in the two cases where $L=L_S$, also $LX \wequi \hat L_A X$ (for $X$ resolvable).
\end{introtheorem}

\begin{remark}
Contrary to the situations in \cref{thm:main:sheaves,thm:main:equivariant} or when $L=L_S$, in the completion situation, i.e., $L=L_n$, we are unable to identify the homological localization $\hat L_A$.
The issue is that, in contrast with $L_n$, it is not clear that if $X$ is $\hat{L}_A$-local, then also $\tau_{\ge 1} X$ is.
In other words, in contrast with the usual completion, it is not clear that discrete objects are $\hat{L}_A$-local.
See also \cite[Conjecture 5.24]{mattis2024unstablepcompletionmotivichomotopy} for further discussion.
\end{remark}

\subsection{Application: unstable Adams spectral sequences}
The cosimplical object $\CB^\bullet$ obtained by the monadic resolution functor yields an unstable spectral sequence,
which is a generalization of the classical unstable Adams spectral sequence,
internal to the $\infty$-topos $\scr X$.

\begin{proposition*}[{Déglise--Pawar}] \label{lem:spectral-sequence:spectral-sequence-for-homological-setup}
    Let $L \colon \scr X_* \adj \scr D \noloc R$ be an adjunction
    and $X \in \scr X_\ast$.
    Then for $F_n = \fib(\mathrm{Tot}^n \CB^\bullet X \to \mathrm{Tot}^{n-1} \CB^\bullet X)$, there is an unstable spectral sequence
    \begin{equation*}
        E^{r,s}_1 = \ul{\pi}_{s-r} (F_r) \Rightarrow \ul{\pi}_{s-r} (\Tot \CB^\bullet X). 
    \end{equation*}
\end{proposition*}
\begin{proof}
    This follows from \cite[Examples 2.32 and 2.37]{DeglisePawar:SpectralSequences} for the tower $(\Tot^n \CB^\bullet)_n$.
\end{proof}
Since we were able to identify the object $\Tot \CB^\bullet X$ in many cases, this spectral sequence
takes a more digestible form, as recorded in the following theorem.
\begin{introtheorem}
  The following unstable spectral sequences exist:
  \begin{enumerate}
    \item Let $\scr X$ be an $\infty$-topos admitting a locally finite-dimensional cover, $A \in \CAlg(\Sp(\scr X)_{\ge 0})$ and $X \in \scr X_*$ be nilpotent.
    \begin{center}
    \renewcommand{\arraystretch}{1.2}
    \begin{tabular}{l|l}
      $\ul\pi_0(A)$ & Spectral Sequence \\
      \hline
      $S^{-1}\ul\Z \subset \ul\Q$ & $E^{r,s}_1 = \ul{\pi}_{s-r}(F_r) \Rightarrow \ul{\pi}_{s-r} (L_S X)$ \\
      $\ul{\Z}/n$ & $E^{r,s}_1 = \ul{\pi}_{s-r}(F_r) \Rightarrow \ul{\pi}_{s-r} (L_n X)$
    \end{tabular}
    \end{center}
    \item Let $G$ be a finite group, $A \in \CAlg(\SH(\B G)_{\ge 0})$, $X \in \Spc(\B G)_*$ be nilpotent and $(1,G)$-connective.
      \begin{center}
      \renewcommand{\arraystretch}{1.2}
      \begin{tabular}{l|l|l}
      $\ul\pi_0(A)$ & Additional assumptions & Spectral Sequence \\
      \hline
      $S^{-1}\ul\pi_0(\1)$ & & $E^{r,s}_1 = \ul{\pi}_{s-r}(F_r) \Rightarrow \ul{\pi}_{s-r} (L_S X)$ \\
      $\ul\pi_0(\1)/n$ & & $E^{r,s}_1 = \ul{\pi}_{s-r}(F_r) \Rightarrow \ul{\pi}_{s-r} (L_n X)$ \\
      $S^{-1}\ul\Z$ & $|G| \in S$ & $E^{r,s}_1 = \ul{\pi}_{s-r}(F_r) \Rightarrow \ul{\pi}_{s-r} (i_\ast L_S i^\ast X)$ \\
      $\ul{\Z}/p$ & $G$ a finite $p$-group & $E^{r,s}_1 = \ul{\pi}_{s-r}(F_r) \Rightarrow \ul{\pi}_{s-r} (L_p X)$
      \end{tabular}
      \end{center}
    \item Let $k$ be a field, $A \in \CAlg(\SH(k)^\veff)$ and $X \in \Spc(k)_*$ be resolvable.
    \begin{center}
    \renewcommand{\arraystretch}{1.2}
    \begin{tabular}{l|l|l}
    $\ul\pi_0(A)$ & Additional assumptions & Spectral Sequence \\
    \hline
    $S^{-1}\ul\pi_0(\1)_*$ & & $E^{r,s}_1 = \ul{\pi}_{s-r}(F_r) \Rightarrow \ul{\pi}_{s-r} (L_S X)$ \\
    $\ul\pi_0(\1)_*/n$ & & $E^{r,s}_1 = \ul{\pi}_{s-r}(F_r) \Rightarrow \ul{\pi}_{s-r} (L_n X)$ \\
    $S^{-1}(\ul\pi_0(\1)_*/\eta)$ & $\mathrm{cd}_2(k)<\infty$ & $E^{r,s}_1 = \ul{\pi}_{s-r}(F_r) \Rightarrow \ul{\pi}_{s-r} (L_S X)$ \\
    $\ul\pi_0(\1)_*/(n,\eta)$ & $\mathrm{cd}_2(k)<\infty$ & $E^{r,s}_1 = \ul{\pi}_{s-r}(F_r) \Rightarrow \ul{\pi}_{s-r} (L_n X)$
    \end{tabular}
    \end{center}
  \end{enumerate}
\end{introtheorem}

\subsection{Ingredients of the proof} \label{subsec:proof-ingredients}
We consider a fairly general adjunction \[ L \colon \scr X_* \adj \scr D \noloc R. \]
To make progress, we shall assume that $\scr X$ is an $\infty$-topos.
The case of sheaves is immediately applicable, as is the equivariant case because genuine $G$-spaces form a presheaf $\infty$-topos.
With slightly more effort, using that motivic spaces are constructed as a localization of an $\infty$-topos, these will also fall into this setup.

In order to understand the Bousfield--Kan completion $\lim_\Delta (RL)^{\bullet+1} X$, we need to look at a more refined object.
Namely, we want to consider the tower of partial totalizations as a \emph{pro-object} 
\[ T_A^\bullet(X) = \lim_{\Delta^{\le \bullet}} c(RL)^{\bullet+1}(X) \in \Pro(\scr X). \]
As pioneered in work of Bousfield, Farjoun and others (see e.g. \cite{dror1973pro,Sullivan:Genetics,grossman1975homotopy,EdwardsHastings:Cech,ArtinMazur:EtaleHomotopy,bousfield1987homology}), pro-spaces behave like spaces in many ways, and similarly the $\infty$-category $\Pro(\scr X)$ behaves in many ways like an $\infty$-topos.
It turns out that the category behaves even more like an $\infty$-topos when considering a localization, namely $\Pro(\scr X_{<\infty})$.
Here $\scr X_{<\infty} \subset \scr X$ is the subcategory of truncated objects.

\begin{example}
One way in which $\Pro(\scr X_{<\infty})$ behaves more like an $\infty$-topos than $\Pro(\scr X)$ is that in the former $\infty$-category, geometric realizations are universal (i.e. stable under pullback), whereas in the latter they are (in general) not.
\end{example}

\begin{remark}
One advantage of working with pro-objects over totalizations is that the implicit infinite limit in a totalization does not interact well with connectivity in the case of $\infty$-topoi of positive homotopy dimension.
In fact, in the main example we have in mind (i.e. motivic spaces), the underlying $\infty$-topos has infinite homotopy dimension, and so even sequential inverse limits can drop connectivity arbitrarily.

Another advantage is that one can prove that the pro-object $T_A^\bullet$ is compatible with geometric realizations (and thus principal fibrations) under relatively weak assumptions.
\end{remark}

These properties will be made precise in \cref{sec:proobjects}.
As we will see below, a major point will be the study of group actions in $\Pro(\scr X_{<\infty})$.
The key takeaway is the comparison of $\Pro(\Grp(\scr X_{< \infty}))$ and $\Grp(\Pro(\scr X_{< \infty}))$ which are \emph{not} equivalent, see \cref{ex:counterexample-levelwise-group}.
Nonetheless, the canonical comparison functor $\Pro(\Grp(\scr X_{< \infty})) \to \Grp(\Pro(\scr X_{< \infty}))$ is fully faithful, and if $G \in \Grp(\Pro(\scr X_{< \infty}))$ is pro-connected, 
then, surprisingly, it lies in the essential image of this functor, see \cref{theorem:connected-levelwise}.
For those groups in the essential image (which we will call \emph{levelwise groups}), one can lift the $\infty$-bundle theory of \cite{Nikolaus_2014}
to the $\infty$-category of pro-truncated objects. Namely, we obtain the following theorem:
\begin{introtheorem}[{\cref{prop:groupoids,prop:pro-actions,ex:pro-actions}}]
  Let $G \in \Grp(\Pro(\scr X_{< \infty}))$ be a levelwise group.
  Then $\Omega \B G \wequi G$.
  Moreover, there is a canonical equivalence between maps $f \colon X \to \B G$ 
  and objects $F$ with a $G$-action, given by sending $f$ to $\fib(f)$ 
  with the induced action by $G = \Omega \B G$.
  If $F$ is an object with a $G$-action, we get a fiber sequence 
  \begin{equation*}
    F \to F \sslash G \to \B G.
  \end{equation*}
\end{introtheorem}

Once this machinery is set up, we proceed to investigate the functor $\tau_{<\infty} T^\bullet_A$ (i.e. the reflection of $T^\bullet_A$ into $\Pro(\scr X_{<\infty})$).
Our main result about it is a version of Bousfield and Kan's principal fibration lemma,
whose proof will concern the first half of this article:
\begin{introtheorem} [{\cref{thm:fiber-sequences}}]
  The functor $\tau_{<\infty} T^\bullet_A \colon \scr X_* \to \Pro(\scr X_{*,<\infty})$ preserves 
  those fiber sequences $F \to E \to B$ where $B$ is simply connected and $F$ is connected, and $LF, L\Omega B$ satisfy certain connectivity assumptions (see \cref{sec:setup} for more details).
\end{introtheorem}
This theorem will allow us to induct on the principalized Postnikov tower of any good enough nilpotent sheaves 
in $\scr X_*$, thus essentially reducing the convergence question to the case of an infinite loop sheaf.

\subsection{Organization}
In \cref{sec:proobjects}, we establish properties of pro-truncated objects in $\infty$-topoi alluded to above, in particular the principal fibration theory.
Then in \cref{sec:setup} we introduce notation and common assumptions for the rest of the article, most importantly, the adjunction $\scr X \adj \scr D$ and the axioms \axiomref{C}, \axiomref{M} and \axiomref{S}.
\cref{subsec:locn} investigates compatibility of the localization functor $\hat L$ with fiber sequences.
In a sense our main results are contained in \cref{sec:BK}: here we prove that, under a certain fairly general list of assumptions, the Bousfield--Kan completion functor $\tau_{<\infty} \Tot^\bullet \CB$ preserves principal fibrations.
We then come to \cref{sec:applications}, where we establish our applications to sheaves, equivariant spaces and motivic spaces.
We conclude with some appendices: in \cref{sec:appendix:coherence-data} we recall how to extract (co)simplicial data out of adjunctions, in \cref{sec:stable-localization} we prove a well-known fact about the stable analog of Bousfield--Kan completion, in \cref{sec:unstable-n-comp} we extend the unstable $p$-completion theory of \cite{mattis2024unstablepcompletionmotivichomotopy} to unstable $n$-completion ($n$ not necessarily prime), and in \cref{sec:slice-convergence} we establish a minor extension of Levine's convergence theorem in motivic homotopy theory.

\subsection{Notation and conventions}
We freely use the language of $\infty$-categories as set out in \cite{highertopoi,higheralgebra}. A \emph{space} will be an object 
of the $\infty$-category $\Spc$ of spaces/homotopy types/$\infty$-groupoids/anima, and a \emph{sheaf} will be an object of an $\infty$-topos.
Here is a table of some of the notation that we use.
\begin{center}
\begin{longtable}{l|l|l}
notation & meaning & reference/definition \\
\hline
$\Pro(\scr C)$ & $\infty$-category of pro-objects & \cref{subsec:pro} \\
$c$ & constant pro-object functor & \cref{subsec:pro} \\
$\mat X$ & materialization of a pro-object & \cref{subsec:pro} \\
$\scr X_{<\infty}$ & truncated objects & \cref{subsec:trunc} \\
$\tau_{<\infty}$ & pro-truncation & \cref{subsec:trunc} \\
$\Grpd(\scr C)$ & $\infty$-category of groupoid objects & \cref{subsec:groupoids} \\
$\Grp(\scr C)$ & $\infty$-category of group objects & \cref{subsec:groupoids} \\
$L, R, \scr D^\sld$ & standard setup & \cref{sec:setup} \\
$\hat{L}$ & homological localization functor & \cref{subsec:locn} \\
$\CB^\bullet$ & cobar resolution & \cref{subsec:setup} \\
$T^\bullet$ & monadic resolution tower & \cref{subsec:setup} \\
$C^n$ & total fiber of the codegeneracy cube & \cref{subsec:connectivity} \\
$S^n$ & total fiber of the face cube & \cref{subsec:connectivity} \\
$L_n$ & unstable $n$-completion functor & \cite[\S3]{mattis2024unstablepcompletionmotivichomotopy}, \cref{sec:unstable-n-comp} \\
$L_S$ & unstable $S$-localization functor & \cite[\S3]{mattis2024unstablearithmeticfracturesquares} \\
$\Spc(\B G)$ & genuine $G$-spaces & $\PSh_\Sigma(\Fin_G)$ \\
$\SH^{S^1}(\B G)$ & $G$-equivariant $S^1$-spectra & $\Sp(\Spc(\B G))$ \\
$\SH(\B G)$ & genuine $G$-spectra & \cref{subsec:equivariant} \\
$\ul\pi_i E$ & Mackey functor valued homotopy objects & \cref{subsec:equivariant} \\
$\ul A$ & Burnside ring Mackey functor & e.g.\ \cite[§V.9]{LewisMaySteinbergerMcClure:EquivariantHomotopyTheory} \\
$\Sm_k$ & smooth quasi-compact $k$-schemes & \\
$\Spc(k)$ & motivic spaces & \cref{subsec:motivic}\\
$\SH^{S^1}(k)$ & motivic $S^1$-spectra & $\Sp(\Spc(k))$ \\
$\SH(k)$ & motivic spectra & \cref{subsec:motivic}\\
$\Spc(k)_{\ge (\nu, \nu)}$ & $\nu$-effective motivic spaces & \cite[§3]{asok2024p1stabilizationunstablemotivichomotopy}, \cref{rmk:L-pres-eff}\\
$\ul\pi_i(X)_*$ & stable motivic homotopy sheaves & \cref{rmk:t-structures-SH} \\
$\SH(k)^\heartsuit$ & heart of the homotopy t-structure & \cref{rmk:t-structures-SH} \\
$\SH(k)^{\veff}$ & very effective motivic spectra & \cref{rmk:t-structures-SH} \\
$\ul\pi^\eff_i(X)$ & stable effective motivic homotopy sheaves & \cref{rmk:t-structures-SH} \\
$\SH(k)^{\eff\heartsuit}$ & heart of the effective t-structure & \cref{rmk:t-structures-SH} \\
$f_n, s_n$ & effective cover and slice functors & \cref{rmk:slices} \\
$\ul{K}^{MW}$ & Milnor--Witt K-theory homotopy module & \\
$\ul{K}^{M}$ & Milnor K-theory homotopy module & \\
$\eta$ & motivic Hopf map & \\
\end{longtable}
\end{center}

\subsection{Acknowledgments}
We are greatly indebted to Mike Hopkins for providing invaluable support at several stages of this project.
During its completion, we benefited from conversations with many people, among them Aravind Asok, Robert Burklund, Emmanuel Farjoun, Marc Hoyois, Lorenzo Mantovani and Georg Tamme.

The authors acknowledge support by the Deutsche Forschungsgemeinschaft
(DFG, German Research Foundation) through the Collaborative Research Center 
TRR 326 \emph{Geometry and Arithmetic of Uniformized Structures}, project number 444845124.

\section{Homotopy theory in pro-topoi} \label{sec:proobjects}
Throughout we let $\scr X$ be an $\infty$-topos \cite[Definition 6.1.0.4]{highertopoi}.
In this section we will show that the $\infty$-category $\Pro(\scr X_{<\infty})$ of pro-truncated objects in $\scr X$ behaves in some ways like an $\infty$-topos.
In particular, we will study the relationship between groupoids and (some kind of) epimorphisms in $\Pro(\scr X_{<\infty})$ (see \cref{prop:groupoids}), and between group actions and maps to classifying spaces (see \cref{prop:pro-actions}).
This culminates in our replacement for principal fibration theory in $\Pro(\scr X_{<\infty})$: for a pro-connected group $G \in \Grp(\Pro(\scr X_{<\infty}))$ there is an equivalence between $G$-torsors and maps to $\B G$, given as usual by taking $G$-orbits respectively taking fibers (see \cref{ex:pro-actions} and \cref{theorem:connected-levelwise}).

\subsection{Pro-objects} \label{subsec:pro}
Given a possibly large $\infty$-category $\scr C$, we denote by $\Pro(\scr C)$ the $\infty$-category of pro-objects in $\scr C$ \cite[\S 0.11.6]{barwick2018exodromy}.
One possible definition is $\Pro(\scr C) = \mathrm{Ind}(\scr C^\op)^\op$.
This comes with the \emph{constant object} functor \[c \colon \scr C \to \Pro(\scr C), \] characterizing $\Pro(\scr C)$ via a universal property: $\Pro(\scr C)$ has cofiltered limits and any functor out of $\scr C$ into an $\infty$-category with cofiltered limits extends uniquely along $c$ to a functor out of $\Pro(\scr C)$ preserving cofiltered limits.
We note the following further facts:
\begin{itemize}
\item $c$ preserves finite limits (see the dual of \cite[Proposition 5.3.5.14]{highertopoi}).
\item The association $\scr C \mapsto \Pro(\scr C)$ canonically assembles into a functor \[ \Pro(\ph) \colon \Cat_\infty \to \Cat_\infty. \]
\item Any object in $\Pro(\scr C)$ can be presented as a cofiltered limit of constant objects.
\item $c$ is fully faithful.
\item If $\scr C$ has cofiltered limits, then $c$ has a right adjoint $\mat \colon \Pro(\scr C) \to \scr C$, given by $\mat \lim_i c X_i \wequi \lim_i X_i$.
\end{itemize}
From this we can deduce that there is an equivalence $\Fun(I, \Pro(\scr C)) \wequi \Pro(\Fun(I, \scr C))$ for any finite poset $I$ \cite[Proposition 5.3.5.15]{highertopoi}.
It follows, for example, that $\Pro(\scr C)_* \wequi \Pro(\scr C_*)$.
However, this fails for infinite diagrams: in general, the canonical functor \[ \Pro(\Fun(I, \scr C)) \to \Fun(I, \Pro(\scr C)) \] is neither fully faithful nor essentially surjective.
We will revisit this theme several times in what follows.

\subsection{Pro-truncated objects} \label{subsec:trunc}
We write $\scr X_{<\infty} \subset \scr X$ for the subcategory of truncated objects, and $\Pro(\scr X_{<\infty}) \subset \Pro(\scr X)$ for the subcategory of pro-truncated objects.
The latter inclusion has a left adjoint \[ \tau_{<\infty} = \lim_n \Pro(\tau_{\le n})\colon \Pro(\scr X) \to \Pro(\scr X_{<\infty}) \] called protruncation; see \cite[\S 4.1.2]{barwick2018exodromy}.
It is pro-induced by the functor $\scr X \to \Pro(\scr X_{<\infty})$ given by $X \mapsto \lim_n c \tau_{\le n} X$.
It preserves all limits \cite[Proposition 3.9]{haine2022fundamentalfibersequenceetale} (the reference 
only proves it for the $\infty$-topos of spaces, but the proof works for an arbitrary $\infty$-topos).

Recall the notion of almost finite colimits from \cite[Definition 1.9]{haine2024profinitecompletionsproducts}
(and in particular the notion of $n$-colimit-cofinal functors, cf.\ \cite[Definition 1.7]{haine2024profinitecompletionsproducts}).
Let us give an example of almost finite colimits:
\begin{lemma}
  Geometric realizations of both simplicial and semisimplicial objects are almost finite.
\end{lemma}
\begin{proof}
  We have to show that both $\Delta^\op$ and $\Delta^{\inj,\op}$ have $n$-colimit-cofinal approximations by a finite category.
  Since the inclusion $\Delta^{\inj,\op} \hookrightarrow \Delta^\op$ is colimit-cofinal (see \cite[Lemma 6.5.3.7]{highertopoi}),
  it suffices to prove this for $\Delta^{\inj,\op}$.
  For any $n$ consider the functor 
  $\Delta^{\inj,\op}_{\le n} \to \Delta^{\inj,\op}$.
  It is clear that $\Delta^{\inj,\op}_{\le n}$ is a finite category,
  and the functor is $n$-colimit-cofinal by \cite[Example 6.14]{du2023reshapinglimitdiagramscofinality}.
\end{proof}
The first part of the next proposition was proven by Haine for the $\infty$-category of spaces, cf.\ \cite[Proposition 1.17]{haine2024profinitecompletionsproducts},
and is the main motivation for considering the notion of almost finite colimits.

\begin{proposition}[Haine] \label{prop:pro-truncated-geometric-realization}
Almost finite colimits in $\Pro(\scr X_{<\infty})$ are universal and commute with cofiltered limits.
In particular geometric realizations (of either simplicial or semisimplicial objects) in $\Pro(\scr X_{<\infty})$ are universal and commute with cofiltered limits.
\end{proposition}

In order to prove this, we will use two further properties of $\Pro(\tau_{\le n})$.

\begin{lemma} \label{lem:pro-objects:pro-truncations-conservativity}
    Let $f \colon X \to Y$ be a morphism in $\Pro(\scr X)$.
    Then $\tau_{<\infty} f$ is an equivalence if and only if $\Pro(\tau_{\le n})(f)$ is an equivalence for all $n$.
\end{lemma}
\begin{proof}
This is immediate from the facts that $\tau_{<\infty} \wequi \lim_n \Pro(\tau_{\le n})$ and $\Pro(\tau_{\le n}) \wequi \Pro(\tau_{\le n}) \circ \tau_{<\infty}$.
(The latter holds because $\Pro(\scr X_{\le n}) \subset \Pro(\scr X_{<\infty}) \subset \Pro(\scr X)$ and the truncations are left adjoints to the appropriate inclusions.)
\end{proof}

The following lemma is an extension of \cite[Proposition 4.13]{DevalapurkarHaine} to pro-objects in $\infty$-topoi.
\begin{lemma} \label{lem:almost-finite:truncation-pullback}
    Let $\scr X$ be an $\infty$-topos.
    The pullback functor
    \[ \Fun(\Lambda_2^2, \Pro(\scr X)) \to \Pro(\scr X) \] sends $\Pro(\tau_{\le n+1})$-equivalences to $\Pro(\tau_{\le n})$-equivalences.
\end{lemma}
\begin{proof}
    By 2-out-of-3, it suffices to show that the $n$-truncated pullback functor inverts the map of spans $(U \to V \leftarrow W) \to (\Pro(\tau_{\le n+1})(U) \to \Pro(\tau_{\le n+1})(V) \leftarrow \Pro(\tau_{\le n+1})(W))$.
    Since all functors in sight preserve cofiltered limits, this reduces to the case where $U,V,W$ are constant, which is a special case of \cite[Proposition 4.13]{DevalapurkarHaine}.
\end{proof}

Also note the following general fact about categories of pro-objects:
\begin{proposition} \label{lem:almost-finite:cofiltered-finite-commutes}
  Let $\scr C$ be any $\infty$-category with finite colimits.
  Then finite colimits commute with cofiltered limits in $\Pro(\scr C)$.
\end{proposition}
\begin{proof}
  Mapping to objects $cX$ with $X \in \scr C$ is jointly conservative,
  and turns cofiltered limits to filtered colimits (and finite colimits to finite limits).
  Hence, the claim follows from the corresponding dual result about spaces, 
  cf.\ \cite[Proposition 5.3.3.3]{highertopoi}.
\end{proof}

\begin{proof}[Proof of \cref{prop:pro-truncated-geometric-realization}]
  We first show that almost finite colimits commute with cofiltered limits in $\Pro(\scr X_{<\infty})$, which is slightly easier.
  Using \cref{lem:pro-objects:pro-truncations-conservativity} and the fact that $\Pro(\tau_{\le n})$ preserves colimits (being a left adjoint) and cofiltered limits, it suffices to prove the same statement about $\Pro(\scr X_{\le n})$.
  But almost finite colimits in $\Pro(\scr X_{\le n})$ are just finite colimits, and finite colimits always commute with cofiltered limits in categories of pro-objects by \cref{lem:almost-finite:cofiltered-finite-commutes}. 

  Now we show that almost finite colimits are universal in $\Pro(\scr X_{<\infty})$.
  Using \cref{lem:pro-objects:pro-truncations-conservativity,lem:almost-finite:truncation-pullback} we reduce to proving the same statement for $\Pro(\scr X_{\le n+1})$.
  But almost finite colimits in this category are finite colimits, so the result holds by \cite[Lemma 1.1]{haine2024profinitecompletionsproducts}.
\end{proof}

\subsection{Epimorphisms and connected pro-objects} \label{sec:epi}

Write $\Fun(\Delta, \scr X)^\epi \subset \Fun(\Delta, \scr X)$ for the full subcategory spanned by those morphisms 
that are effective epimorphisms \cite[the discussion after Corollary 6.2.3.5]{highertopoi}.

\begin{definition}
  We say that $f \in \Fun(\Delta^1, \Pro(\scr X))$ is a \emph{levelwise effective epimorphism}
  if under the canonical equivalence $\Pro(\Fun(\Delta^1, \scr X)) \wequi \Fun(\Delta^1, \Pro(\scr X))$
  the morphism $f$ corresponds to an object in the full subcategory $\Pro(\Fun(\Delta^1, \scr X)^\epi)$.

  We write $\Fun(\Delta^1, \Pro(\scr X))^\epi \subset \Fun(\Delta^1, \Pro(\scr X))$
  for the full subcategory of levelwise effective epimorphisms,
  so that there is a canonical equivalence $\Fun(\Delta^1, \Pro(\scr X))^\epi \wequi \Pro(\Fun(\Delta^1, \scr X)^\epi)$.

  We say that $f \in \Fun(\Delta^1, \Pro(\scr X_{<\infty}))$ is a \emph{levelwise effective epimorphism}
  if it is one considered as an object of $\Fun(\Delta^1, \Pro(\scr X))$, i.e., if it can be presented as a cofiltered limit of effective epimorphisms between truncated objects (use that $\tau_{\le n}$ preserves effective epimorphisms and hence $\tau_{<\infty}$ preserves levelwise effective epimorphisms).
\end{definition}

\begin{remark}\label{rmk:effective-epi:levelwise-surjective}
  If $f \colon X \to Y$ is a morphism in $\scr X$,
  then $f$ is an effective epimorphism if and only if $\pi_0(f)$ is surjective
  (as a map of sheaves, i.e., an epimorphism in the $1$-category $\scr X_{\le 0}$),
  essentially by \cite[Corollary 6.2.3.5]{highertopoi}.
  In particular, we see that $g \colon U \to V$ in $\Pro(\scr X)$ 
  is a levelwise effective epimorphism if and only if $\Pro(\pi_0)(g)$
  is a levelwise surjection.
\end{remark}

\begin{definition}
  We say that an object $X \in \Pro(\scr X)_*$ is \emph{pro-connected} 
  if the map $* \to X$ is a levelwise effective epimorphism.
  We write $\Pro(\scr X)_{*,\ge 1} \subset \Pro(\scr X)_*$ for the full subcategory of pro-connected objects.
\end{definition}

\begin{lemma}
  Let $X \in \Pro(\scr X)_*$ be an object.
  $X$ is pro-connected if and only if $X$ lives in the full subcategory 
  $\Pro(\scr X_{*,\ge 1}) \subseteq \Pro(\scr X_*) \wequi \Pro(\scr X)_*$.
\end{lemma}
\begin{proof}
  Suppose first that $X \in \Pro(\scr X_{*,\ge 1})$.
  Then there exists a cofiltered limit presentation $X = \lim_i c X_i$,
  with $X_i$ connected.
  In particular, $* \to X_i$ is an effective epimorphism for all $i$,
  which shows that $* \to X$ is a levelwise effective epimorphism.

  Suppose on the other hand that $* \to X$ is a levelwise effective epimorphism.
  Thus, this map can be written as a cofiltered limit of effective epimorphisms
  $A_i \to X_i$ in $\scr X_*$.
  Set $X_i' = X_i \amalg_{A_i} *$.
  Then since $c$ preserves colimits and finite colimits commute with cofiltered limits in pro-objects by \cref{lem:almost-finite:cofiltered-finite-commutes}, we find \[ \lim_i cX_i' \wequi \lim_i cX_i \amalg_{\lim_i cA_i} * \wequi \lim_i cX_i \wequi X. \]
  (Indeed $\lim_i cA_i = *$ by assumption.)
  By construction $X_i'$ is connected
  (as effective epimorphisms are stable under cobasechange, cf.\ \cite[Corollary 6.5.1.17 with $n = 0$]{highertopoi}), concluding the proof.
\end{proof}

\subsection{Groupoids} \label{subsec:groupoids}
Let $\scr C$ be either $\scr X$, $\Pro(\scr X)$, or $\Pro(\scr X_{<\infty})$,
so that we can speak about the $\infty$-category $\Fun(\Delta^1, \scr C)^\epi$,
cf.\ \cref{sec:epi}.
By a groupoid in $\scr C$ we mean an object of $\Fun(\Delta^\op, \scr C)$ satisfying the Segal condition \cite[Proposition 6.1.2.6 (4'')]{highertopoi}.
Left Kan extending to $\Delta^\op_+$ and restricting to $\Delta^\op_{+, \le 0} \wequi \Delta^1$ yields a functor $|\ph|$, participating in an adjunction \[ |\ph|\colon \Fun(\Delta^\op, \scr C) \adj \Fun(\Delta^1, \scr C)\noloc N. \]
By construction, the right adjoint is given by right Kan extending to $\Delta^\op_+$ and then restricting to $\Delta^\op$.
It is clear that the image of $N$ is contained in $\Grpd(\scr C)$.
If $\scr C = \scr X$, then the image of $|\ph|$ consists of effective epimorphisms,
cf.\ \cite[Lemma 6.2.3.13]{highertopoi} (using that each of the morphisms $X_n \to |X_\bullet|$ factors over $X_0$).

\begin{lemma}
If $\scr C = \Pro(\scr X)$ or $\scr C = \Pro(\scr X_{<\infty})$, then the image of $|\ph|$ consists of levelwise effective epimorphisms.
\end{lemma}
\begin{proof}
By \cref{rmk:effective-epi:levelwise-surjective}, we have to see that for any $X_\bullet \in \Grpd(\scr C)$
the map $\Pro(\pi_0)(X_0) \to \Pro(\pi_0)(|X_\bullet|)$ 
is a levelwise surjection. 
Now, $\Pro(\pi_0)(|X_\bullet|) \wequi \Pro(\pi_0)(\colim_{\Delta^{\inj,\op}_{\le 1}} X_\bullet)$
by \cite[Lemma 1.2.4.17]{higheralgebra} and \cite[Proposition A.1]{hesselholt2023diracgeometryicommutative} 
(since $\Pro(\pi_0)$ commutes with the colimit, and $\Pro(\scr X_{\le 0})$ is a $1$-category), 
and hence the question depends only on the finite diagram $X_\bullet|_{\Delta^{\inj,\op}_{\le 1}}$.
The result now follows using the equivalence $\Pro(\Fun(\Delta^{\inj,\op}_{\le 1}, \scr X)) \wequi \Fun(\Delta^{\inj,\op}_{\le 1}, \Pro(\scr X))$, and by using \cite[Lemma 6.2.3.13]{highertopoi} on each level. 
\end{proof}

We write $\Grp(\scr C) \subset \Grpd(\scr C)$ for the full subcategory 
of \emph{group objects}, i.e.\ those $U_\bullet$ such that $U_0 = *$.

\begin{proposition} \label{prop:groupoids}
\begin{enumerate}
\item We have an equivalence $\Fun(\Delta^1, \Pro(\scr X_{<\infty}))^\epi \wequi \Pro(\Grpd(\scr X_{<\infty}))$,
  and thus $\Pro(\scr X_{<\infty})_{*,\ge 1} \wequi \Pro(\Grp(\scr X_{<\infty}))$.
\item The composite \[ \Pro(\Grpd(\scr X_{<\infty})) \wequi \Fun(\Delta^1, \Pro(\scr X_{<\infty}))^\epi \xrightarrow{N} \Grpd(\Pro(\scr X_{<\infty})) \] is the canonical functor.
\item The functor $N\colon \Fun(\Delta^1, \Pro(\scr X_{<\infty}))^\epi \to \Grpd(\Pro(\scr X_{<\infty}))$ is fully faithful.
\end{enumerate}
(In fact the first two statements hold for all of $\Pro(\scr X)$.)
\end{proposition}
\begin{proof}
(1) The fact that $\Fun(\Delta^1, \Pro(\scr X))^\epi \wequi \Pro(\Grpd(\scr X))$ follows immediately 
since by definition $\Fun(\Delta^1, \Pro(\scr X))^\epi \wequi \Pro(\Fun(\Delta^1, \scr X)^\epi)$ and $\Fun(\Delta^1, \scr X)^\epi \wequi \Grpd(\scr X)$; for the latter see \cite[Theorem 6.1.0.6]{highertopoi}.

To get the result for pro-truncated objects,
it will suffice to show that the equivalence $\Fun(\Delta^1, \scr X)^\epi \wequi \Grpd(\scr X)$ 
restricts to an equivalence $\Fun(\Delta^1, \scr X_{< \infty})^\epi \wequi \Grpd(\scr X_{<\infty})$.
Let us first assume that $f \colon X \to Y$ is an effective epimorphism, such that both $X$ and $Y$ are in $\scr X_{<\infty}$, say in $\scr{X}_{\le M}$ 
for some $M \gg 0$.
Now, $N(f)$ is (the restriction of) the right Kan extension of $f$ to $\Delta_+$. As right Kan extensions  
are computed using limits, and $\scr X_{\le M} \subset \scr X$ is stable under limits, we see that $N(f) \in \Grpd(\scr X_{\le M}) \subseteq \Grpd(\scr X_{< \infty})$.
For the other direction, assume that $X_\bullet \in \Grpd(\scr X_{<\infty})$.
Pulling back $X_0 \to |X_\bullet|$ along itself yields $X_1 \to X_0$ (since $N(|X_\bullet|) \wequi X_\bullet$).
This map being truncated \cite[Lemma 5.5.6.14]{highertopoi}, we deduce that $X_0 \to |X_\bullet|$ is also truncated \cite[Proposition 6.2.3.17]{highertopoi}.
Thus, $|X_\bullet|$ is truncated by \cref{lemm:eff-epi-ntrunc} below.

That this equivalence restricts further to an equivalence $\Pro(\scr X_{<\infty})_{*,\ge 1} \wequi \Pro(\Grp(\scr X_{<\infty}))$ 
is clear.

(2) Holds by construction.

(3) We want to show that $|N(\ph)| \wequi \id$.
Since geometric realizations commute with cofiltered limits (\cref{prop:pro-truncated-geometric-realization}), so does the functor $|\ph|$ (and so does $N$, being a right adjoint).
Any object in $\Fun(\Delta^1, \Pro(\scr X_{<\infty}))^\epi$ being a cofiltered limit of effective epimorphisms in $\scr X_{<\infty}$, we reduce to such (constant) objects.
This case follows from the same result for $\scr X$, for which see \cite[Theorem 6.1.0.6]{highertopoi} (where $N$ is even an equivalence).
\end{proof}

We used the following fact, for which we could not locate a reference.
\begin{lemma} \label{lemm:eff-epi-ntrunc}
Let $\scr X$ be an $\infty$-topos and $f\colon X \to Y \in \scr X$ an $n$-truncated effective epimorphism.
Then $X$ is $(n+1)$-truncated if and only if $Y$ is $(n+1)$-truncated.
\end{lemma}
\begin{proof}
If $Y$ is $(n+1)$-truncated then so is $X$ by \cite[Lemma 5.5.6.14]{highertopoi}.
We establish the converse.
Let $L\colon \PSh(\scr C) \adj \scr X\noloc R$ be a presentation of $\scr X$ as a left exact localization of a presheaf $\infty$-topos \cite[Definition 6.1.0.4]{highertopoi}.
Recall that both $R$ and $L$ preserve $i$-truncated morphisms and objects \cite[Proposition 5.5.6.16]{highertopoi}.
Write $RX \xrightarrow{f_0} Y_0 \to RY$ for the epi-mono factorization \cite[Example 5.2.8.16 for $n = -1$]{higheralgebra} of $Rf$.
Since $Rf$ is $n$-truncated and $Y_0 \to RY$ is $(-1)$-truncated (whence $n$-truncated), $RX \to Y_0$ is also $n$-truncated (apply \cite[Proposition 5.5.6.16]{highertopoi} in $\PSh(\scr C)_{/RY}$).
As $X \to Y$ is an effective epimorphism and $L$ preserves epi-mono factorizations, we see that $Lf_0 \wequi f$.
Using once more that $L$ preserves $i$-truncated maps, we see that we may replace $\scr X$ by $\PSh(\scr C)$ and thus reduce to $\scr X = \Spc$.
In this case the claim follows by examination of the long exact sequence of homotopy groups.
\end{proof}

Thanks to the above result, the following makes good sense.
\begin{definition}
We call objects in the essential image of the fully faithful functor $\Pro(\Grpd(\scr X_{<\infty})) \to \Grpd(\Pro(\scr X_{<\infty}))$ \emph{levelwise groupoids}.
\emph{Levelwise groups} are defined similarly.
\end{definition}

\begin{example} \label{ex:counterexample-levelwise-group}
For any group $G \in \Grp(\Pro(\scr X_{<\infty}))$, \cref{prop:groupoids} supplies us with a map $G \to \Omega \B G (:= N|G|)$, which is in fact the initial map to a levelwise group.
However, not all groups are levelwise groups, as the following example shows.
Consider the pro-set $G \coloneqq \lim_k cG_k$,
where $G_k \coloneqq [-\frac{1}{2^k}, \frac{1}{2^k}] \subset \mathbb{R}$ is a closed interval (considered as a set),
and the transition morphisms are given by inclusion.
We define a group structure on $G$ as the inverse limit of the maps $G_{k+1} \times G_{k+1} \to G_k$ sending $(x,y)$ to $x+y$ (where the addition is the addition from $\mathbb{R}$).
We claim that $G$ is not a levelwise group. For this, consider the initial map to a levelwise group 
$G \to \Omega \B G$; it suffices to show that this map is not an equivalence.
This can be checked after materializing.
The left-hand side yields $\mat (G) = \lim_k G_k = \cap_k G_k = \{0\}$.
It thus suffices to show that $\mat (\Omega \B G) = \Omega \mat(\B G) \neq *$.
In fact, $\pi_1(\mat(\B G)) \neq 0$, as we shall show now.
Denote by $B_{k,\bullet}$ the following inverse system of \emph{semi}-simplicial diagrams:
\begin{center}
  \begin{tikzcd}[column sep=1.5em]
    B_{k,\bullet} = (* &\ar[l, shift left] \ar[l, shift right] G_k &  \ar[l, shift left=5pt] \ar[l, shift right=5pt] \ar[l] G_{k+1} \times G_{k+1} & \ar[l, shift left] \ar[l, shift left=7pt] \ar[l, shift right] \ar[l, shift right=7pt] G_{k+2} \times G_{k+2} \times G_{k+2} \dots).
  \end{tikzcd}
\end{center}
Then $|\lim_k cB_{k,\bullet}| \wequi \B G$.
Using that in pro-truncated spaces cofiltered limits commute 
with geometric realizations (even of semisimplicial objects), and $c$ is a fully faithful left adjoint (with right adjoint the materialization) we deduce that $\mat (\B G) \wequi \lim_k \B G_k$, where $\B G_k := |B_{k,\bullet}|$.
Thus, $\pi_1$ of the materialization is nonzero if we can find a compatible family of nonzero elements in 
the $\pi_1(\B G_k)$. 
Indeed, such a family then defines a nonzero element of $\lim_k \pi_1(\B G_k)$,
and there is a surjection $\pi_1(\lim_k \B G_k) \to \lim_k \pi_1(\B G_k)$. 
For every $x \in G_k$ we get an element $[x] \in \pi_1 (\B G_k)$.
There is a canonical map $\B G_k \to \B\R$ (where we view $\R$ as a discrete group) which maps the loop 
$[x]$ to the corresponding loop $[x] \in \pi_1(\B\R)$, and thus those elements are all distinct (and in particular nonzero if $x$ is nonzero).
Now set $x_{n} = [\frac{1}{2^n}] \cdot \dots \cdot [\frac{1}{2^n}] \in \pi_1(\B G_{n})$,
the $2^n$-fold product (in $\pi_1$) of the loop represented by $\frac{1}{2^n}$.
Since in $\B G_{n-1}$ we impose the relation $[\frac{1}{2^n}] \cdot [\frac{1}{2^n}] = [\frac{1}{2^{n-1}}]$,
it follows that $x_{n}$ maps to $x_{n-1}$.
By the argument above, this compatible family is nonzero (as it gets mapped to $[1] \in \B\R$), thus proving the claim.
\end{example}

We have the following useful criterion for detecting levelwise groups.
\begin{theorem} \label{theorem:connected-levelwise}
Let $G \in \Grp(\Pro(\scr X_{< \infty}))$ be pro-connected, i.e., the unit $* \to G$ is a levelwise effective epimorphism.
Then $G$ is a levelwise group.
\end{theorem}
For the proof, we need the following identification of the $\scr E_1$-structures 
on the two different bar constructions one can perform on an $\scr E_2$-group,
for which we were unable to locate a proof in the literature.
For any $\infty$-category $\scr C$ with finite products and geometric realizations,
and any $G \in \Grp(\scr C)$, write $\B_\bullet G \in \Fun(\Delta^\op, \scr C)$ for the corresponding simplicial object,
and $\B_{\scr C}G \coloneqq \B G \coloneqq |\B_\bullet G| \in \scr C$ for the geometric realization.

\begin{proposition} \label{lem:two-bar-constructions}
  Let $\scr C$ be an $\infty$-category with finite products and all colimits,
  such that geometric realizations commute with finite products (e.g.\ $\scr X$ or $\Pro(\scr X_{< \infty})$).
  There is an equivalence of functors
  $\Grp(\B_{\scr C}) \wequi \B_{\Grp(\scr C)} \colon \Grp(\Grp(\scr C)) \to \Grp(\scr C)$.
\end{proposition}
\begin{proof}
Let $\scr C_0 \subset \scr C$ be a small full subcategory closed under finite products.
Consider the commutative diagram
\begin{equation*}
\begin{tikzcd}
\Grp(\Grp(\scr C_0)) \ar[r,"y"] & \Grp(\Grp(\PSh(\scr C_0))) \ar[d, shift left]\ar[d, shift right] \ar[r,"\colim"] & \Grp(\Grp(\scr C))\ar[d, shift left]\ar[d, shift right]  \\
& \Grp(\PSh(\scr C_0)) \ar[r,"\colim"] & \Grp(\scr C),
\end{tikzcd}
\end{equation*}
where $y \colon \scr C_0 \to \PSh(\scr C_0)$ and $\colim \colon \PSh(\scr C_0) \to \scr C$ preserve finite products (whence groups) and the vertical arrows are $\Grp(\B)$ respectively $\B_\Grp$.
An equivalence between the left hand vertical arrows thus yields an equivalence between the two composites $\Grp(\Grp(\scr C_0)) \to \Grp(\scr C)$.
Writing $\scr C$ as the filtered colimit of small full subcategories $\scr C_0$ closed under finite products, we can reduce from arbitrary $\scr C$ to $\scr C = \PSh(\scr C_0)$.
In fact, since $\Grp(\scr C) \wequi \Grp(\scr C_*)$, we may reduce further to $\scr C = \PSh(\scr C_0)_*$.
I.e., from now on we can (and will) assume that $\scr C = \scr X_*$ for an $\infty$-topos $\scr X$.

  Note that $\Grp(\Grp(\scr C)) \subset \Fun(\Delta^\op \times \Delta^\op, \scr C)$ and under this inclusion, the different directions in which we can take geometric realizations on a bisimplicial object
    correspond to the two functors $\Grp(\B_{\scr C})$ and $\B_{\Grp(\scr C)}$.
  In particular, since colimits commute, we see that \[ \B \circ \Grp(\B) \wequi \B \circ \B_{\Grp}\colon \Grp(\Grp(\scr C)) \to \scr C. \]
  The result follows since in an $\infty$-topos $\scr X$, the functor $\B\colon \Grp(\scr X_*) \wequi \Grp(\scr X) \to \scr X_*$ is fully faithful
  \cite[Lemma 7.2.2.11]{highertopoi}.
\end{proof}
\begin{proof}[Proof of \cref{theorem:connected-levelwise}]
  Consider the commutative diagram 
  \begin{center}
    \begin{tikzcd} 
      &\Grp(\Pro(\scr X_{<\infty})) \ar[dl, "\Grp(\Pro(\Omega))"'] \ar[d, "\Grp(\Omega)"] \\
      \Grp(\Pro(\Grp(\scr X_{<\infty}))) \ar[r] \ar[d, "\B"] & \Grp(\Grp(\Pro(\scr X_{<\infty}))) \ar[d, "\B"] \\
      \Pro(\Grp(\scr X_{<\infty})) \ar[r] &\Grp(\Pro(\scr X_{<\infty}))\rlap{,}
    \end{tikzcd}
  \end{center}
  where the horizontal functors forget that a group is levelwise.
  The square commutes by naturality of the classifying sheaf functor $\B$.
  The triangle commutes already before applying $\Grp(-)$:
  Indeed, all the involved functors preserve cofiltered limits,
  hence we may check this on constant objects, where it is clear.
  Let $G \in \Grp(\Pro(\scr X_{<\infty}))$ be pro-connected.
  Our goal is to show that $G$ 
  is in the essential image of the bottom horizontal functor.
  Chasing the commutative diagram, it suffices to show that the right-hand composite sends $G$ to itself.
  By \cref{lem:two-bar-constructions}, we see that 
  $\B (\Grp(\Omega) (G)) \wequi \Grp(\B \Omega)(G)$.
  The result follows since $\B\Omega X \wequi X$ for any pro-connected $X \in \Pro(\scr X_{<\infty})_*$;
  indeed this was shown in \cref{prop:groupoids}(1).
\end{proof}

\subsection{Group actions} \label{subsec:progroups}
Let $\scr C$ be an $\infty$-category with pullbacks and $I$-indexed colimits.
Given $B \in \Fun(I, \scr C)$, the functor $\colim_I\colon \Fun(I, \scr C)_{/B} \to \scr C_{/\colim_I B}$ admits a right adjoint $F\colon T \mapsto B \times_{d(\colim_I B)} d(T)$.
Here $d\colon \scr C \to \Fun(I, \scr C)$ is the diagonal functor.
By construction, $F$ lands in the subcategory $\Fun(I, \scr C)_{/B}^\cart \subset \Fun(I, \scr C)_{/B}$ of \emph{cartesian diagrams},
the full subcategory spanned by those natural transformations $\eta \colon B' \to B$ 
that are \emph{cartesian} in the sense of \cite[Definition 6.1.3.1]{highertopoi}.

Now we specialize to $\scr C = \Pro(\scr X_{<\infty})$ and $I = \Delta^\op$.
Let $G \in \Grp(\Pro(\scr X_{<\infty}))$.
Recall that then $\Fun(\Delta^\op, \Pro(\scr X_{<\infty}))_{/\B_\bullet G}^\cart$ can be seen as encoding objects with a $G$-action \cite[Definition 3.1]{Nikolaus_2014}.
We obtain an adjunction \[ |\ph|\colon \Fun(\Delta^\op, \Pro(\scr X_{<\infty}))_{/\B_\bullet G}^\cart \adj \Pro(\scr X_{<\infty})_{/\B G}\noloc F. \]
\vspace{-\abovedisplayskip}\vspace{-\topsep} 
\begin{proposition} \label{prop:pro-actions}
If $G$ is a levelwise group, then the above adjunction is an adjoint equivalence.
\end{proposition}
\begin{proof}
Let $\alpha = (X \to \B G) \in \Pro(\scr X_{<\infty})_{/\B G}$.
Then $\alpha = \lim_i c \alpha_i$ in $\Fun(\Delta^1, \Pro(\scr X_{<\infty}))$, where $\alpha_i\colon X_i \to B_i \in \scr X_{<\infty}$ with $B_i$ connected (where $\lim_i B_i \wequi \B G$).
Since geometric realizations commute with cofiltered limits in $\Pro(\scr X_{<\infty})$ 
(\cref{prop:pro-truncated-geometric-realization}) we find $|F\alpha| \wequi \tau_{<\infty} \lim_i c|F_i \alpha_i|$, 
where \[ |\ph|\colon \Fun(\Delta^\op, \scr X)^\cart_{/\B_\bullet \Omega B_i} \adj \scr X_{/B_i}\noloc F_i \] is the canonical adjoint equivalence 
\cite[Theorem 6.1.3.9 (3)]{highertopoi}.
It follows that the counit $|F\alpha| \to \alpha$ is an equivalence.

We now prove the same for the unit.
We first treat a special case.
By \cref{lemm:free-action} below, the forgetful functor $U\colon \Fun(\Delta^\op, \Pro(\scr X_{<\infty}))_{/\B_\bullet G}^\cart \to \Pro(\scr X_{<\infty})$ has a left adjoint, sending $X \in \Pro(\scr X_{<\infty})$ to $G \times X$.
Write $s\colon * \to \B G$ for the canonical map.
Comparing universal properties we see that $|G \times X| \wequi s_\sharp X$, and so the object underlying $F|G \times X|$ is $* \times_{\B G} X \wequi F\B G \times X$, which by \cref{prop:groupoids} is just $G \times X$.
It follows that the unit $G \times X \to F|G \times X| \wequi G \times X$ is homotopic to the identity map, and so in particular an equivalence.
To treat the general case, first observe that $U$ is conservative and preserves geometric realizations (since geometric realizations commute with finite products in $\Pro(\scr X_{<\infty})$),
hence is monadic, so any object in $\Fun(\Delta^\op, \Pro(\scr X_{<\infty}))_{/\B_\bullet G}^\cart$ can be written as a geometric realization of objects of the form $G \times X$, cf.\ the proof of \cite[Proposition 4.7.3.14]{higheralgebra}.
Since both $|\ph|$ and $F$ preserve geometric realizations (the former being a left adjoint, and the latter by \cref{prop:pro-truncated-geometric-realization}), we conclude.
\end{proof}

If $G$ is acting on an object $F \in \Pro(\scr X_{<\infty})$, by abuse of notation we also write $F$ for the associated 
cartesian diagram $F_\bullet \to \B_\bullet G$, and write $F \sslash G \coloneqq |F^\bullet|$ for the homotopy quotient of 
the group action. By construction, this comes equipped with maps $F \to F \sslash G \to \B G$.

\begin{example}\label{ex:pro-actions}
If $G$ is a levelwise group acting on $F \in \Pro(\scr X_{<\infty})$, then $F \to F \sslash G \to \B G$ is a fiber sequence.
(This is a reformulation of the fact that the counit of the above adjunction is an equivalence.)
\end{example}

We used the following fact, for which we could not find a reference.
\begin{lemma} \label{lemm:free-action}
Let $\scr C$ be an $\infty$-category with finite products and $G \in \Grp(\scr C)$.
The forgetful functor $U\colon \Fun(\Delta^\op, \scr C)^\cart_{/\B_\bullet G} \to \scr C$ has a left adjoint $F$ such that for $X \in \scr C$ the counit $X \to UFX$ is given by the canonical map $X \wequi X \times * \to X \times G$.
\end{lemma}
\begin{proof}
Increasing the size of the universe if necessary, we may assume that $\scr C$ is small.
The claim implies that $(FX)_n \wequi G^n \times X$, so it suffices to prove it in the larger $\infty$-category $\PSh(\scr C)$.
We may thus assume that $\scr C$ is an $\infty$-topos.
Now $\Fun(\Delta^\op, \scr C)^\cart_{/\B_\bullet G} \wequi \scr C_{/\B G}$ \cite[Theorem 6.1.3.9 (3)]{highertopoi}.
Writing $s\colon * \to \B G$ for the canonical map, the functor $U$ identifies with $s^*$ and so the left adjoint $F$ is given by $s_\sharp$.
It follows that $UFX \wequi X \times_{\B G} *$.
The result follows by considering the following diagram of pullback squares
\begin{equation*}
\begin{CD}
UFX @>>> G @>>> * \\
@VVV   @VVV   @VVV \\
X @>>> * @>>> \B G.
\end{CD}
\end{equation*}
\end{proof}

\section{Setup for the rest of the article} \label{sec:setup}
Throughout this article, we will work with the data of an adjunction of $\infty$-categories \[ L_+\colon \scr X \adj \scr D\noloc R. \]
Here $\scr X$ is an $\infty$-topos, $\scr D$ is stable and presentably symmetric monoidal, and $L_+$ is symmetric monoidal.
We write $L\colon \scr X_* \to \scr D$ for the induced functor.
We assume given $\scr D^\sld_{\ge 0} \subset \scr D$, which is the non-negative part of a $t$-structure (thus $\scr D^\sld_{\ge 0}$ is closed under colimits and extensions in $\scr D$).
We denote by $\scr D^\sld \subset \scr D$ the localizing subcategory generated by $\scr D^\sld_{\ge 0}$ and call it the subcategory of \emph{solid objects}.
Objects in $\scr D^\sld_{\ge 0}$ are called \emph{solid connective}.
We are concerned with the following \emph{axioms}:
\begin{itemize}
  \axiom{C}{C}{The restricted functor $R\colon \scr D^\sld_{\ge 0} \to \scr X$ preserves sifted colimits.}
  \axiom{M}{M}{For $X \in \scr X_*$ and $E \in \scr D^\sld_{\ge 0}$ we have $LX \otimes E \in \scr D^\sld_{\ge 0}$.}
  \axiom{S}{S}{\begin{itemize}
  \axiom{1}{S1}{We have $LR(\scr D^{\sld}_{\ge 0}) \subset \scr D^{\sld}_{\ge 0}$.}
  \axiom{2}{S2}{For $E \in \scr D^\sld_{\ge 1}$ the canonical map $\pi_1 LRE \to \pi_1 E$ is an isomorphism. (Here $\pi_1$ refers to the $t$-structure specified by $\scr D^\sld_{\ge 0}$.)}
  \axiom{3}{S3}{$R\colon \scr D^\sld_{\ge 0} \to \scr X$ is conservative.}
  \end{itemize}}   
\end{itemize}

For convenience, we denote by $\scr X^\sld_{*,\ge n} \subset \scr X_*$ the subcategory of objects which are $n$-connective and such that $LX \in \scr D^\sld_{\ge n}$.

\subsection{First consequences}
The following result is standard:
\begin{lemma}[Stable splitting] \label{lem:stable-splitting}
  Let $X, Y \in \scr{X}_*$.
  Then there is a split cofiber sequence 
  \begin{equation*}
    LX \oplus LY \to L(X \times Y) \to L(X \wedge Y),
  \end{equation*}
  natural in $X$ and $Y$.
\end{lemma}
\begin{proof}
  Apply $L$ to the cofiber sequence $X \vee Y \to X \times Y \to X \wedge Y$,
  and use that $L(X \vee Y) \wequi LX \oplus LY$ to construct a splitting of the inclusion.
\end{proof}
\begin{lemma}\label{lem:bar-construction-stable}
  The functor $R \colon \scr D \to \scr X$ factors 
  canonically over $\Grp(\scr X)$.
  Moreover, if axiom \axiomref{C} is satisfied, then for every $E \in \scr D^\sld_{\ge 0}$, there is an equivalence 
  $\B R E \wequi R \Sigma E$.
\end{lemma}
\begin{proof}
  Since $\scr D$ is stable, we have $\Grp(\scr D) \wequi \scr D$,
  and since $R$ preserves products, we get an induced functor 
  $\scr D \wequi \Grp(\scr D) \to \Grp(\scr X)$ that refines $R$.
  For the second statement, note that 
  $\Sigma E$ can be computed as the 
  geometric realization of the simplicial object 
  \begin{center}
  \begin{tikzcd}[column sep=1.5em]
    0 &\ar[l, shift left] \ar[l, shift right] E &  \ar[l, shift left=5pt] \ar[l, shift right=5pt] \ar[l] E \oplus E & \ar[l, shift left] \ar[l, shift left=7pt] \ar[l, shift right] \ar[l, shift right=7pt] \cdots,
  \end{tikzcd}
  \end{center}
  see e.g.\ \cite[Lemma 2.7]{bachmann2021norms}.
  But under the canonical equivalence $\Grp(\scr D) \wequi \scr D$,
  this is just the bar construction of the group structure on $E$.
  Now, since $R$ commutes with sifted colimits of objects in $\scr D^\sld_{\ge 0}$ by assumption,
  and $E$ (and thus also $E^{\oplus n}$) is in $\scr D^\sld_{\ge 0}$,
  the result follows.
\end{proof}
\begin{remark} \label{rem:conenctivity-geom-real}
  We will frequently use the following statement:
  Suppose that $X_\bullet$ is a simplicial object in a stable $\infty$-category $\scr E$ with a t-structure.
  If there exists an $n$ such that $X_0 \in \scr E_{\ge n+1}$ and $X_i \in \scr E_{\ge n}$ for all $i > 0$,
  then $|X_\bullet| \in \scr E_{\ge n+1}$.
  Indeed, this follows immediately from \cite[Proposition 1.2.4.5]{higheralgebra}.
\end{remark}
\begin{lemma} \label{lemm:first-consequences}
\begin{enumerate}
\item Axiom \axiomref{C} implies that $R\scr D^\sld_{\ge n} \subset \scr X_{*,\ge n}$.
\item Axioms \axiomref{C} and \axiomref{S1} together imply that $R\scr D^\sld_{\ge n} \subset \scr X_{*,\ge n}^\sld$.
\item Axiom \axiomref{M} implies that for $X \in \scr X_{*,\ge n}$ and $E \in \scr D^\sld_{\ge m}$ we have $LX \otimes E \in \scr D^\sld_{\ge n+m}$.
\item Axioms \axiomref{C}, \axiomref{S1} and \axiomref{S3} together imply that every object in $\scr D^\sld_{\ge n}$ can be written as a sifted colimit of objects of the form $LX$, for $X \in \scr X^\sld_{*,\ge n}$.
\item Axioms \axiomref{C}, \axiomref{S} and \axiomref{M} together imply that for $E \in \scr D^\sld_{\ge n}$ with $n \ge 1$ we have $\fib(LRE \to E) \in \scr D^\sld_{\ge n+1}$.
\end{enumerate}
\end{lemma}
\begin{proof}
(1) We prove the claim by induction on $n$, the case $n=0$ being vacuous.
Now let $E \in \scr D^\sld_{\ge n}$ with $RE \in \scr X_{*,\ge n}$; it suffices to show that $R\Sigma E \in \scr X_{*,\ge n +1}$.
But $R\Sigma E \wequi \B RE$ by \cref{lem:bar-construction-stable}.
The result follows from the fact that the bar construction takes $n$-connective objects to $(n+1)$-connective objects; see \cite[Remarks 5.2.2.19 and 5.2.6.17]{higheralgebra}.

(2) Arguing inductively again, suppose $E \in \scr D^\sld_{\ge n}$ with $LRE \in \scr D^\sld_{\ge n}$.
Then $LR\Sigma E$ is the geometric realization of a simplicial object $LR S_\bullet$ with $LR S_i = LR(E^{\oplus i}) \in \scr D_{\ge n}^\sld$.
Since $LR S_0 = 0 \in \scr D^\sld_{\ge n+1}$, it follows that $|LR S_\bullet| \in \scr D^\sld_{\ge n+1}$ as needed (\cref{rem:conenctivity-geom-real}).

(3) This is clear if $X \wequi \Sigma^n X'$.
It follows in general since $\scr X_{*,\ge n}$ is generated under (sifted) colimits by objects of this form,
see for example \cite[Proposition 2.39]{mattis2025etale}.

(4) Note that by definition and (2) the adjunction $L \dashv R$
restricts to an adjunction $L \colon \scr X_{*,\ge n}^\sld \adj \scr D_{\ge n}^\sld \noloc R$.
Hence, the result follows from the proof of \cite[Proposition 4.7.3.14]{higheralgebra} 
if we can show that this induced adjunction is monadic.
For this, it suffices to show that $R \colon \scr D_{\ge n}^\sld \to \scr X_{*,\ge n}^\sld$ is conservative 
and preserves sifted colimits. Conservativity follows immediately from \axiomref{S3}.
Since both $\scr D_{\ge n}^\sld \subset \scr D_{\ge 0}^\sld$ 
and $\scr X_{*,\ge n}^\sld \subset \scr X_*$ are closed under (sifted) colimits 
(for the first, note that this is true for any t-structure, and for the second use 
that both $n$-connective objects, and objects that under $L$ map to $\scr D_{\ge n}^\sld$ are closed under 
colimits),
this is an immediate consequence of \axiomref{C}.

(5)
We first treat the case $n=1$.
The functor $\scr D^\sld_{\ge 1} \to \scr D^\sld_{\ge 1}$, $E \mapsto \fib(LRE \to E)$ (which is well-defined by \axiomref{S2} and (2)) preserves sifted colimits by \axiomref{C}.
Thus, to prove that it takes values in $\scr D^\sld_{\ge 2}$, by (4) we need to check this only for objects of the form $E=LX$, where $X \in \scr X^\sld_{*,\ge 1}$.
In this case the map $LRE \to E$ is split and so induces an epimorphism on $\pi_2$, as needed.
Now we prove the case of general $n$ by induction.
Thus assume given $E \in \scr D^\sld_{\ge n}$ with $\fib(LRE \to E) \in \scr D^\sld_{\ge n+1}$.
Consider again the simplicial object $S_\bullet$ from (2).
We note that \[ LR(E \oplus E) \wequi L(RE \times RE) \wequi LRE \oplus (LRE \otimes LRE) \oplus LRE, \] 
(using \cref{lem:stable-splitting}) with $LRE \otimes LRE \in \scr D^\sld_{\ge n+1}$ (use (1), (3) and recall $n \ge 1$), and similarly for the higher terms.
It follows that the fiber of the canonical map $|LRS_\bullet| \wequi LR\Sigma E \to \Sigma E \wequi |S_\bullet|$ is the geometric realization of a simplicial object with entries $\fib(LR S_\bullet \to S_\bullet)$, i.e., a sum of ``cross terms'' (like $LRE \otimes LRE$) and copies of $\fib(LRE \to E)$. 
Both of these lie in $\scr D_{\ge n+1}^\sld$, and the degree zero term vanishes, whence by \cref{rem:conenctivity-geom-real} the geometric realization itself lies in $\scr D_{\ge n+2}^\sld$.
This proves the claim.
\end{proof}

\begin{proposition} \label{prop:stability-consequence}
Axiom \axiomref{M} implies the following.
Let $F \to E \to B$ be a fiber sequence, $B$ connected and $F \in \scr X^\sld_{*,\ge n}$.
Then $\fib(LE \to LB) \in \scr D^\sld_{\ge n}$.
In particular, if $F, B \in \scr X^\sld_{*,\ge n}$ then also $E \in \scr X^\sld_{*,\ge n}$.
\end{proposition}
\begin{proof}
Writing $E \wequi F \sslash \Omega B$ and $B \wequi * \sslash \Omega B$, we see that $E \to B$ is a colimit of maps of the form $\mathrm{pr_X} \colon F \times X \to X$.
It will be enough to show that $\fib(L(F \times X) \to LX) \in \scr D^\sld_{\ge n}$.
But by \cref{lem:stable-splitting}, $L(F \times X) \wequi LF \oplus LX \oplus (LF \otimes LX)$, so the fiber is $LF \oplus (LF \otimes LX)$, which lies in $\scr D^\sld_{\ge n}$ by \axiomref{M}.
\end{proof}

\begin{example} \label{ex:solid-products}
As a particular case of the last part of \cref{prop:stability-consequence}, we see that if $F, B \in \scr X^\sld_{*,\ge n}$ then also $F \times B \in \scr X^\sld_{*,\ge n}$.
\end{example}

\section{Homological localization} \label{subsec:locn}
The left adjoint functor $L_+\colon \scr X \to \scr D$ induces a Bousfield localization 
$\hat L\colon \scr X \to \scr X$ such that for a morphism $\alpha$ in 
$\scr X$, $L_+\scr \alpha$ is an equivalence if and only if $\hat L \alpha$ is.
This exists by the general theory of Bousfield 
localizations in presentable $\infty$-categories, as outlined in \cite[§5.5.4]{highertopoi}.
We say that $X \in \scr X$ is \emph{$L$-complete} if it is a local object for $\hat{L}$,
and that a morphism $f \colon X \to Y$ in $\scr X$ is an \emph{$L$-equivalence} 
if $Lf$ (or equivalently $\hat{L}f$) is an equivalence.

\begin{remark} \label{rem:locn:L-plus-eq-iff-L-eq}
Note that $\hat L$ preserves the terminal object and hence defines an endofunctor of $\scr X_*$.
Moreover, $\hat L|_{\scr X_*}$ is the Bousfield localization functor corresponding to $L\colon \scr X_* \to \scr D$.
(Indeed given a morphism $\alpha$ in $\scr X_*$, we have $L(\alpha) \oplus L(\id_{S^0}) \wequi L_+(\alpha)$, so that $L(\alpha)$ is an equivalence if and only if $L_+(\alpha)$ is.)
\end{remark}

\begin{remark}
The functor $\hat L$ preserves finite products, since $L_+$ is symmetric monoidal.
\end{remark}

\begin{theorem} [Bousfield--Kan fiber lemma] \label{thm:locn}
Let $L_+\colon \scr X \to \scr D$ be a left adjoint, symmetric monoidal functor from an $\infty$-topos to a presentably  symmetric monoidal, stable $\infty$-category.
Assume that $\scr D$ has a left-separated $t$-structure (i.e., has no nonzero $\infty$-connective objects) compatible with the symmetric monoidal structure such that $L_+(\scr X) \subset \scr D_{\ge 0}$.

Let $F \to E \to B \in \scr X_*$ be a fiber sequence.
Assume that $B$ and $\hat L B$ are simply connected.
Then $\hat L F \to \hat L E \to \hat L B$ is a fiber sequence.
\end{theorem}
\begin{proof}
We shall first prove the result under the additional assumption that $\hat L \Omega B \wequi \Omega \hat L B$.
Set $\hat F\coloneqq \fib(\hat L E \to \hat L B)$.
This is $L$-complete; it thus suffices to show that the canonical map $F \to \hat F$ is an $L$-equivalence, or equivalently an $L_+$-equivalence (\cref{rem:locn:L-plus-eq-iff-L-eq}), 
i.e., that $L_+F \to L_+\hat F$ is an equivalence.
Recall that $F$ has an action by $\Omega B$ such that $E \wequi F \sslash \Omega B$ \cite[Proposition 3.8]{Nikolaus_2014}, and similarly $\hat{L} E \wequi \hat{F} \sslash \Omega \hat{L} B$.
Moreover, since $F \to \hat F$ arises from a map of fiber sequences, this map intertwines the actions along $\Omega B \to \Omega \hat{L} B$.
Since the latter map is in fact an $L$-equivalence (by our additional assumption) and $L_+$ is a symmetric monoidal functor, we obtain an augmented $\scr E_1$-algebra $G=L_+\Omega B$ acting on $L_+ F \to L_+ \hat F$, with quotients given by $L_+ E$ and $L_+ \hat L E \wequi L_+ E$.
In particular, $G$ acts on $\cof(L_+ F \to L_+ \hat F)$ with contractible quotient.
Writing the quotient of the group action as geometric realization, we find that
\begin{center}
  \begin{tikzcd}[column sep=1.5em]
    0 = \bigg|\cof(L_+ F \to L_+ \hat F) &\ar[l, shift left] \ar[l, shift right] \cof(L_+ F \to L_+ \hat F) \otimes G &  \ar[l, shift left=5pt] \ar[l, shift right=5pt] \ar[l] \cof(L_+ F \to L_+ \hat F) \otimes G \otimes G & \ar[l, shift left] \ar[l, shift left=7pt] \ar[l, shift right] \ar[l, shift right=7pt] \cdots \bigg|.
  \end{tikzcd}
\end{center}
We will inductively show that $\cof(L_+F \to L_+\hat F)$ 
is $n$-connective for every $n$; the case $n = 0$ holds trivially. So assume we know that $\cof(L_+F \to L_+\hat F)$ 
is $n$-connective (and hence so are the other objects of the geometric realization).
It now follows from \cite[Proposition 1.2.4.5]{higheralgebra} 
that 
\begin{equation*}
  0 = \pi_n(0) \wequi \mathrm{Coeq}\left(\pi_n \cof(L_+F \to L_+\hat F) \otimes \pi_0(G) \rightrightarrows \pi_n \cof(L_+F \to L_+\hat F)\right).
\end{equation*}
Since by assumption $\Omega B$ is connected, $\pi_0(G) = \pi_0(L_+\Omega B) \wequi \pi_0(\1)$, which implies that the two maps in the coequalizer are equal (in fact, they are isomorphic to the maps obtained by an analogous diagram where we act with $\1$ using the unit map $\1 \to G$, i.e., the trivial action).
Thus, $0 = \pi_n\cof(L_+F \to L_+\hat F)$, concluding the induction.
We showed that $\cof(L_+F \to L_+\hat F)$ is $\infty$-connective, hence zero by assumption, as desired.

It remains to prove that $\hat L \Omega B \wequi \Omega \hat L B$, or equivalently,
since $\Omega \hat L B$ is $L$-complete, that the canonical map $\Omega B \to \Omega \hat L B$ 
is an $L$-equivalence.
Since $B$ and $\hat{L}B$ are connected, $\B \Omega B \wequi B$ and
$\B \Omega \hat L B \wequi \hat{L} B$.
Since $B \to \hat{L}B$ is an $L$-equivalence by construction,
we see that $\Omega B \to \Omega \hat L B$ becomes an $L$-equivalence after applying the bar construction.
Write $X \coloneqq \Omega B$, and $Y \coloneqq \Omega \hat{L} B$.
Thus, since $L$ commutes with colimits, we see that
\begin{center}
  \begin{tikzcd}[column sep=1.5em]
    0 = \left|L\cof(B^\bullet X \to B^\bullet Y)\right| = \bigg|0 &\ar[l, shift left] \ar[l, shift right] L\cof(X \to Y) &  \ar[l, shift left=5pt] \ar[l, shift right=5pt] \ar[l] L\cof(X \times X \to Y \times Y) & \ar[l, shift left] \ar[l, shift left=7pt] \ar[l, shift right] \ar[l, shift right=7pt] \cdots \bigg|.
  \end{tikzcd}
\end{center}
By assumption both $X$ and $Y$ are connected,
whence $L\cof(X \to Y) \in \scr D_{\ge 1}$.
We will inductively show that $L\cof(X \to Y) \in \scr D_{\ge n}$,
which implies the claim since $\scr D$ is left-separated.
So assume that $L\cof(X \to Y) \in \scr D_{\ge n}$ for some $n \ge 1$.
For $k \ge 2$, the $k$-th term of the above geometric realization is also in $\scr D_{\ge n}$:
Indeed, the $k$-th term is given by $L\cof(X^{\times k} \to Y^{\times k})$,
which by stable splitting, \cref{lem:stable-splitting}, is the direct sum of the following objects: 
\begin{itemize}
  \item $L\cof(X \to Y)$,
  \item $L\cof(X^{\times k-1} \to Y^{\times k -1})$, and
  \item $L\cof(X \wedge X^{\times k-1} \to Y \wedge Y^{\times k-1})$.
\end{itemize}
The first object is in $\scr D_{\ge n}$ by assumption, 
and the second object by induction on $k$ (or the assumption, if $k=2$).
The third object is even in $\scr D_{\ge n+1}$: applying the octahedral axiom to $X \wedge X^{\times k-1} \to Y \wedge X^{\times k-1} \to Y \wedge Y^{\times k-1}$ we obtain a cofiber sequence
\begin{equation*}
  L \cof(X \to Y) \otimes L X^{\times k-1} \to L\cof(X \wedge X^{\times k-1} \to Y \wedge Y^{\times k-1}) \to LY \otimes L\cof(X^{\times k-1} \to Y^{\times k-1})
\end{equation*}
where the outer terms are both in $\scr D_{\ge n+1}$ (by what we already concluded above and using that $L X^{\times k-1}$ and $LY$ are in $\scr D_{\ge 1}$),
and thus the same is true for the middle term.

Our next goal is to identify the three maps $L\cof(X \times X \to Y \times Y) \to L\cof(X \to Y)$. 
Since $L$ commutes with the cofiber, and the stable splitting is functorial,
it will suffice to analyze the three maps $L(X \times X) \to LX$ (and for $Y$, which is analogous).
By definition of the bar construction, these maps are given by 
$L(\mathrm{pr}_1)$, $L(\mathrm{pr}_2)$ and $L(m)$, where $m \colon X \times X \to X$ is the multiplication map on the loop space.
We claim that under the stable splitting equivalence, the following diagrams commute:
\begin{center}
  \begin{tikzcd}
    & L(X \times X) \ar[dl, "L(\mathrm{pr}_1)"'] \ar[dr, "L(\mathrm{pr}_2)"] \ar[d, "\wequi"] && L(X \times X) \ar[d, "\wequi"] \ar[drr, "L(m)"]\\
    LX & LX \oplus LX \oplus L(X \wedge X) \ar[l, "{(\id, 0, 0)}"] \ar[r, "{(0, \id, 0)}"'] &LX &LX \oplus LX \oplus L(X \wedge X) \ar[rr, "{(\id, \id, \widetilde{m})}"'] &&LX.
  \end{tikzcd}
\end{center}
Here, $\widetilde{m}$ is defined as the third component of the induced map.
For the left diagram, this is true by construction.
For the right diagram, it is only necessary to identify $L(m)$ on the first two components.
Since the first inclusion $LX \to L(X \times X)$ is induced by the map $X = X \vee * \to X \times X$,
and since $m$ is the multiplication of a group object where the unit is given by $* \to X$,
it follows that the composition $X \to X \times X \xrightarrow{m} X$ is homotopic to the identity,
even before applying $L$. The same argument works for the second component.

In particular, under the stable splitting equivalence,
we see that $L(\mathrm{pr}_1) - L(m) + L(\mathrm{pr}_2) \wequi -\widetilde{m} \circ \mathrm{pr}_3$.

Consider now the associated spectral sequence of the above simplicial object $C_\bullet$,
as explained in \cite[§1.2.4]{higheralgebra}. It has signature 
\begin{equation*}
  E_1^{p,q} = H_p(\mathrm{DK}^{\mathrm{un}}(\pi_q(C_\bullet))) \implies \pi_{p+q} |C_\bullet| = 0,
\end{equation*}
where $\mathrm{DK}^{\mathrm{un}}$ denotes the associated unnormalized chain complex of a simplicial object in an abelian category.
By the identification of the maps in the simplicial object done above, we see that for every $q$ the associated unnormalized chain complex 
looks as follows:
\begin{equation*}
  0 \xleftarrow{0} \pi_q(L\cof(X \to Y)) \xleftarrow{-\pi_q(\cof(\widetilde{m}_{X} \to \widetilde{m}_{Y})) \mathrm{pr}_3} \pi_q(L\cof(X \to Y))^{\oplus 2} \oplus \pi_q(L\cof(X \wedge X \to Y \wedge Y)) \leftarrow \cdots.
\end{equation*}
Note that we use here that the map $X \to Y$ is a map of loop spaces, so that we get a map 
$\widetilde{m}_X \to \widetilde{m}_Y$. In particular, 
\begin{equation*}
  H_1(\mathrm{DK}^{\mathrm{un}}(\pi_q(C_\bullet))) \wequi \mathrm{coker}\left(\pi_q(\cof(\widetilde{m}_X \to \widetilde{m}_Y)) \colon \pi_q(L\cof(X \wedge X \to Y \wedge Y) \to \pi_q(L\cof(X \to Y)))\right).
\end{equation*}
By the connectivity estimates done above, we see that $\pi_{n} L\cof(X \wedge X \to Y \wedge Y) = 0$,
whence $E_1^{1,n+1} = H_1(\mathrm{DK}^{\mathrm{un}}(\pi_{n+1}(C_\bullet))) \wequi \pi_n(L\cof(X \to Y))$.
By inspection, no differential goes in or out of this group,
therefore it survives to the $E_\infty$-page.
Since we know that the spectral sequence converges to $0$,
we get that $\pi_n(L\cof(X \to Y)) = 0$.
This concludes our induction.
\end{proof}

\begin{remark}
The above proof can be simplified somewhat by using Bar-Cobar duality.
Specifically:\NB{provide references?}
\begin{enumerate}
\item The functor $(\ph)\sslash G = (\ph) \otimes_G \1$ upgrades to an equivalence between $G$-modules and $\B G$-comodules, thus is conservative.
\item The functor $(\ph) \mapsto B(\ph) \wequi \1 \otimes_{(\ph)} \1$ upgrades to an equivalence between augmented connected $\scr E_1$-algebras and coaugmented simply connected $\scr E_1$-coalgebras, so is conservative as well.
\end{enumerate}
We included direct arguments in our special case to keep the proof more self-contained.
\end{remark}

\begin{remark}
The assumption that $\hat L B$ is simply connected often does not apply for $\infty$-topoi of positive homotopy dimension.
This restricts the usefulness of the above fiber lemma.
However, it can be utilized to prove more general fiber lemmas by reducing to the case of a presheaf $\infty$-topos; 
see \cite[Proposition 3.19]{mattis2024unstablepcompletionmotivichomotopy} and \cite[Proposition 2.8]{mattis2025etale}.
Moreover, for rationalization it is always true (at least for $\infty$-topoi with enough points) 
that $\hat L B$ is still simply connected, cf.\ \cite[Proposition 3.12]{mattis2024unstablearithmeticfracturesquares}.
\end{remark}

\section{Bousfield--Kan completion} \label{sec:BK}
The main goal of the section is to prove a generalized version of Bousfield and Kan's Principal Fibration Lemma in our more general context (see \cref{sec:setup}).
This is achieved in \cref{thm:fiber-sequences} which states that, assuming the axioms \axiomref{C}, \axiomref{M}, and \axiomref{S}, the generalized Bousfield--Kan completion functor preserves certain fiber sequences.

We begin the section by recalling how to define the functor of pro-objects $T^\bullet \colon \scr X_* \to \Pro(\scr X_*)$ associated to the adjunction $L \dashv R$.
Our next order of business is to show that $T^\bullet$ preserves finite products (\cref{prop:products}).
So far we have not used any axioms.
Making heavy use of the axioms, we next show that $\tau_{<\infty}T^\bullet$ preserves geometric realizations of connected solid object (\cref{cor:geometric-realn}) and that each of the functors $T^n$ preserves connected solid objects (\cref{cor:connected}).
The main theorem follows fairly straightforwardly from this.

\subsection{Setup} \label{subsec:setup}
Out of the adjunction $L\colon \scr X_* \adj \scr D\noloc R$ we construct an augmented cosimplicial endofunctor 
$\CB^\bullet_+\colon \scr X_* \to \Fun(\Delta_+, \scr X_*)$, see \cref{lem:adjunction:cobar-construction}.
Informally, $\CB^\bullet_+$ looks like 
\begin{center}
    \begin{tikzcd}[column sep=2em]
        \id \ar[r] &RL \ar[r, shift left] \ar[r, shift right] &(RL)^2 \ar[r, shift left=5pt] \ar[r] \ar[r, shift right=5pt] &(RL)^3 \cdots.
    \end{tikzcd}
\end{center}
Taking partial totalizations, we obtain 
\[ T^n = \Tot^n \CB^\bullet = \lim_{\Delta_{\le n}} \CB^\bullet \colon \scr X_* \to \scr X_* \] assembling into a tower \[ \dots \to T^2 \to T^1 \to T^0. \]
We also denote by $T^\bullet$ the associated pro-object,
i.e., $T^\bullet = \lim_n cT^n = \lim_{\Delta} c\CB^\bullet \colon \scr X_* \to \Pro(\scr X_*)$,
which comes equipped with an augmentation $c \to T^\bullet$.

\subsection{Products}
\begin{proposition} \label{prop:products}
The functor $T^\bullet\colon \scr X_* \to \Pro(\scr X_*)$ preserves finite products.
\end{proposition}
\begin{proof}
As $T^\bullet$ preserves the terminal object, it suffices to show that $T^\bullet$ preserves binary products.
Let $X, Y \in \scr X_*$.
Consider the augmented cosimplicial object
\begin{equation*}
    W^\bullet_+ \colon \Delta_+ \xrightarrow{((\CB)^\bullet_+)^{2}} \Fun(\scr X_*,\scr X_*)^{2} \xrightarrow{(\mathrm{ev}_X, \mathrm{ev}_Y)} \scr X_*^2 \xrightarrow{- \times -} \scr X_*.
\end{equation*}
Out of this, we can construct a double augmented cosimplicial object as follows:
\begin{equation*}
    V^{\bullet,\bullet}_+ \colon \Delta_+ \times \Delta_+ \xrightarrow{\CB^\bullet_+ \times W^\bullet_+} \Fun(\scr X_*,\scr X_*) \times \scr X_* \xrightarrow{\operatorname{eval}} \scr X_*.
\end{equation*}
Informally, $V^{\bullet,\bullet}_+$ looks like the following picture (where we do not draw the codegeneracy maps):
\begin{center}
    \begin{tikzcd}[column sep=0.8em]
        \vdots &\vdots &\vdots &\\
        (RL)^2(X \times Y) \ar[u, shift left=5pt] \ar[u] \ar[u, shift right=5pt] \ar[r] &(RL)^2(RLX \times RLY) \ar[u, shift left=5pt] \ar[u] \ar[u, shift right=5pt] \ar[r, shift left] \ar[r, shift right] &(RL)^2((RL)^2X \times (RL)^2Y) \ar[u, shift left=5pt] \ar[u] \ar[u, shift right=5pt] \ar[r, shift left=5pt] \ar[r] \ar[r, shift right=5pt] &\cdots\\
        RL(X \times Y) \ar[u, shift left] \ar[u, shift right] \ar[r] &RL(RLX \times RLY) \ar[u, shift left] \ar[u, shift right] \ar[r, shift left] \ar[r, shift right] &RL((RL)^2X \times (RL)^2Y) \ar[u, shift left] \ar[u, shift right] \ar[r, shift left=5pt] \ar[r] \ar[r, shift right=5pt] &\cdots\\
        X \times Y \ar[u] \ar[r] &RLX \times RLY \ar[u] \ar[r, shift left] \ar[r, shift right] &(RL)^2X \times (RL)^2Y \ar[u] \ar[r, shift left=5pt] \ar[r] \ar[r, shift right=5pt] &\cdots
    \end{tikzcd}
\end{center}
By definition we have
    \begin{equation*}
        \lim_{\Delta} cV^{\bullet, -1} 
        = \lim_{\Delta} c\CB^\bullet(X \times Y) 
        = T^\bullet (X \times Y).
    \end{equation*}
    Similarly, we have the following equivalence:
    \begin{align*}
        \lim_{\Delta} cV^{-1, \bullet}
        &= \lim_{\Delta} \operatorname{eval}(c, W^\bullet) \\
        &= \lim_{\Delta} c(\CB^\bullet(X) \times \CB^\bullet(Y)) \\
        &\wequi \lim_{\Delta} c\CB^\bullet(X) \times \lim_{\Delta} c\CB^\bullet(Y) \\
        &= T^\bullet(X) \times T^\bullet(Y),
    \end{align*}
    where we used that $c$ commutes with finite limits and that $\Delta$ is a cosifted category.
    Hence, we get a span
    \begin{equation*}
        T^\bullet(X \times Y)
        \wequi \lim_{\Delta} cV^{\bullet, -1} 
        \to \lim_{\Delta \times \Delta} cV^{\bullet, \bullet} 
        \leftarrow \lim_{\Delta} cV^{-1, \bullet} 
        \wequi T^{\bullet}(X) \times T^{\bullet}(Y).
    \end{equation*}
    The proposition follows if we can show that the arrows are equivalences.
    For this, it suffices to show that for every $n \ge 0$ 
    both $V^{n,\bullet}$ and $V^{\bullet,n}$ are split cosimplicial objects, cf.\ \cref{rmk:adjunction:split-limit}.
  Unwinding the definitions, we see that 
    \begin{equation*}
        V^{n, \bullet}_+ \wequi (RL)^{n+1}((\CB)^\bullet_+(X) \times (\CB)^\bullet_+(Y)).
    \end{equation*}
    This in turn is equivalent to 
    \begin{equation*}
        (RL)^{n}R\left(L(\CB)^\bullet_+(X) \oplus L(\CB)^\bullet_+(Y) \oplus \left(L(\CB)^\bullet_+(X) \otimes L(\CB)^\bullet_+(Y)\right)\right),
    \end{equation*}
    using \cref{lem:stable-splitting}.
    But now both $L(\CB)^\bullet_+(X)$ and $L(\CB)^\bullet_+(Y)$ are split by 
    \cref{lem:adjunction:split}, and so are their sums and tensor products.
    Similarly, we see that 
    \begin{equation*}
        V^{\bullet, n}_+ \wequi (\CB)^\bullet_+\left((RL)^{n+1}(X) \times (RL)^{n+1}(Y)\right) \wequi (\CB)^\bullet_+\left(R(L(RL)^{n}(X) \times L(RL)^{n}(Y))\right),
    \end{equation*}
    since $R$ commutes with products as it is a right adjoint.
    This is also split by \cref{lem:adjunction:split}.
\end{proof}

\subsection{Connectivity} \label{subsec:connectivity}
From now on to the end of the section, we will assume the axioms \axiomref{C}, \axiomref{M}, and \axiomref{S}.
We shall be utilizing the \emph{codegeneracy cubes} and the \emph{face cubes}, whose total fibers we denote respectively by \[ C^n \colon \scr X_* \to \scr X_* \quad\text{and}\quad S^n \colon \scr D \to \scr D. \]
Informally speaking, the codegeneracy cubes are made out of the codegeneracy maps in the cosimplicial object $\CB^\bullet$, whereas the face cubes are made out of the face maps of the simplicial object extracted from the adjunction $L \dashv R$.
For a careful definition of both cubes see \cref{def:adjunction:codegen-cube}.
We use two main facts about these cubes:
\begin{itemize}
\item We have \newline
  \begin{minipage}{\dimexpr\linewidth-\leftmargini\relax}
    \begin{equation}\tag{Cube1}\label{eq:CRSL}
      C^n \wequi R S^n L, 
    \end{equation}
  \end{minipage} 
  \newline\newline
  i.e., the codegeneracy maps in the cobar construction for $X$ are just obtained by applying $R$ to the face maps of the bar construction for $LX$.
  See \cref{lem:adjunction:cube-relation} for a reference.
\item We have the fiber sequence 
  \newline
  \begin{minipage}{\dimexpr\linewidth-\leftmargini\relax}
    \begin{equation}\tag{Cube2}\label{eq:C-T}
      \Omega^n C^n \to T^n \to T^{n-1}, 
    \end{equation}
  \end{minipage} 
  \newline\newline i.e., $C^n$ relates to an inductive construction of $T^\bullet$.
  (This follows from \cite[Corollary 1.2.4.18]{higheralgebra} together with the pointwise formula for right Kan extensions. Note that the stability assumption in the statement of \emph{loc.~cit.~}is not actually used in the proof, just pointedness.)
\end{itemize}

\begin{lemma} \label{lemm:Sn-conn}
The functor $S^n \colon \scr D^\sld_{\ge 1} \to \scr D^\sld$ sends $i$-connected maps to $(n+i)$-connected maps.
\end{lemma}
\begin{proof}
Let $(A_n)$ be the desired statement about $S^n$, let $(B_n)$ be the (a priori weaker) statement that $S^n$ sends $\scr D^\sld_{\ge i}$ to $\scr D^\sld_{\ge i+n}$ (for $i \ge 1$).
We shall prove that $(B_n)$ implies $(A_n)$ and $(A_n)$ implies $(B_{n+1})$.
Since $S^0 = \id$, both $(A_0)$ and $(B_0)$ are trivial; hence by induction we will prove what we want.

$(A_n) \Rightarrow (B_{n+1})$:
Let $E \in \scr D^\sld_{\ge i}$.
The map $LRE \to E$ has fiber in $\scr D^\sld_{\ge i+1}$, by \cref{lemm:first-consequences}(5).
Hence, by $(A_n)$, the fiber of $S^n(LRE) \to S^n(E)$ lies in $\scr D^\sld_{\ge i+1+n}$.
But this fiber is $S^{n+1}(E)$ by \cref{lem:adjunction:cube-tofib-relation}, proving the claim.

$(B_n) \Rightarrow (A_n)$:
Let $F \to X \to Y$ be a fiber sequence in $\scr D$ with $X, Y \in \scr D^\sld_{\ge 1}$ and $F \in \scr D^\sld_{\ge i}$.
We write $Y \wequi \cof(F \to X)$ as the geometric realization of the simplicial object $Y_\bullet$ with entries $Y_t = X \oplus F^{\oplus t}$.
Since $S^n$ preserves sifted colimits, we find that $\fib(S^n(X) \to S^n(Y))$ is the geometric realization of the simplicial object $F_\bullet$, with $F_t = \fib(S^n(X) \to S^n(X \oplus F^{\oplus t}))$.
Note that $S^n(A \oplus B) \wequi S^n(A) \oplus S^n(B) \oplus S^n(A,B)$, defining a new functor $S^n(\ph,\ph)$.
In terms of this we have $F_t \wequi \Omega S^n(F^{\oplus t}) \oplus \Omega S^n(X, F^{\oplus t})$.
In particular $F_0 = 0$.
By $(B_n)$ we know that $S^n(F^{\oplus t}) \in \scr D^\sld_{\ge n+i}$.
Using \cref{rem:conenctivity-geom-real}, it will thus suffice to show that $S^n(X, F^{\oplus t}) \in \scr D^\sld_{\ge n+i}$.
We shall prove more generally that for $Z \in \scr D^\sld_{\ge i}$ we have $S^n(X, Z) \in \scr D^\sld_{\ge n+i}$.
When $i=1$ this follows from $(B_n)$ ($S^n(X,Z)$ being a summand of $S^n(X \oplus Z)$).
Inductively, assume the claim proved for $i$ and let $\Sigma Z \in \scr D^\sld_{\ge i+1}$.
Then $S^n(X, \Sigma Z) \wequi |S^n(X, Z^{\oplus \bullet})|$ (as $S^n(X, \ph)$ preserves sifted colimits, being a summand of $S^n(X \oplus \ph)$).
Since each term in the simplicial object lies in $\scr D^\sld_{\ge i}$ (by induction) and the zeroth term is zero, the claim follows (again by \cref{rem:conenctivity-geom-real}).
\end{proof}

\begin{proposition} \label{prop:Tn-conn}
Let $F \to X \to Y \in \scr X_*$ be a fiber sequence with $X, Y \in \scr X^\sld_{*,\ge 1}$ and $F \in \scr X^\sld_{*,\ge i}$ (with $i \ge 1$).
Then \[ \fib(T^n(X) \to T^n(Y)) \in \scr X^\sld_{*,\ge i}. \]
\end{proposition}
\begin{proof}
Recall the fiber sequence \eqref{eq:C-T}.
By induction using \cref{prop:stability-consequence}, we thus reduce to showing that the functor $\Omega^nC^n$ has the desired property.
Now $\Omega^nC^n \wequi R\Omega^nS^nL$ by \eqref{eq:CRSL}, so by \cref{lemm:Sn-conn} it is enough to show that $\fib(LX \to LY) \in \scr D^\sld_{\ge i}$.
This is \cref{prop:stability-consequence} again.
\end{proof}

\begin{corollary} \label{cor:connected}
The functor $T^n$ preserves $\scr X^\sld_{*,\ge i}$ (for $i \ge 1$).
\end{corollary}
\begin{proof}
It is clear that $T^n(*)=*$, so this is the special case of \cref{prop:Tn-conn} applied to a map to the terminal object.
\end{proof}

\subsection{Geometric realizations}
\begin{proposition} \label{prop:Tn-sifted-colim}
The functor $T^n$ preserves sifted colimits of connected solid objects.
\end{proposition}
\begin{proof}
We prove the claim by induction on $n$, the case $n=0$ being clear since $T^0=RL$ preserves sifted colimits of connective solid objects by axiom \axiomref{C}.
Let $X_\bullet \in \scr X^\sld_{*,\ge 1}$ be a sifted diagram.
By induction, we know that $|T^n(X_\bullet)| \wequi T^n(|X_\bullet|)$.
We have
\begin{align*}
  |\Omega^{n+1}C^{n+1}(X_\bullet)| & \stackrel{(1)}{\wequi} |R\Omega^{n+1}S^{n+1}(LX_\bullet)|  \stackrel{(2)}{\wequi} R|\Omega^{n+1}S^{n+1}(LX_\bullet)| \\
  & \stackrel{(3)}{\wequi} R\Omega^{n+1}S^{n+1}(L|X_\bullet|) \space \stackrel{(1)}{\wequi} \Omega^{n+1}C^{n+1}(|X_\bullet|),
\end{align*}
where (1) holds by \eqref{eq:CRSL}, (2) holds because $\Omega^{n+1}S^{n+1}$ preserves solid connected objects (\cref{lemm:Sn-conn}) and $R$ preserves sifted colimits of solid connected objects \axiomref{C}, and (3) holds because $\Omega^{n+1}S^{n+1}$ is a total fiber computed in a stable $\infty$-category and hence preserves all colimits that are preserved by $(LR)^k$,
in particular sifted colimits of solid connected objects.
Recall the fiber sequence \eqref{eq:C-T}.
Using that pullbacks preserve sifted colimits in $\scr X$ with connected base \cite[Lemma 5.5.6.17]{higheralgebra}, we see that $|\Omega^{n+1}C^{n+1}(X_\bullet)| \to |T^{n+1}(X_\bullet)| \to |T^n(X_\bullet)|$ is a fiber sequence.
We now have a morphism of fiber sequence
\begin{equation*}
\begin{CD}
|\Omega^{n+1}C^{n+1}(X_\bullet)| @>>> |T^{n+1}(X_\bullet)| @>>> |T^n(X_\bullet)| \\
@VVV  @VVV   @VVV \\
\Omega^{n+1}C^{n+1}(|X_\bullet|) @>>> T^{n+1}(|X_\bullet|) @>>> T^n(|X_\bullet|).
\end{CD}
\end{equation*}
The left- and right-hand maps are equivalence by what we have already said, and the base is connected, so we can 
conclude by using e.g.\ \cite[Proposition 3.8]{Nikolaus_2014}.
\end{proof}

\begin{corollary} \label{cor:geometric-realn}
The functor $\tau_{<\infty}T^\bullet$ preserves geometric realizations of connected solid objects.
\end{corollary}
\begin{proof}
Immediate from \cref{prop:Tn-sifted-colim,prop:pro-truncated-geometric-realization}.
\end{proof}

\subsection{Fiber sequences}
We now have everything to prove our version of the celebrated Principal Fibration Lemma 
of Bousfield and Kan \cite{BousfieldKan:HomotopyLimitsCompLoc}.
\begin{theorem} [Principal Fibration Lemma] \label{thm:fiber-sequences}
Assume axioms \axiomref{C}, \axiomref{M} and \axiomref{S}.

Let $F \to E \to B \in \scr X_*$ be a fiber sequence with $B$ simply connected and $F, \Omega B \in \scr X^\sld_{*,\ge 1}$.
Then $\tau_{<\infty} T^\bullet F \to \tau_{<\infty} T^\bullet E \to \tau_{<\infty} T^\bullet B$ is a fiber sequence.
\end{theorem}
\begin{proof}
We know that $\Omega B$ is a group acting on $F$ with quotient $F \sslash \Omega B \wequi E$ \cite[Proposition 3.8]{Nikolaus_2014}.
By \cref{prop:products}, $\tau_{<\infty} T^\bullet \Omega B$ is a group acting on $\tau_{<\infty} T^\bullet F$.
Note that $\Omega B^n \times F^{\epsilon} \in \scr X^\sld_{*,\ge 1}$ for all $n, \epsilon \ge 0$ (\cref{ex:solid-products}).
Thus, \cref{cor:geometric-realn} implies that the quotient is \[ \tau_{<\infty} T^\bullet F \sslash \tau_{<\infty} T^\bullet \Omega B \wequi \tau_{<\infty} T^\bullet(F \sslash \Omega B) \wequi \tau_{<\infty} T^\bullet E. \]
Since $T^\bullet \Omega B$ is pro-connected by \cref{cor:connected}, it is a levelwise group by \cref{theorem:connected-levelwise}, and so by \cref{ex:pro-actions} we have a fiber sequence \[ \tau_{<\infty} T^\bullet F \to \tau_{<\infty} T^\bullet E \wequi \tau_{<\infty} T^\bullet F \sslash \tau_{<\infty} T^\bullet \Omega B \to \B \tau_{<\infty} T^\bullet \Omega B. \]
Using once more that $\tau_{<\infty} T^\bullet$ preserves geometric realizations and finite products (of connected solid objects) we learn that $\B \tau_{<\infty} T^\bullet \Omega B \wequi \tau_{<\infty} T^\bullet \B \Omega B$, which is the same as $\tau_{<\infty} T^\bullet B$, as needed.
\end{proof}

\section{Applications} \label{sec:applications}
In this final section we reap the fruits of the previous sections to prove generalized versions of Bousfield and Kan's \cite{BousfieldKan:HomotopyLimitsCompLoc} convergence of their completion functor.
We will start with an adjunction \[ \Sigma^\infty_+ \colon \scr C \adj \scr E \noloc \Omega^\infty, \] where $\scr C$ is some $\infty$-category of generalized spaces, and $\scr E$ is an associated $\infty$-category of spectra.
We will pick $A \in \CAlg(\scr E)$ and write $\scr C$ as a localization of an $\infty$-topos $\scr X$.
Putting $\scr D = \Mod_A(\scr E)$, we obtain a composed adjunction \[ \scr X \adj \scr C \adj \scr E \adj \scr D \] to which we shall apply the results of the previous sections.

There are four classes of examples for this broad framework that we have in mind.
In \cref{subsec:spaces} we consider the case $\scr X = \scr C = \Spc$ the usual $\infty$-category of spaces, and $\scr E = \Sp$ the usual $\infty$-category of spectra.
We thus provide an account of the original results of Bousfield and Kan.
We will assume that $A$ is connective.
The behavior of the completion is then strongly influenced by $\pi_0(A)$, which we shall assume to be either a subring of $\Q$ of the form $S^{-1}\Z$, or a quotient of $\Z$, whence $\Z/n$.
Depending on this, the Bousfield--Kan completion coincides in favorable cases (i.e., for nilpotent spaces) with either the unstable $S$-localization or the unstable $n$-completion.

Next in \cref{subsec:sheaves} we generalize somewhat, by allowing $\scr C=\scr X$ to be an $\infty$-topos satisfying some mild assumptions.
We still set $\scr E = \Sp(\scr X)$ and assume that $\ul\pi_0(A)$ is a constant sheaf $S^{-1}\ul\Z$ or $\ul\Z/n$.
Again we prove that the Bousfield--Kan completion (or more precisely, its connected cover) agrees (for nilpotent sheaves) with an unstable $S$-localization or $n$-completion functor.

Our next example, treated in \cref{subsec:equivariant}, is $\scr C = \Spc(\B G)$, i.e. the $\infty$-category of genuine $G$-equivariant spaces.
We only consider finite groups $G$.
We can set $\scr X = \scr C$, since $\Spc(\B G)$ is a presheaf topos.
However, we no longer set $\scr E = \Sp(\Spc(\B G))$.
Indeed, as the practice of genuine equivariant stable homotopy theory shows, a much more useful $\infty$-category of spectra is $\SH(\B G)$, obtained by inverting on $\Spc(\B G)_*$ all the representation spheres.
We again pick $A \in \SH(\B G)_{\ge 0}$.
But now $\ul\pi_0(\1_{\SH(\B G)})=\ul A$ (the Burnside ring Mackey functor) is no longer the constant presheaf $\ul{\Z}$, and so more options arise for $\ul\pi_0 A$.
We treat four cases, (1) $\ul\pi_0 A = S^{-1}\ul A$, (2) $\ul\pi_0 A = \ul A/n$, (3) $\ul\pi_0 A = S^{-1}\ul \Z$ and (4) $\ul\pi_0 A = \ul \Z/p$.
Cases (1) and (2) behave very similarly to the previous situations, and we call them \emph{Burnside resolutions}.
Cases (3) and (4) are more delicate.
We call them \emph{$\ul\Z$-resolutions}.

Our final example is the $\infty$-category of motivic spaces $\scr C = \Spc(k)$, where $k$ is a perfect field; see \cref{subsec:motivic}.
As in the equivariant case, the correct $\infty$-category of spectra is not the naive stabilization, but rather the $\infty$-category of motivic spectra $\scr E = \SH(k)$.
Moreover, in contrast to the previous three examples, $\scr C$ is not a topos, and so we must specify $\scr X = \Shv_\Nis(\Sm_k)$.
As in the sheaves case, $\scr X$ has positive (even infinite) homotopy dimension, so care has to be taken with completions and connectivity assumptions.
As in the equivariant case, $\ul\pi_0(\1_k)_*$ is more complicated than just $\ul\Z$, and so more options arise for $\ul\pi_0(A)_*$.
Here we treat the cases $\ul\pi_0(A)_* = S^{-1} \ul\pi_0(\1_k)_*$, $\ul\pi_0(A)_* = \ul\pi_0(\1_k)_*/n$ (which we call \emph{Milnor--Witt resolutions}) and $\ul\pi_0(A)_* = S^{-1} \ul\pi_0(\1_k)_*/\eta$, $\ul\pi_0(A)_* = \ul\pi_0(\1_k)_*/(n,\eta)$ (which we call \emph{Milnor resolutions}).

\subsection{Spaces} \label{subsec:spaces}
We set $\scr X = \Spc$.
In this subsection we explain how our results specialize to those of Bousfield--Kan \cite{BousfieldKan:HomotopyLimitsCompLoc}.
Our presentation is designed as a drop-in replacement for \emph{loc. cit.}

Let $A \in \CAlg(\Sp)_{\ge 0}$ such that $\pi_0(A)$ is either a subring of $\Q$ or a quotient of $\Z$, so that $\pi_0(A) \otimes_\Z \pi_0(A) \wequi \pi_0(A)$.
Set $\scr D = \Mod_A(\Sp)$, $\scr D^\sld_{\ge 0} = \scr D_{\ge 0}$ and $L_+ = A \otimes \Sigma^\infty_+$.
Hence, we have $\Spc_{*,\ge n}^\sld \supseteq \Spc_{*,\ge n}$ for all $n \ge 0$.
The axioms are readily verified.
(In fact, this is a special case of the situation of \cref{subsec:sheaves}, where we verify the axioms more fully in \cref{lem:sheaves:axioms}.)
In particular, for $X \in \scr D_{\ge 1}$ we have \[ \pi_1(LRX) \wequi \pi_1(X) \otimes_\Z \pi_0(A) \wequi \pi_1(X) \otimes_{\pi_0(A)} \pi_0(A) \otimes_\Z \pi_0(A) \wequi \pi_1(X), \] whence \axiomref{S2} is satisfied.
To emphasize the dependence on $A$, we denote the localization functor (in the sense of \cref{subsec:locn}) by $\hat L_A$ and the completion functor (in the sense of \cref{sec:BK}) by $T^\bullet_A$.

Let us recall the stable analog.
On $\Sp_{\ge 0}$ we can define a Bousfield localization $L_A^{st}$, where $L_A^{st}(\alpha)$ is an equivalence if and only if $\alpha \otimes A$ is one.
(In fact, for $E \in \Sp_{\ge 0}$, $L_A^{st} E$ coincides with the totalization of corresponding monadic resolution, see e.g.\ \cite[Theorem 1.0.4]{mantovani2021localizationscompletionsstableinftycategories}, but we do not need this.)
Its effect can be described explicitly, see \cref{lem:stable-bousfield}:
\begin{enumerate}
\item If $\pi_0 A = \Z/n$ then $L_A^{st} E \wequi E_n^\comp = \lim_i E/n^i$.
\item If $\pi_0 A = S^{-1} \Z \subset \Q$ then $L_A^{st} E \wequi S^{-1} E$ (so that $\pi_* (L_A^{st} E) \wequi \pi_* (E) \otimes \pi_0 (A)$).
\end{enumerate}

We can now state the main result of this section.
\begin{theorem}[{Bousfield--Kan {\cite{BousfieldKan:HomotopyLimitsCompLoc}}}] \label{thm:BK-spaces}
\begin{enumerate}
\item Let $X \in \Spc_*$ be connected and nilpotent (i.e.\ $\pi_1(X)$ is nilpotent,
  with nilpotent action on $\pi_n(X)$ for all $n \ge 2$).
  The canonical map $\hat L_A X \to \Tot T^\bullet_A X$ is an equivalence.
  In fact, if $\pi_0 (A) \subset \Q$ then already the canonical map $\tau_{<\infty} \hat L_A X \to \tau_{<\infty} T_A^\bullet X$ is an equivalence.
\item For $E \in \Sp_{\ge 1}$ the canonical map $\hat L_A \Omega^\infty E \to \Omega^\infty L_A^{st} E$ is an equivalence.
\item Let $F \to E \to B$ be a fiber sequence of pointed, connected, nilpotent spaces, with $B$ simply connected.
  Then $\hat L_A F \to \hat L_A E \to \hat L_A B$ is a fiber sequence.
\item Let $X_\bullet$ be a tower of pointed, connected, nilpotent spaces with $X = \lim_i X_i$.
  Assume that the connectivity of $X \to X_i$ tends to infinity with $i$.
  Then $\hat L_A X \wequi \lim_i \hat L_A X_i$.
\end{enumerate}
\end{theorem}
Note that this result supplies us with a recipe to compute $\hat L_A X$ (or equivalently $\Tot T^\bullet_A$) for $X$ pointed, connected and nilpotent.
Indeed, $X$ being nilpotent, we can find fiber sequences $X_{i+1} \to X_i \to B_i$ where $B_i$ is a simply connected Eilenberg--Mac Lane space, the connectivity of the $B_i$ tends to infinity with $i$, and $X \wequi \lim_i X_i$.
By (2) we know $\hat L_A B_i$ (in particular we know its homotopy groups), whence by (3) we can inductively understand $\hat L_A X_i$.
Finally, by (4) we have $\hat L_A X \wequi \lim_i \hat L_A X_i$.
The following formulae can be inductively deduced from the stable situation:
\begin{example}
Let $\pi_0 A \subset \Q$ and $X$ be simply connected.
Then $\pi_* \hat L_A X \wequi \pi_* X \otimes \pi_0 A$.
In particular if $\pi_0 A = \Z$ then $\hat L_A X \wequi X$.
\end{example}
\begin{example}
Let $\pi_0 A = \Z/p$, $X$ simply connected and $\pi_n X$ finitely generated for all $n$.
Then $\pi_* \hat L_A X \wequi \pi_* X \otimes \Z_p^\comp$.
\end{example}

The proof runs slightly differently depending on if $\pi_0 A \subset \Q$ or not, with the former case being somewhat easier.
We treat this first.
\begin{proof}[Proof of \cref{thm:BK-spaces} when $\pi_0 A \subset \Q$.]
We first prove (2).
It is clear from the definitions that $\Omega^\infty L_A^{st} E$ is $L$-local.
It thus suffices to show that $\Omega^\infty E \to \Omega^\infty L_A^{st} E$ is an $L$-equivalence.
We can write $L_A^{st} E$ as the filtered colimit of a diagram of copies of $E$, with maps given by multiplication by elements in $\Z \cap \pi_0(A)^\times$.
It thus suffices to show: if $E \in \Sp_{\ge 1}$ and $n \in \Z \cap \pi_0(A)^\times$, then $L\Omega^\infty(E \xrightarrow{n} E)$ is an equivalence.
Observe that $L \Omega^\infty E \wequi \lim_i L \Omega^\infty(E_{\le i})$:
Indeed, the map $\Omega^\infty E \to \Omega^\infty(E_{\le i})$ is $(i+1)$-connective,
and $L$ preserves connectivity of maps (\cref{prop:stability-consequence}).
In particular, the map $L \Omega^\infty E \wequi \lim_i L \Omega^\infty(E_{\le i})$ has $\infty$-connective
fiber, and thus is an equivalence. Therefore, we are reduced to the case that $E$ is bounded.
Appeal to \cref{thm:locn} shows that given a fiber sequence $E_1 \to E_2 \to E_3$ with $E_3 \in \Sp_{\ge 2}$, the result for $E_1$ and $E_3$ implies the result for $E_2$.
This way, working through the principalized Postnikov tower, we reduce to the case where $E$ is concentrated in a single degree, whence, a connected $H\Z$-module.
This $\infty$-category is generated under sifted colimits by $E=\Sigma H\Z$ itself, and so we reduce to this single case.
Now $\Omega^\infty E \wequi S^1$ and $\Omega^{\infty}(n\colon E \to E)$ is indeed the degree $n$ map from $S^1$ to itself.
Upon stabilization, this corresponds to the multiplication by $n$ map from the semiadditive structure (as this might be checked on $\pi_1(\Sigma^\infty S^1) \wequi \pi_1(S^1)$ by the Hurewicz theorem), 
and so becomes inverted upon $\otimes A$, as needed.

\cref{thm:locn} now implies the special case of (3) when $B$ is a simply-connected Eilenberg--Mac Lane space.
From this we deduce by induction that if $X$ is pointed, connected, nilpotent and truncated, then $\pi_* \hat L_A X \wequi \pi_* X \otimes \pi_0 A$.

Now we show a special case of (4), namely, the one where there are fiber sequences $X_{i+1} \to X_i \to B_i$ with $B_i$ simply connected Eilenberg--Mac Lane spaces, with connectivity tending to infinity.
Setting $X' = \lim_i \hat L_A X_i$, the key point is to show that the connectivity of $X' \to \hat L_A X_i$ tends to infinity with $i$.
For this it suffices to show that the connectivity of $\hat L_A X_{i+1} \to \hat L_A X_i$ tends to infinity with $i$.
But we have the fiber sequences $\hat L_A X_{i+1} \to \hat L_A X_i \to \hat L_A B_i$ (\cref{thm:locn}), so the claim follows.

Now we prove (1) for arbitrary (pointed, connected, nilpotent) $X$.
Present $X \wequi \lim_i X_i$
with fiber sequences $X_{i+1} \to X_i \to B_i$ where $B_i$ is a simply-connected Eilenberg--Mac Lane space and the connectivity of the $B_i$ tends to infinity.
Since the connectivity of $X \to X_i$ tends to infinity with $i$, so does the connectivity of $T^\bullet X \to T^\bullet X_i$ by \cref{prop:Tn-conn}.
Consequently, $\tau_{<\infty} T^\bullet X \wequi \lim_i \tau_{<\infty} T^\bullet X_i$.
Similarly, $\tau_{<\infty} \hat{L}_A X \wequi \lim_i \hat{L}_A X_i$ (via the special case of (4) above), 
and so it will be enough to show that $\hat{L}_A X_i \to \tau_{<\infty} T^\bullet X_i$ is an equivalence (in $\Pro(\Spc_{<\infty})$).
Since the fiber sequences $X_{i+1} \to X_i \to B_i$ are preserved by $\tau_{<\infty} T^\bullet$ (\cref{thm:fiber-sequences})
and by $\hat{L}_A$ (via the special case of (3) above), we are reduced to proving the claim for the $B_i$.
But, $T^\bullet B_i \wequi T^\bullet \hat L_A B_i$.
Now note that $\hat L_A B_i \wequi \Omega^\infty \Sigma^k S^{-1} \pi_k(B_i)$ 
for some $k$. But $S^{-1} \pi_k(B_i)$ is a $(S^{-1}\Z = \pi_0(A))$-module,
and hence also admits the structure of an $A$-module.
Therefore, $\hat L_A B_i$ is in the image of $R$,
and thus the cosimplicial object is split by \cref{lem:adjunction:split}.

The general case of (3) is now an immediate consequence of \cref{thm:fiber-sequences} ($\tau_{<\infty} T^\bullet$ preserves fiber sequences), and the general case of (4) follows from \cref{prop:Tn-conn} ($T^\bullet$ preserves connectivity).
\end{proof}

\begin{proof}[Proof of \cref{thm:BK-spaces} when $\pi_0 A = \Z/n$.]
We again first prove (2).
To begin with, let $E \in \Sp_{\ge 2}$, and consider the fiber sequence $F \to E \to L_A^{st} E$.
Then $F$ is connected and $0 = L_A^{st} F = F_n^\comp$, so that multiplication by $n$ is an equivalence on $F$.
As before $\Omega^\infty L_A^{st} E$ is $L$-local, and so it suffices to prove that $\Omega^\infty E \to \Omega^\infty L_A^{st} E$ is an $L$-equivalence.
Writing $L_A^{st} E$ as the cofiber of $F \to E$, rewriting the cofiber as a geometric realization, using semiadditivity of $\Sp$ and the fact that $\Omega^\infty$ preserves sifted colimits, we find that $\Omega^\infty L_A^{st} E \wequi \Omega^\infty(E) \sslash \Omega^\infty(F)$.\NB{details? kind of a version of \cref{lem:bar-construction-stable}}
Hence, \[ L_+ \Omega^\infty L_A^{st} E \wequi L_+ \Omega^\infty(E) \sslash L_+ \Omega^\infty(F). \]
We have seen in the previous proof that multiplication by $n$ is invertible on $\Sigma^\infty \Omega^\infty F$, and thus $L \Omega^\infty F = 0$.
It follows that $L_+ \Omega^\infty F \wequi A$ and so $L_+ \Omega^\infty L_A^{st} E \wequi L_+ \Omega^\infty(E)$, as needed.
Now suppose that $E \in \Sp_{\ge 1}$.
Applying \cref{thm:locn} to the fiber sequence $\Omega^\infty E \to * \to \Omega^\infty \Sigma E$
(note that we already know that $\hat{L} \Omega^\infty \Sigma E \cong \Omega^\infty L^{st}_A E$ is simply connected) we deduce that \[ \hat L_A \Omega^\infty E \wequi \Omega \hat L_A \Omega^\infty \Sigma E \wequi \Omega \Omega^\infty L_A^{st} \Sigma E \wequi \Omega^\infty L_A^{st} E, \] as required.

As in the previous proof, this implies the special case of (3) when $B$ is an Eilenberg--Mac Lane space.
From this we deduce by induction that if $X$ is pointed, $d$-connected, nilpotent and truncated, then $\hat L_A X$ satisfies the same assumptions.

Using the exact same argument as in the previous proof yields the special case of (4).

Next we prove (1) in the special case where $X$ is a simply connected Eilenberg--Mac Lane space.
Denote the functor $\Tot T^\bullet$ by $F$.
Let $B$ be an abelian group.
For $i \ge 0$ we have $\Omega FK(B, i+3) \wequi FK(B,i+2)$, since $F$ preserves fiber sequences (\cref{thm:fiber-sequences}).
We thus obtain an $\Omega$-spectrum $G \coloneqq (FK(B,i+2))_{i \ge 0}$ as well as a map of $\Omega$-spectra $f\colon \Sigma^2 HB \to G$ (inducing in each level the map $K(B,i+2) \to FK(B,i+2)$).
We claim that the fiber of multiplication by $n$ on $f$ is an equivalence.
Indeed, by \cref{thm:fiber-sequences} again, the functor $F$ preserves the fiber sequences $F_i \to K(B, i+2) \xrightarrow{n} K(B,i+2)$ and $F_i \to K(B/n, i+1) \to K(B[n], i+3)$.
But both $B[n]$ and $B/n$ are $A$-modules, whence $FK(B/n,i+1) \wequi K(B/n,i+1)$ and $FK(B[n], i+3) \wequi K(B[n],i+3)$ by \cref{lem:adjunction:split}, 
from which we deduce $FF_i \wequi F_i$, proving the claim.
Consequently, the map $\Sigma^2 HB \to G$ is an $n$-equivalence, so an $L_A^{st}$-equivalence, whence by (2) we deduce that $\hat L_A K(B,i+2) \wequi \hat L_A FK(B,i+2)$.
Since $FK(B,i+2)$ is clearly $L$-local, we have proved $FK(B,i+2) \wequi \hat L_AK(B,i+2)$, for any $i \ge 0$.

Finally, we prove the general case of (1).
Writing $X = \lim_i X_i$ with $X_{i+1} \to X_i \to B_i$ and $B_i$ simply connected Eilenberg--Mac Lane spaces of connectivity tending to $\infty$, as in the previous proof we find that $FX \wequi \lim_i FX_i$, and there are fiber sequences $FX_{i+1} \to FX_i \to FB_i$.
Since the same holds for $\hat L_A$ in place of $F$ by the special cases of (3) and (4) we established, showing that $FX \wequi \hat L_A X$ reduces to showing that $FB_i \wequi \hat L_A B_i$.
This we treat just before.

The general cases of (3) and (4) follow as in the previous proof.
\end{proof}

\subsection{Sheaves} \label{subsec:sheaves}
In this subsection we show the analog of \cref{thm:BK-spaces} for (nice enough) $\infty$-topoi.
Let $\scr X$ be an $\infty$-topos which admits a locally finite-dimensional cover \cite[Definition 5.1]{mattis2024unstablearithmeticfracturesquares}.
In particular, $\scr X$ is Postnikov-complete \cite[Lemma 5.3]{mattis2024unstablearithmeticfracturesquares} and has enough points (this is part of the definition).

\begin{example} \label{rmk:locally-finite-dimensional-examples}
  There are two main examples of such $\scr X$:
  Any $\infty$-topos locally of homotopy dimension $\le N$ (e.g., a presheaf $\infty$-topos) 
  admits a locally finite-dimensional cover \cite[Example A.1]{mattis2024unstablearithmeticfracturesquares},
  and the Zariski, Nisnevich and étale $\infty$-topoi on smooth $X$-schemes, 
  where $X$ is some nice enough base scheme,
  cf.\ \cite[Proposition A.3]{mattis2024unstablearithmeticfracturesquares} 
  and \cite[Proposition 2.25]{mattis2025etale}.
  The existence of a locally finite-dimensional cover is a technical assumption 
  that guarantees that certain homological localization functors 
  (such as unstable $p$-completion) commute with the limit along (weak) Postnikov towers.
\end{example}
Let $A \in \CAlg(\Sp(\scr X))_{\ge 0}$ be a connective sheaf of commutative ring spectra,
such that $\ul\pi_0(A)$ is a constant sheaf of commutative rings 
with values either a subring of $\Q$ or a quotient of $\Z$.
As in the case of spaces we get $\ul\pi_0(A) \otimes_{\ul\Z} \ul\pi_0(A) \wequi \ul\pi_0(A)$.
Set $\scr D = \Mod_A(\Sp(\scr X))$, $\scr D^{\sld}_{\ge 0} = \scr D_{\ge 0}$ 
and $L_+ = A \otimes \Sigma^\infty_+$. 
Therefore, we also get $\scr X_{*,\ge n}^\sld \supseteq \scr X_{*,\ge n}$ for all $n \ge 0$.
Since $\scr X$ is Postnikov complete, it follows that the t-structure defined by $\scr D_{\ge 0}$ is left complete.
\begin{lemma} \label{lem:sheaves:axioms}
  The axioms \axiomref{C}, \axiomref{M} and \axiomref{S} hold for the adjunction $L \dashv R$.
\end{lemma}
\begin{proof}
  The functor $R \colon \scr D_{\ge 0}^\sld \to \scr X$ is given by the composition $\Mod_A(\Sp(\scr X)_{\ge 0}) \to \Sp(\scr X)_{\ge 0} \to \scr X$;
  for the axioms \axiomref{C} and \axiomref{S3} it therefore suffices to show that both functors 
  preserve sifted colimits and are conservative. For the first this is \cite[Proposition 4.8.5.8 (4) and Corollary 4.2.3.7 (2)]{higheralgebra}. For the second, since we have enough points, 
  both statements can be checked on stalks, where the statements reduce to \cite[Corollary 5.2.6.27]{higheralgebra}.

  Axiom \axiomref{M} follows since already $LX \in \scr D_{\ge 0}^\sld$ for all $X \in \scr X$,
  and $\scr D_{\ge 0}^\sld$ is stable under tensor products. The same argument shows axiom \axiomref{S1}.

  For axiom \axiomref{S2}, let $E \in \scr D_{\ge 1}^\sld$.
  By the Hurewicz theorem we have $\ul\pi_1(\Sigma^\infty RE) \wequi \ul\pi_1(E)$ and thus \[ \ul\pi_1(LRE) \wequi \ul\pi_1(E) \otimes_{\ul\Z} \ul\pi_0(A) \wequi \ul\pi_1(E) \otimes_{\ul\pi_0(A)} \ul\pi_0(A) \otimes_{\ul\Z} \ul\pi_0(A) \wequi \ul\pi_1(E). \qedhere \]
\end{proof}

We can again identify the localization functors:
\begin{lemma} \label{lem:sheaves:functor-id}
Let $E \in \Sp(\scr X)_{\ge 0}$ and $X \in \scr X_*$ be nilpotent.
If $\ul\pi_0(A) = S^{-1}\ul\Z \subset \ul\Q$, then
    \begin{enumerate}
      \item $L_A^{st} E \wequi S^{-1} E$, and
      \item $\hat{L}_A X\wequi L_S X$, where $L_S$ is the unstable $S$-localization functor, cf.\ \cite[§3]{mattis2024unstablearithmeticfracturesquares}.
    \end{enumerate}

Analogously, if $\ul\pi_0(A) = \ul\Z/n$, then
    \begin{enumerate}
      \item $L_A^{st} E \wequi E^\comp_n = \lim_k E/n^k = \prod_{p|n} E^\comp_p$, and
      \item $\hat{L}_A X \wequi L_n X \wequi \prod_{p|n} L_p X$, the product over the unstable $p$-completion functors, cf.\ \cite[§3]{mattis2024unstablepcompletionmotivichomotopy}.
    \end{enumerate}
\end{lemma}
\begin{proof}
  For the stable identifications see \cref{lem:stable-bousfield}. 
  We thus identify the unstable localization functors.

  Suppose that $\ul\pi_0(A) = S^{-1} \ul\Z$.
  We first construct a natural transformation $L_S \to \hat{L}_A$, and then show that 
  it is an equivalence on nilpotent sheaves.
  By definition, $L_S$ is the Bousfield localization at morphisms $f_{U,n,k} \colon U \times (S^n \wedge S^1) \xrightarrow{\id \times \id \wedge k} U \times (S^n \wedge S^1)$
  where $n \ge 0$, $U \in \scr X_*$, $k \in S$ and $k \colon S^1 \to S^1$ is multiplication by $k$.
  Hence, by the universal property, it suffices to show that $\hat{L}_A$ inverts those morphisms,
  i.e., that $L$ inverts those morphisms.
  Using the stable splitting, \cref{lem:stable-splitting}, we reduce to showing that $A \otimes \Sigma^\infty(S^n \wedge S^1 \xrightarrow{\id \wedge k} S^n \wedge S^1)$ 
  is an equivalence. This map is equivalent to $(A \xrightarrow{k} A) \otimes \Sigma^{n+1} \1$,
  which is an equivalence since $k \in S$.
  Thus, we get a natural transformation $\phi \colon L_S \to \hat{L}_A$,
  such that $\phi_X$ is a $L$-equivalence for all $X \in \scr X_*$.
  Let $X \in \scr X_*$ be nilpotent.
  In order to see that $\phi_X$ is an equivalence,
  it thus suffices to show that $L_S X$ is $L$-local.
  Since $X$ is nilpotent, working through the layers of a principalized Postnikov tower using \cite[Lemma A.15]{mattis2024unstablepcompletionmotivichomotopy} as well as \cite[Lemma 3.13, Lemma 3.18 and Proposition 6.9]{mattis2024unstablearithmeticfracturesquares}, we see that $L_S X$ lies in the closure under limits of $\Omega^\infty S^{-1} \Sp(\scr X)_{\ge 2}$.
  We conclude since the class of $L$-local objects is closed under limits, $\Omega^\infty$ converts $L_A^{st}$-local objects into $L$-local objects, and $\Omega^\infty S^{-1} \Sp(\scr X)_{\ge 2} = \Omega^\infty L_A^{st} \Sp(\scr X)_{\ge 2}$ (by (1)).
  
  Suppose now that $\ul\pi_0(A) = \ul\Z/n$. 
  Let $f \colon X \to Y \in \scr X$.
  By the stable identification, we know that $f$ 
  is an $L$-equivalence if and only if 
  $f$ is an $n$-equivalence in the sense of \cref{def:unstable-n:def-n-equiv}.
  Thus, we see that $\hat{L}_A \wequi L_n$,
  and the claim thus follows from \cref{lem:unstable-n:main-thm}.
\end{proof}

We need the following fact about highly connected towers:
\begin{lemma} \label{lem:sheaves:highly-connected}
  Let $X \in \scr X_*$ be a sheaf, and
  let $(X_i)_i$ be a highly connected tower in $\scr X_*$ (in the sense of \cite[Definition 6.1]{mattis2024unstablearithmeticfracturesquares},
  i.e., the connectivity of $X_{i+1} \to X_i$ tends to $\infty$)
  under $X$.
  Then $X \to \lim_i X_i$ is an equivalence if and only if 
  the connectivity of $X \to X_i$ tends to $\infty$.
\end{lemma}
\begin{proof}
  Assume first that $X \wequi \lim_i X_i$.
  We want to argue as in \cite[Lemma 2.4]{mattis2025etale} that the connectivity 
  of $X \to X_i$ tends to $\infty$. Unfortunately, the reference requires $\scr X$ 
  to be locally of homotopy dimension $\le N$ for some $N$.
  The same proof works in our case, but using \cite[Corollary 6.6]{mattis2024unstablearithmeticfracturesquares}
  instead of \cite[Corollary 6.5]{mattis2024unstablearithmeticfracturesquares}.

  Suppose on the other hand that the connectivity of $X \to X_i$ tends to $\infty$.
  Write $Y \coloneqq \lim_i X_i$, so that there is a canonical map $X \to Y$.
  By the above, the connectivity of $Y \to X_i$ tends to $\infty$,
  hence, by two-out-of-three, the connectivity of $X \to Y$ tends to $\infty$.
  Since this is independent of $i$, we see that $X \to Y$ is $\infty$-connective,
  hence we conclude by hypercompleteness.
\end{proof}

\begin{theorem} \label{thm:sheaves:main-thm}
Let $\scr X$ be an $\infty$-topos which admits a locally finite-dimensional cover, and let $A \in \CAlg(\Sp(\scr X)_{\ge 0})$ with $\ul\pi_0 A$ either a constant subsheaf of $\Q$ or a constant quotient sheaf of $\Z$.
  \begin{enumerate}
    \item Let $X \in \scr X_*$ be nilpotent. The canonical map $\hat{L}_A X \to \tau_{\ge 1}\Tot T_A^\bullet X$ is an equivalence.
      In fact, if $\ul\pi_0(A) \subseteq \ul\Q$, then already the canonical map $\tau_{<\infty} \hat{L}_A X \to \tau_{< \infty} T^\bullet X$
      is an equivalence.
    \item For $E \in \Sp(\scr X)_{\ge 1}$ the canonical map $\hat{L}_A \Omega^\infty E \to \tau_{\ge 1} \Omega^\infty L^{st}_A E$
      is an equivalence. In fact, if $\ul\pi_0(A) \subseteq \ul\Q$, then already the canonical map $\hat{L}_A \Omega^\infty E \to \Omega^\infty L^{st}_A E$
      is an equivalence.
    \item Let $F \to E \to B$ be a fiber sequence of pointed nilpotent sheaves.
      Then the canonical map $\hat{L}_A F \to \tau_{\ge 1} \fib{(\hat{L}_A E \to \hat{L}_A B)}$ is an equivalence.
      In fact, if $\ul\pi_0(A) \subseteq \ul\Q$, then already $\hat{L}_A F \to \fib{(\hat{L}_A E \to \hat{L}_A B)}$ is an equivalence.
    \item Let $X_\bullet$ be a tower of pointed, nilpotent sheaves with $X = \lim_i X_i$.
      Assume that the tower is locally highly connected subordinate to some cover $\scr U$ (in the sense of \cite[Definition 6.1]{mattis2024unstablearithmeticfracturesquares}).
      Then the canonical map $\hat{L}_A X \to \lim_i \hat{L}_A X_i$ is an equivalence.
  \end{enumerate}
\end{theorem}
\begin{proof}[Proof of \cref{thm:sheaves:main-thm} if $\ul\pi_0(A) = S^{-1}\ul\Z$]
  Under the identifications of $L_A^{st}$ and $\hat{L}_A$ from \cref{lem:sheaves:functor-id},
  (2), (3) and (4) follow from \cite[Lemma 3.18, Lemma 3.13 and Proposition 6.9]{mattis2024unstablearithmeticfracturesquares},
  respectively.

  Now we prove (1), so let $X \in \scr X_*$ be nilpotent. Write $F \coloneqq T^\bullet_A$.
  We can present $X \wequi \lim_i X_i$ with fiber sequences $X_{i+1} \to X_i \to B_i$ where $B_i$ is an Eilenberg--Mac Lane space and the connectivity of the $B_i$ tends to infinity,
  cf.\ \cite[Lemma A.15]{mattis2024unstablepcompletionmotivichomotopy}.
  By (3) and (4), the functor $\hat{L}_A$ preserves these fiber sequences and the limit appearing in the tower,
  and similarly $F$ preserves these fiber sequences by \cref{thm:fiber-sequences}.
  Noting that $(X_i)_i$ is a highly connected tower (under $X$) we deduce from  \cref{lem:sheaves:highly-connected} and  \cref{prop:Tn-conn} that the connectivity of $FX \to FX_i$ tends to infinity with $i$, even uniformly in the levels of the pro-object.
  By definition of $\tau_{<\infty}$, it follows that $\tau_{<\infty} FX \wequi \tau_{<\infty} \lim_i FX_i$.
  We have thus reduced to the case where $X = \Omega^\infty \Sigma^k M$ for some sheaf of abelian groups $M$ and $k \ge 2$.
  Since both sides invert $L$-equivalences, we can replace $X$ by $\hat{L}_A X$,
  and we have to see that $X \to \tau_{<\infty} FX$ is an equivalence.
  By (2) we know $X = \hat{L}_A \Omega^\infty \Sigma^k M \wequi \Omega^\infty \Sigma^k S^{-1}M$.
  The result follows from \cref{lem:adjunction:split}, since $S^{-1}M$ is an $(S^{-1}\ul\Z=\ul\pi_0A)$-module, whence an $A$-module.
\end{proof}
\begin{proof}[Proof of \cref{thm:sheaves:main-thm} if $\ul\pi_0(A) = \ul\Z/n$]
  We will use without mention the identifications of $L_A^{st}$ and $\hat{L}_A$ from \cref{lem:sheaves:functor-id}.
  Then (2) is \cref{lem:unstable-n:infinite-loop-space},
  (3) is \cref{lem:unstable-n:fiber-sequence},
  and (4) is \cref{lem:unstable-n:tower}.

  We now establish (1).
  Since $\tau_{\ge 1}\colon \scr X_* \to \scr X_{*,\ge 1}$ is a right adjoint it preserves limits.
  Similarly, the functor $\hat{L}_A$ restricts to $\scr X_{*,\ge 1}$ (this follows from \cite[Lemma 3.12]{mattis2024unstablepcompletionmotivichomotopy}), and the restricted functor preserves appropriate fiber sequences (by \cref{lem:unstable-n:fiber-sequence}).
  Thus, as in the proof for $\ul\pi_0(A) = S^{-1}\ul\pi_0(\1)$, we reduce to the case $X = \Omega^\infty \Sigma^k M$ an Eilenberg--Mac Lane sheaf, with $k \ge 2$.
  Write $F \coloneqq \tau_{\ge 1}\Tot T_A^\bullet = \tau_{\ge 1}\mat \tau_{<\infty} T_A^\bullet$.
  Then \[ G(M) := (F\Omega^\infty \Sigma^k M,  F \Omega^\infty \Sigma^{k+1} M, \dots) \in \Sp(\scr X_{*,\ge1}) \wequi \Sp(\scr X), \] since $F$ preserves loops as above.
  The functor $G$ also preserves fiber sequences of $2$-connective spectra.
  There is a canonical map $\Sigma^k M \to G(M)$, which is an equivalence whenever $M$ is a finite extension of $A$-modules, essentially by \cref{lem:adjunction:split}.
  It follows that $\Sigma^k M \to G(M)$ is an $n$-equivalence, i.e., induces an equivalence on fibers of multiplication by $n$.
  Since $\Sigma^k M$ is connected, this implies that also $\Sigma^k M \to \tau_{\ge 1}G(M)$ is an $n$-equivalence
  (this follows easily from e.g.\ \cite[Lemma 2.9 and Corollary 2.11]{mattis2024unstablepcompletionmotivichomotopy};
  note that in the proof there it is never used that $p$ is a prime).
  Hence, using \cref{lem:unstable-n:infinite-loop-space}, we see that \[ \Omega^\infty \Sigma^k M \to \Omega^\infty \tau_{\ge 1}G(M) \wequi F(\Omega^\infty \Sigma^k M) \] is also an $n$-equivalence.
  Since the target is $n$-complete by construction, this concludes the proof.
\end{proof}

\subsection{Equivariant spaces} \label{subsec:equivariant}
Let $G$ be a finite group.
We will now apply our machinery to the study of $G$-equivariant spaces, whose $\infty$-category we denote by $\Spc(\B G)$.
By Elmendorf's theorem \cite[Theorem 1]{Elmendorf:SytemsofFixedPoints} this is equivalent to an $\infty$-category of presheaves \[ \Spc(\B G) \wequi \PSh_\Sigma(\Fin_G). \]
Here $\Fin_G$ denotes the $1$-category of finite $G$-sets.
In particular, $\Spc(\B G)$ is a topos locally of homotopy dimension $0$, so various results of the previous sections apply.

An additional complication which arises is that the appropriate notion of stabilization in this situation is not the naive one, i.e., $\SH^{S^1}(\B G) \coloneqq \Sp(\Spc(\B G))$ only plays a minor role in the theory.
Instead, the category of \emph{genuine $G$-spectra} is obtained by inverting the regular representation sphere $S^G \in \Spc(\B G)_*$, (see \cite[§I.2]{LewisMaySteinbergerMcClure:EquivariantHomotopyTheory} or \cite[Definition 4.1]{Cnossen:TwistedAmbidexterityinEHT} for an $\infty$-categorical definition), \[ \SH(\B G) = \Spc(\B G)_*[(S^G)^{-1}]; \] this fits into the usual adjunction \[ \Sigma^\infty \colon \Spc(\B G)_* \adj \SH(\B G) \noloc \Omega^\infty. \]
A helpful alternative picture is the description as \emph{spectral Mackey functors} \cite{Guillou_2024,nardin2016parametrizedhighercategorytheory}, i.e., \[ \SH(\B G) \wequi \Fun^\times(\Span(\Fin_G)^\mathrm{op}, \Sp), \] where $\Span(\Fin_G)$ denotes the $(2,1)$-category of spans in finite $G$-sets \cite[Appendix C]{bachmann2021norms}.
The stabilization functor $\Spc(\B G)\to \SH(\B G)$ factors over $\SH^{S^1}(\B G)$, cf.\ \cite[Corollary 4.24]{robalo2013noncommutativemotivesiuniversal}.
This in turn yields an adjunction $\sigma^\infty \colon \SH^{S^1}(B G) \adj \SH(B G) \noloc \omega^\infty$.
Alternatively the adjunction is obtained from the canonical functor $\Fin_G \to \Span(\Fin_G)$ by passing to spectral presheaves via left Kan extension and restriction.
In particular, the left adjoint $\sigma^\infty$ preserves compactly generating families by \cite[Proposition 5.5.8.10(6)]{highertopoi} and construction, and hence $\omega^\infty$ is conservative and preserves all limits and colimits.

The above spectral Mackey functor description makes it clear, for example, that $\SH(\B G)$ has a $t$-structure with heart the abelian category of Mackey functors $\Fun^\times(\Span(\Fin_G)^\mathrm{op}, \mathrm{Ab})$.
Given $X \in \Spc(\B G)_*$ we write $\ul\pi_i(X) \in \Fun^\times(\Fin_G^\mathrm{op}, \mathrm{Set})$ for the zero-truncation of $\Omega^i X$, i.e., \[ \ul\pi_i(X)(T) \wequi [T_+ \wedge S^i, X]. \]
Similarly, given $E \in \SH(\B G)$ we denote by $\ul\pi_i(X) \in \Fun^\times(\Span(\Fin_G)^\mathrm{op}, \mathrm{Ab})$ the homotopy objects in the $t$-structure; one then has \[ \ul\pi_i E|_{\Fin_G} \wequi \ul\pi_0 \Omega^\infty \Sigma^{-i} E. \]

In order to apply our machinery in this situation, we need to construct a $t$-structure on $\SH(\B G)$ such that if $E$ is $1$-connective in the $t$-structure, then $\Sigma^\infty \Omega^\infty E \to E$ induces an isomorphism on the first homotopy object.
Note that the above $t$-structure (with homotopy objects $\ul\pi_i E$) will not do (unless $G=1$): by the tom Dieck splitting theorem \cite[Satz 2]{tomDieck:Orbittypen} we see that for $E \in \SH(\B G)$ with $\ul\pi_i E = 0$ for $i \le 0$ we have 
\[ \ul\pi_1(\Sigma^\infty \Omega^\infty E)(*) \wequi \bigoplus_{(H) \subset G} \ul\pi_1(E)(G/H)_{WH}. \]
Here, $WH \coloneqq N_G H/H$ denotes the Weyl group.
This cannot agree with $\ul\pi_1(E)(*)$ unless $\ul\pi_1(E)(G/H) = 0$ for all proper subgroups $H$.
Inspired by this, we make the following definition.

\begin{definition} \label{def:equivariant:n-G-connective}
For $H \le G$ we define $l(H)= \max\{r|H< H_1 < \cdots < H_r = G\}$,
by convention $l(G) = 0$.
Let $n \ge 0$.
We define $X \in \Spc(\B G)_*$ to be \emph{$(n,G)$-connective} if for every $H \le G$ and $i < n + l(H)$ we have $\ul\pi_i(X)(G/H) = *$.
Similarly, for $n \in \Z$ we call $E \in \SH(\B G)$ \emph{$(n,G)$-connective} if $\Omega^\infty \Sigma^i E$ is $(n+i, G)$-connective for all $i \ge -n$.
\end{definition}

\begin{remark}
In other words, we require that $\Map_*(G/H_+, X)$ is $(n+l(H))$-connective, where $l(H)$ is the maximal length of a proper chain of subgroups of $G$ starting at $H$.
Note that $l(H)$ is bounded by the number of prime factors (with multiplicity) of $[G:H]$, and so in particular by $[G:H]$ itself.
\end{remark}

\begin{example}
Since $\Map(S^0, S^G)$ is $1$-connective ($S^G$ decomposes as $S^1$ smash the reduced regular representation sphere) and $S^G|_{\Spc(BH)} \wequi (S^H)^{\wedge [G:H]}$ we see that $S^G$ is $(1,G)$-connective.
More generally for $X \in \Spc(\B G)_*$, and $n \ge 1$ $(S^G)^{\wedge n} \wedge X$ is $(n,G)$-connective.
\end{example}

\begin{lemma} \label{lem:equivariant:connectivity}
Let $X \in \Spc(\B G)_*$ and $E \in \SH(\B G)$ be $(n,G)$-connective ($n \ge 0$ for $X$).
\begin{enumerate}
\item $\Sigma^\infty X$ is $(n,G)$-connective.
\item $\Omega^\infty E$ is $(n,G)$-connective (provided $n \ge 0$).
\item $\tau_{\ge i} X, \tau_{\le i} X$ are $(n,G)$-connective for $i \ge 0$.
\item $\tau_{\ge i} E, \tau_{\le i} E$ are $(n,G)$-connective for $i \in \Z$.
\item Assuming $n \ge 1$, the fiber of $\Sigma^\infty \Omega^\infty E \to E$ is $(n+1,G)$-connective.
\end{enumerate}
\end{lemma}
\begin{proof}
(2) follows immediately from the definition, and (3) and (4) are clear since $\tau_{\ge i}, \tau_{\le i}$ commute with fixed points (i.e. evaluating at $G/H$).

Let us prove (1).
Let $X \in \Spc(\B G)$ be $(n,G)$-connective and $H \le G$.
For every $i < n+l(H)$ there is an equivalence $\ul \pi_i(\Sigma^\infty X)(G/H) \wequi \bigoplus_{(H') \le H} \ul\pi_i(\Sigma^\infty(X(G/H')_{WH'}))$ by the tom Dieck splitting.
Since by assumption $X(G/H')$ is $(n+l(H'))$-connective and $\Sigma^\infty\colon \Spc \to \Sp$ preserves connectivity, (1) follows.

Now we prove (5).
Write $F \coloneqq \fib(\Sigma^\infty \Omega^\infty E \to E)$.
As $\Sigma^\infty \Omega^\infty E$ is $(n,G)$-connective by (1) and (2), it follows that $F$ is $(n-1,G)$-connective.
It is left to show, that the morphisms $\ul\pi_{i+l(H)}(\Sigma^\infty \Omega^\infty E)(G/H) \to \ul\pi_{i+l(H)}(E)(G/H)$ are isomorphisms for all $H \le G$ and $i=n$, and epimorphisms for $i=n+1$.
Since $n \ge 0$, we may as well check this after applying $\Omega^\infty$.
But then $\Sigma^\infty \Omega^\infty E \to E$ splits, so the surjectivity (at $i=n+1$) is clear.
Finally, by the tom Dieck splitting, we have that \[ \ul\pi_{n+l(H)}(\Sigma^\infty \Omega^\infty E)(G/H) \wequi \ul\pi_{n+l(H)}(E)(G/H) \oplus \bigoplus_{(H') < H}\ul\pi_{n+l(H)}(E)(G/H')_{WH'}. \]
Now for $H' < H$ we have $l(H') > l(H)$, and so $\ul\pi_{n+l(H)}(E)(G/H')_{WH'} = 0$.
This concludes the proof.
\end{proof}

Now let $A \in \CAlg(\SH(\B G)_{\ge 0})$ such that $\ul\pi_0 (A)$ is idempotent.
Furthermore, set $\scr X = \Spc(\B G)$, $\scr D = \Mod_A(\SH(\B G))$, $L_+ = A \otimes \Sigma^\infty_+$ and \[ \scr D^\sld_{\ge 0} = \{E \in \Mod_A(\SH(\B G)) \mid \text{ $E$ is $(0,G)$-connective}\}. \]
Using \cref{lem:equivariant:connectivity}(1), we have for all $n \ge $ that
\[\Spc(\B G)_{*,\ge n}^\sld \supseteq \{X \in \Spc(\B G)_* \mid \text{ $X$ is $(n,G)$-connective}\}. \]
Let us observe the following about the interaction of the two $t$-structures.
\begin{lemma} \label{lemm:2-t-structures}
Let $E \in \SH(\B G)_{\ge 0}$ and $F \in \SH(\B G)^\sld_{\ge 0}$.
Then $E \otimes F \in \SH(\B G)^\sld_{\ge 0}$.
Moreover, for $H \subset G$ we have \[ \pi_{l(H)}(E \otimes F)(G/H) \wequi \pi_0(\Phi^H E) \otimes \pi_{l(H)}(F(G/H)). \]
\end{lemma}
\begin{proof}
The first statement being stable under colimits in $E$, we may assume $E = \Sigma^\infty_+ G/H$.
In this case $(\Sigma^\infty G/H_+ \otimes F)(G/H') \wequi F(G/H \times G/H')$, the orbit $\Sigma^\infty_+ G/H$ being self-dual.
Now $G/H \times G/H'$ decomposes into a finite sum of $G/H''$, for $H''$ subconjugate to $H'$.
Thus, $l(H'') \ge l(H')$ and so the connectivity of $F(G/H \times G/H')$ is at least as high as the connectivity of $F(G/H')$, as was to be shown.

For the final statement, we may assume $H=G$.
Let $L_\Phi\colon \SH(\B G) \to \SH(\B G)$ denote the localization annihilating $\Sigma^\infty_+ G/H$ for all proper subgroups $H$.
This is a smashing localization, $(\ph)^G$ induces an equivalence $L_\Phi \SH(\B G) \wequi \Sp$, and $\Phi^G(\ph) = L_\Phi(\ph)^G$, essentially by definition, see \cite[Corollary II.9.6]{LewisMaySteinbergerMcClure:EquivariantHomotopyTheory} and \cite[Theorem 6.11]{MathewNaumannNoel:Nilpotence-equivariant-descent}.
Thus, we find \[ \pi_0((E \otimes F)(*)) \wequi \ul\pi_0(E \otimes F)(*) \wequi \ul\pi_0(E \otimes \tau_{\le 0}F)(*) \wequi \ul\pi_0(L_\Phi(E) \otimes \tau_{\le 0}F)(*) \wequi \pi_0(\Phi^G E) \otimes \pi_0(F(*)), \] using that $\tau_{\le 0}F \wequi L_\Phi \tau_{\le 0}F$ (and $L_\Phi$ is smashing).
\end{proof}

\begin{lemma} \label{lem:equivariant:adjunction-axioms}
The axioms \axiomref{C}, \axiomref{M} and \axiomref{S} hold for the adjunction $L \dashv R$.
\end{lemma}
\begin{proof}
Since $\scr D^\sld_{\ge 0}$ is accessible and closed under colimits and extensions, it is indeed the nonnegative part of a $t$-structure \cite[Proposition 1.4.4.11]{higheralgebra}.

Consider the commutative diagram
\begin{center}
   \begin{tikzcd}
     \SH^{S^1}(\B G)_{\ge 0} \ar[d, "\Omega^\infty"']  &\SH(\B G)_{\ge 0} \ar[l, "\omega^\infty"'] \ar[dl, "\Omega^\infty"] &\scr D^\sld_{\ge 0} \ar[l] \\
     \Spc(\B G) \rlap{.}
   \end{tikzcd}
 \end{center}
The functor $R$ is given by the composition form the top right to the bottom left.
For axioms \axiomref{C} and \axiomref{S3} it therefore suffices to show that all functors in the composition preserve sifted colimits and are conservative.
The first one is monadic and preserves finite colimits, $\omega^\infty$ is conservative and preserves all colimits by the discussion above, and the claim for the vertical $\Omega^\infty$ is contained in the proof of \cref{lem:sheaves:axioms}.

Axiom \axiomref{M} is immediate from \cref{lemm:2-t-structures}.

Axiom \axiomref{S1} is immediate from \cref{lem:equivariant:connectivity}(1,2).

For \axiomref{S2}, let $E \in \scr D^{\sld}_{\ge 1}$.
Using \cref{lem:equivariant:connectivity}(5) and \cref{lemm:2-t-structures} we compute \[ \pi_{1+l(H)}(LRE)(G/H) \wequi \pi_{1+l(H)}(A \otimes E)(G/H) \wequi \pi_0(\Phi^H A) \otimes \pi_{1+l(H)}E(G/H). \]
Since $E \in \SH(\B G)^\sld_{\ge 1}$ is an $A$-module, $\pi_{1+l(H)}E(G/H)$ is a $(\Phi^H A)$-module.
But $\ul\pi_0 A$ is idempotent, whence so is $\pi_0 \Phi^H A$, and so $\pi_0(\Phi^H A) \otimes \pi_{1+l(H)}E(G/H) \wequi \pi_{1+l(H)}E(G/H)$, as needed.
\end{proof}

We will study $\hat L_A$ and $\Tot T^\bullet_A$ for various choices of $A$ in the next few subsections.
They will be related to the usual unstable $S$-periodization and $n$-completion functors.
This is based on the following.

\begin{lemma} \label{lem:equivariant:unstable-stable-S-localization}
  Let $S \subset \Z$, and let $E \in \SH(\B G)$ be $1$-connective.
  Denote by $L_S \colon \Spc(\B G) \to \Spc(\B G)$ the unstable $S$-localization functor.
  Then $L_S \Omega^{\infty} E \wequi \Omega^\infty S^{-1} E$.

  Similarly denote by $L_n \colon \Spc(\B G) \to \Spc(\B G)$ the unstable $n$-completion functor from \cref{def:unstable-n:def-n-equiv}.
  Then $L_n \Omega^{\infty} E \wequi \Omega^\infty E^{\comp}_n$.
\end{lemma}
\begin{proof}
Since $\omega^\infty \colon \SH(\B G) \to \SH^{S^1}(\B G)$ preserves limits, colimits and $1$-connective objects, we may replace $\SH(\B G)$ by $\SH^{S^1}(\B G)$ throughout.
This being the stabilization of a presheaf $\infty$-topos,
we conclude by \cite[Lemma 3.18]{mattis2024unstablearithmeticfracturesquares} and \cref{lem:unstable-n:infinite-loop-space}.
\end{proof}

Many of our arguments will rely on decomposing genuine $G$-spaces into infinite loop spaces of genuine $G$-spectra.
The usual Postnikov tower is not sufficient for this.
Indeed, if $X \in \Spc(\B G)_*$, then $\ul\pi_i (X)$ is (for $i \ge 2$) a \emph{coefficient system}, i.e., an object of $\Fun^\times(\Fin_G^\mathrm{op}, \mathrm{Ab})$.
The corresponding Eilenberg--Mac Lane space is the infinite loop space of a genuine $G$-spectrum if and only if the coefficient system extends to a Mackey functor, which is not true in general.
Luckily, we can actually build all coefficient systems out of Mackey functors, in an appropriate sense.

\begin{lemma} \label{lem:equivariant:resolution-of-coeff-sys}
  Let $G$ be a finite group.
  Every coefficient system $C \in \Fun^\times(\Fin_G^\mathrm{op}, \mathrm{Ab})$ admits a resolution by Mackey functors \[ 0 \to C \to C^0 \to C^1 \to \dots \to C^N \to 0, \] i.e.:
  \begin{itemize}
  \item Each $C^i$ is the coefficient system underlying a Mackey functor.
  \item The above sequence of coefficient systems is exact.
  \end{itemize}
  Moreover, the following hold:
  \begin{itemize}
  \item The resolution is functorial in $C$.
  \item The number $N$ is independent of $C$ (in fact $N \le |G|$).
  \item Let $H \le G$ such that whenever $K \le H$ we have $C(G/K) = 0$.
    Then also $C^i(G/H) = 0$.
  \end{itemize}
\end{lemma}
\begin{proof}
  The adjunction $\sigma^\infty \dashv \omega^\infty$ can be restricted to an adjunction on the heart $F \colon\Fun^\times(\Fin_G^\mathrm{op}, \mathrm{Ab}) \adj \Fun^\times(\Span(\Fin_G)^\mathrm{op}, \mathrm{Ab})\noloc U$, see \cite[Proposition 1.3.17(iii)]{BeilinsonFaisceauxPerverse}.
  In particular, this is the free-forgetful adjunction between coefficient systems and Mackey functors.
  For $C\in \Fun^\times(\Fin_G^\mathrm{op}, \mathrm{Ab})$ define $C^i \coloneqq UF(\mathrm{coker}(C^{i-2} \to C^{i-1}))$ (set $C^{-1} \coloneqq C, C^0 \coloneqq 0$).
  This construction is clearly functorial in $C$.
  Exactness follows by construction from the fact that for any coefficient system $D$, the canonical map $D \to UFD$ is injective.
  In fact, by the tom Dieck splitting, we have \[ UFD(G/H) \wequi \bigoplus_{(H') \le H} D(G/H')_{WH'}. \]
  This also immediately proves the last bullet point.
  The only thing that remains to be done is to show that $C^N = 0$ for $N$ large enough.
  For this, note that the tom Dieck splitting formula from above implies that if $D(G/H') = 0$ for all $H'$ properly subconjugate to $H$, then $UFD(G/H) \wequi D(G/H)$ and so $(UFD/D)(G/H) = 0$.
  The construction of $C^*$ thus implies that $C^i(G/H) = 0$ as soon as $i > l(H)$.
\end{proof}

Now let $X \in \Spc(\B G)_*$ be $(1,G)$-connective.
Note that then $X(G/H)$ is simply connected for every proper subgroup $H$.
From this it follows that $X$ is nilpotent if and only if $X(*)$ is nilpotent.
Applying a functorial principalized Postnikov tower construction sectionwise, we can build a principalized Postnikov tower for $X$.
Thus, we obtain fiber sequences \[ X_{i+1} \to X_i \to K(A_i, n_i+1) \in \Spc(\B G)_*. \]
Here each $A_i$ is a coefficient system, $n_i \ge 1$, and there are compatible maps $X \to X_i$ of connectivity tending to infinity with $i$.
Moreover, $\ul\pi_* (X_i)$ is a subquotient of $\ul\pi_*(X)$, and $A_i$ is a subquotient of $\ul\pi_{n_i} (X)$.
This implies that each $X_i$ is $(1,G)$-connective and each $K(A_i, n_i+1)$ is $(2,G)$-connective.

\begin{construction} \label{cons:equiv-postnikov}
We refer to the data above as a $(1,G)$-connective principalized Postnikov tower for $X$.
\end{construction}

\begin{remark} \label{rmk:equivariant:Mackey-coeff}
Let $A$ be a coefficient system and $n \ge 1$ such that $K(A, n)$ is $(2,G)$-connective.
Rewriting the resolution of \cref{lem:equivariant:resolution-of-coeff-sys} as a system of short exact sequences, we obtain fiber sequences $K(A_i,n) \to K(M_i,n) \to K(A_{i+1},n)$ with $A_1=A$, $M_i$ a Mackey functor and $A_i = 0$ for $i$ sufficiently large.
Moreover, each $K(A_i, n)$ and $K(M_i, n)$ are $(2,G)$-connective, by the last claim of that lemma.
\end{remark}

Combining \cref{cons:equiv-postnikov} and \cref{rmk:equivariant:Mackey-coeff} yields our desired decomposition of nilpotent $(1,G)$-connective spaces into $(2,G)$-connective infinite loop spaces.

\subsubsection{Burnside Resolutions}
In this subsection, we assume that $\ul\pi_0(A)$ is either $S^{-1} \ul\pi_0(\1)$ for $S \subset \Z\setminus 0$ or $\ul\pi_0(\1)/n$ for $n \in \N$.

We can identify the associated localization functors.
\begin{lemma} \label{lem:equivariant:functor-id}
  Let $E \in \SH(\B G)_{\ge 0}$ and $X \in \Spc(\B G)_*$ nilpotent.
  If $\ul\pi_0(A) = S^{-1}\ul\pi_0(\1)$, then
    \begin{enumerate}
      \item $L_A^{st} E \wequi S^{-1} E$, and 
      \item $\hat{L}_A X\wequi L_S X$, where $L_S$ is the unstable $S$-localization functor, cf.\ \cite[§3]{mattis2024unstablearithmeticfracturesquares}.
    \end{enumerate}
    
    Analogously, if $\ul\pi_0(A) = \ul\pi_0(\1)/n$, then
    \begin{enumerate}
      \item $L_A^{st} E \wequi E^\comp_n = \lim_k E/n^k = \prod_{p|n} E^\comp_p$, and 
      \item $\hat{L}_A X \wequi L_n X \wequi \prod_{p|n} L_p X$, the product over the unstable $p$-completion functors, cf.\ \cite[§3]{mattis2024unstablepcompletionmotivichomotopy}.
    \end{enumerate}
\end{lemma}
\begin{proof}
  The identification of $L_A^{st}$ is \cref{lem:stable-bousfield}.
The tom Dieck splitting theorem implies that $\sigma^\infty \colon \SH^{S^1}(\B G) \to \SH(\B G)$ is conservative.
Suppose first that $\ul\pi_0(A) = S^{-1}\ul\pi_0(\1)$.
Then a map $f \colon Y \to Z$ of $G$-spaces is an $A$-equivalence if and only if $\Sigma^\infty_+ f$ is an $S$-equivalence (by the stable identification), if and only if $\Sigma^\infty_{S^1}(f_+)$ is an $S$-equivalence (by conservativity).
Thus, it follows from \cref{lem:sheaves:functor-id} that for $X \in \Spc(\B G)_*$ nilpotent we have $\hat{L}_A X \wequi L_S X$.
The case $\ul\pi_0(A) = \ul\pi_0(\1)/n$ is treated similarly, additionally using \cref{lem:unstable-n:main-thm}.
\end{proof}

\begin{theorem} \label{thm:equivariant:main-thm}
  \begin{enumerate}
    \item Let $X \in \Spc(\B G)_*$ be nilpotent and $(1,G)$-connective. The canonical map $\hat{L}_A X \to \Tot T^\bullet_A X$ is an equivalence.
      In fact, if $\ul\pi_0(A) =  S^{-1}\ul\pi_0(\1)$, then already the canonical map $\tau_{<\infty} \hat{L}_A X \to \tau_{< \infty} T^\bullet_A X$
      is an equivalence.
    \item For $E \in \SH(\B G)_{\ge 1}$ the canonical map $\hat{L}_A \Omega^\infty E \to \Omega^\infty L^{st}_A E$
      is an equivalence.
    \item Let $F \to E \to B$ be a fiber sequence of pointed nilpotent $G$-spaces.
      Then the canonical map $\hat{L}_A F \to \fib{(\hat{L}_A E \to \hat{L}_A B)}$ is an equivalence.
    \item Let $X_\bullet$ be a tower of pointed, nilpotent $G$-spaces with $X = \lim_i X_i$.
      Assume that the tower is highly connected.
      Then the canonical map $\hat{L}_A X \to \lim_i \hat{L}_A X_i$ is an equivalence.
  \end{enumerate}
\end{theorem}
\begin{proof}[Proof of \cref{thm:equivariant:main-thm} if $\ul\pi_0 (A) = S^{-1} \ul\pi_0(\1)$ ]
  Under the identifications of $\hat{L}_A$ and $L^{st}_A$ from \cref{lem:equivariant:functor-id}, (3) and (4) follow from \cite[Lemma 3.13 and Proposition 6.9]{mattis2024unstablearithmeticfracturesquares}, respectively.
  Statement (2) then follows by combining \cref{lem:equivariant:unstable-stable-S-localization,lem:equivariant:functor-id}.

  We now prove (1). 
  Let $X$ be nilpotent and $(1,G)$-connective.
  Working through a principalized Postnikov tower, as in the proof of \cref{thm:sheaves:main-thm}, we can reduce to the case that $X = K(C, k)$ for a coefficient system $C$ and $k\ge 2$.
  Note that as pointed out in \cref{cons:equiv-postnikov}, the space $K(C,k)$ is $(2,G)$-connective.
  (A key point is that by \cref{cons:equiv-postnikov}, we never leave the world of $(1,G)$-connective $G$-spaces, and so our axiomatics apply.)
  Using \cref{lem:equivariant:resolution-of-coeff-sys,rmk:equivariant:Mackey-coeff,thm:fiber-sequences} (we can use the latter since in all the fiber sequences the base is $(2,G)$-connective), this reduces to the $(2,G)$-connective $G$-space $X = \Omega^\infty \Sigma^k M$, where $M$ is a Mackey functor. 
  Since both sides invert $L$-equivalences, we can replace $X$ by $\hat{L}_A X$, and we have to see that $X \to \tau_{< \infty}T^\bullet_A X$ is an equivalence.
  By (2) we know $X = \hat{L}_A \Omega^\infty \Sigma^k M \wequi \Omega^\infty \Sigma^k S^{-1}M$.
  The result follows from \cref{lem:adjunction:split}, since $S^{-1}M$ is an $(S^{-1}\ul\pi_0(\1) = \ul\pi_0(A))$-module, whence an $A$-module.
\end{proof}

\begin{proof}[Proof of \cref{thm:equivariant:main-thm} if $\ul\pi_0 (A) = \ul\pi_0(\1)/n$ ]
  We will use without mention the identifications of $L_A^{st}$ and $\hat{L}_A$ from \cref{lem:equivariant:functor-id}.
  Then (3) is \cref{lem:unstable-n:fiber-sequence},
  and (4) is \cref{lem:unstable-n:tower}.
  Moreover, (2) follows by combining \cref{lem:equivariant:unstable-stable-S-localization,lem:equivariant:functor-id}.

  We now show (1).
  As in the proof for $\ul\pi_0(A) = S^{-1} \ul\pi_0(\1)$, we reduce to the case that $X = \Omega^\infty \Sigma^k M$ is $(2,G)$-connective for a Mackey functor $M$.
  Write $F \coloneqq \Tot T^\bullet_A$.
  Then \[ P(M) := (F\Omega^\infty \Sigma^k M,  F \Omega^\infty \Sigma^{k+1} M, \dots) \in \SH^{S^1}(\B G), \] since $F$ preserves loops by \cref{thm:fiber-sequences}.
  The functor $P$ also preserves fiber sequences of $2$-connective spectra.
  There is a canonical map $\omega^\infty\Sigma^k M \to P(M)$, which is an equivalence whenever $M$ is a finite extension of $A$-modules, essentially by \cref{lem:adjunction:split}.
  It follows that $\Sigma^k M \to P(M)$ is an $n$-equivalence, i.e., induces an equivalence on fibers of multiplication by $n$.
  Hence, using \cref{lem:unstable-n:infinite-loop-space}, we see that \[ \Omega^\infty \Sigma^k M \to \Omega^\infty P(M) \wequi F(\Omega^\infty \Sigma^k M) \] is also an $n$-equivalence.
  Since the target is $n$-complete by construction, this concludes the proof.
\end{proof}

\subsubsection{$\underline{\Z}$ Resolutions: Localization}
Let $G$ be a finite group and $I$ the augmentation ideal of the Burnside ring Mackey functor $\ul{A}$, i.e., the kernel of the surjective rank map $\ul A \to \ul \Z$.
Throughout this subsection we assume that $\ul\pi_0 (A)\wequi S^{-1}(\ul\pi_0(\1)/I) \wequi S^{-1}\ul\Z$ for $S \subset \Z\setminus 0$ and $|G| \in S$.
Consider the element $e^G_e = [G/e]/|G| \in A(G)[1/|G|]$.
This is immediately checked to be an idempotent.
For $E \in S^{-1}\SH(\B G)$ we put \[ (e^G_e)^{-1} E = \colim E \xrightarrow{e^G_e} E \xrightarrow{e^G_e} \cdots. \]
We define $(1-e^G_e)^{-1} E$ similarly.
Then the canonical map $E \to (e^G_e)^{-1} E \oplus (1-e^G_e)^{-1} E$ is an equivalence, as is immediately verified on homotopy Mackey functors.

\begin{remark} \label{rmk:1+-equiv-locn}
Note that $(1-e^G_e)^{-1} E = 0$ if and only if $[G]\ul\pi_*E = 0$, i.e., $E(G/e) = 0$. \NB{$[G]$ acts on $\ul\pi_i(E)(G/e)$ by multiplication by $|G|$, so if this is zero then $E(G/e) = 0$. This implies that $[G] \ul\pi_i E = 0$ because multiplication by $[G]$ factors through restriction to $G/e$.}
\end{remark}

To identify the localization $\hat L_A$, it will be useful to consider the inclusion $i \colon \B G \to \Fin_G$ sending the unique object of the source to $G$.
This induces an adjunction \[ i^* \colon \Spc(\B G) \adj \Fun(\B G, \Spc) \noloc i_*. \]
Here $i^*X$ is just $X(G)$ viewed as a naive $G$-space, and $i_*$ turns a naive $G$-space into a genuine one by declaring the genuine fixed points to be the homotopy fixed points.
In particular, $i_*$ is fully faithful.
(The functor $i^*$ has a further left adjoint $i_!$, which does not concern us here.)
We also have \[ i^* \colon \SH(\B G) \adj \Fun(\B G, \Sp) \noloc i_*, \] with similar properties, e.g.\ \cite[Theorem II.2.7]{NikolausScholze:OnTC}.

\begin{lemma} \label{lem:equivariant:Z:localization:functor-id}
Let $E \in \SH(\B G)$ be bounded below and $X \in \Spc(\B G)_*$ be nilpotent.
\begin{enumerate}
\item We have $L_A^{st}E \wequi (e^G_e)^{-1} S^{-1}E$.
\item Alternatively, we have $L_A^{st}E \wequi i_*S^{-1} i^* E$.
\item We have $\hat L_A X \wequi i_* L_S i^* X$.
\end{enumerate}
\end{lemma}
\begin{proof}
(1)
The morphism $\ul{A} \to \ul{A}/I$ admits a section given by $1 \mapsto [G/e]$,
implying that $S^{-1}\ul{A}[1/e^G_e] \xrightarrow{\wequi} S^{-1}\ul{A}/I \wequi \ul\pi_0 (A)$.

The previous sentence shows that $e^G_e$ maps to a unit in $\ul\pi_0(A)$ and so $L_A^{st} E \wequi L_A^{st} (e^G_e)^{-1} S^{-1} E$.
It will thus be enough to show that $E' := (e^G_e)^{-1} S^{-1} E$ is $A$-local.
This follows from a Postnikov tower argument, using that $E'$ is bounded below (since $E$ is) and each $\ul\pi_i(E')$ as a module over $(e^G_e)^{-1}S^{-1} \ul A \wequi \ul\pi_0(A)$.

(2) Since $i_*$ is fully faithful, $i_*S^{-1}i^*$ is a localization functor, namely at the underlying $S$-equivalences.
By \cref{rmk:1+-equiv-locn}, so is $(e^G_e)^{-1}S^{-1}$.
The claim thus follows from (1).

(3) By (2), a map $f$ in $\Spc(\B G)$ is an $L$-equivalence if and only if $S^{-1}\Sigma^\infty_+ i^*f$ is an equivalence.
In particular, $\hat L_A = \hat L_A i_* i^*$, the functor $i_*$ being fully faithful (i.e., part of a localization).
The subcategory of objects of the form $i_*(\ph)$ is equivalent to $\Fun(\B G, \Spc)$, and under this equivalence, $\hat L_A$ on this subcategory just coincides with localization at the homological $S$-equivalences.
For nilpotent objects, this is given by $L_S$ (\cref{lem:sheaves:functor-id}).
This concludes the proof.
\end{proof}

\begin{theorem} \label{thm:equivariant:Z:loc:main-thm}
  \begin{enumerate}
    \item Let $X \in \Spc(\B G)_*$ be nilpotent and $(1,G)$-connective. The canonical map $\tau_{<\infty} \hat{L}_A X \to \tau_{< \infty} T^\bullet_A X$
      is an equivalence.
    \item For $E \in \SH(\B G)_{\ge 1}$ the canonical map $\hat{L}_A \Omega^\infty E \to \Omega^\infty L^{st}_A E$
      is an equivalence.
    \item Let $F \to E \to B$ be a fiber sequence of pointed nilpotent $(1,G)$-connective $G$-spaces.
      Then $\hat{L}_A F \wequi \fib{(\hat{L}_A E \to \hat{L}_A B)}$.
    \item Let $X_\bullet$ be a tower of pointed nilpotent $(1,G)$-connective $G$-spaces with $X = \lim_i X_i$.
      Assume that the tower is highly connected.
      Then $\hat{L}_A X \wequi \lim_i \hat{L}_A X_i$.
  \end{enumerate}
\end{theorem}
\begin{proof}
  (2), (3) and (4) are immediate from \cref{lem:equivariant:Z:localization:functor-id},
  the fact that $i^*$ and $i_*$ preserve all limits (since they are right adjoints),
  and the corresponding facts about $L_S$, cf.\ \cite[Lemma 3.18, Lemma 3.13 and Proposition 6.9]{mattis2024unstablearithmeticfracturesquares}.

  For statement (1), arguing in the usual way (see e.g. \cref{thm:equivariant:main-thm}(1)), we reduce the case $X = K(M,i)$ where $i \ge 2$ and $M$ is Mackey functor with $M \wequi i_*i^* S^{-1} M$.
  But then $M$ is a module over $S^{-1} \ul\Z$, and so the claim follows from \cref{lem:adjunction:split}.
\end{proof}

\subsubsection{$\underline{\Z}$ Resolutions: Completion}

For this subsection let $G$ be a finite $p$-group.
We continue to write $I$ for the augmentation ideal of the Burnside ring Mackey functor $\ul A$.
Assume that $\ul\pi_0 (A) \wequi \ul{A}/(I,p) \wequi \ul{\F}_p$.

\begin{lemma} \label{lem:equivariant:Z:completion:(un)stable-id}
  Let $E \in \SH(\B G)_{\ge 0}$.
  Then $L^{st}_A E \wequi E^{\comp}_p$.

  Moreover, $\hat{L}_A X \wequi L_p X$ for $X \in \Spc(\B G)_*$ nilpotent.
\end{lemma}
\begin{proof}
  To show the stable identification use that $L^{st}_A E \wequi E^{\comp}_{I,p}$ by \cref{lem:stable-bousfield}.
  Since $E^{\comp}_{I,p}$ is $p$-complete it suffices to show that $E^\comp_p$ is $\ul\F_p$-complete.
  Let $(E_{\le n})_n$ be a Postnikov tower with $E \wequi  \lim_n E_{\le n}$.
  It is enough to show that all $(E_{\le n})^\comp_p$ are $\ul\F_p$-complete as completion commutes with limits, e.g.\ \cite[Lemma 2.5]{mattis2024unstablepcompletionmotivichomotopy}.
  Using the fiber sequence $(E_{= n})^\comp_p \to (E_{\le n})^\comp_p \to (E_{< n})^\comp_p$ we reduce to the case of $\Sigma^k M$, where $M$ is a Mackey functor.
  We have that $M^\comp_p \wequi \lim_n M\sslash p^n$ and moreover a fiber sequence $M \sslash p^n \xrightarrow{p} M\sslash p^{n+1} \to M \sslash p$.
  Hence, it suffices to show that $M \sslash p$ is $\ul\F_p$-local.
  There is a fiber sequence $\Sigma (M[p]) \to M \sslash p \to M/p$.
  Therefore, we may reduce to the case where $M$ is $p$-torsion and thus a $(\ul\pi_0(\1)/p)$-module.
  It is even a $(\ul\pi_0(\1)/I^n)$-module for some $n \ge 0$ by \cite[Proposition 1.12]{Laitinen:BurnsideRingStableCohomotopyFiniteGroup} (it is here where we need $G$ to be a finite $p$-group).
  For $k \ge 0$ we obtain exact sequences $I^k M \to I^{k-1}M \to I^{k-1}M/I^{k}M$.
  Since each $I^{k-1}M/I^kM$ is an $\ul{A}/I$-module in addition to being an $\ul{A}/p$-module, it is in fact an $\ul\F_p$-module.
  Since $I^nM = 0$ it follows that $M$ is $\ul\F_p$-local.
  This shows the claim in the stable case.

  The unstable claim follows by the same argument as in \cref{lem:equivariant:functor-id}.
\end{proof}

\begin{theorem} \label{thm:equivariant:Z:completion:main-thm}
  \begin{enumerate}
    \item Let $X \in \Spc(\B G)_*$ be nilpotent and $(1,G)$-connective. The canonical map $\hat{L}_A X \to \Tot T^\bullet X$ is an equivalence.
    \item For $E \in \SH(\B G)_{\ge 1}$ the canonical map $\hat{L}_A \Omega^\infty E \to \Omega^\infty L^{st}_A E$
      is an equivalence.
    \item Let $F \to E \to B$ be a fiber sequence of pointed nilpotent $G$-spaces.
      Then the canonical map $\hat{L}_A F \to \fib{(\hat{L}_A E \to \hat{L}_A B)}$ is an equivalence.
    \item Let $X_\bullet$ be a tower of pointed, nilpotent $G$-spaces with $X = \lim_i X_i$.
      Assume that the tower is highly connected.
      Then the canonical map $\hat{L}_A X \to \lim_i \hat{L}_A X_i$ is an equivalence.
  \end{enumerate}
\end{theorem}
\begin{proof}
  (3) and (4) where already shown in \cref{thm:equivariant:main-thm}.
  For (2) combine \cref{lem:equivariant:unstable-stable-S-localization,lem:equivariant:Z:completion:(un)stable-id}.

  For (1) we repeat the arguments from \cref{thm:equivariant:main-thm}(1) in the case $\ul\pi_0A = \ul\pi_0(\1)/p$.
  We must eventually show that $\omega^\infty \Sigma^k M \to G(M)$ is a $p$-equivalence for Mackey functors $M$, and we know this if $M$ is a module over $\ul\F_p$.
  As in that proof, one reduces to the case where $M$ is $p$-torsion.
  As in the proof of \cref{lem:equivariant:Z:completion:(un)stable-id}, this implies that $M$ is a finite extension of $\ul\F_p$-modules, from which we conclude.
\end{proof}

\subsection{Motivic spaces} \label{subsec:motivic}
Let $k$ be a perfect field.
Recall the category of motivic spaces $\Spc(k)$ and motivic spectra $\SH(k)$, e.g. from \cite[\S2.2 and \S4.1]{bachmann2021norms}.
Write $\Sigma^{p+q,q} \colon \SH(k) \to \SH(k)$ for the autoequivalence 
given by tensoring with the motivic sphere $\Sigma^p \Gm^{q}$.
Pick $A \in \CAlg(\SH(k)^\veff)$ (see \cref{rmk:t-structures-SH} for $\SH(k)^\veff$) and let $\scr D \coloneqq \Mod_A(\SH(k))$.
We will have to pick a category of solid objects, and there are some choices for this: for $\nu \ge 0$ put $\scr D^{\nu\sld}_{\ge 0} := \Sigma^{\nu,\nu}\Mod_A(\SH(k)^\veff)$.
For most purposes $\nu=1$ is the only relevant case.
For all $n\ge 0$ we have $\scr X_{*,\ge n}^{\nu\sld} \supseteq \Spc(k)_{*,\ge (\nu+n,\nu)}$,
cf.\ \cref{rmk:L-pres-eff}.

Since $\Spc(k)$ is famously \emph{not} an $\infty$-topos, we cannot directly work 
with the adjunction $\Spc(k) \adj \scr D$, but have to work with $\scr X \coloneqq \Shv_\Nis(\Sm_k)$,
the $\infty$-topos of Nisnevich sheaves on (quasi-compact) smooth $k$-schemes.
This $\infty$-topos admits a locally finite-dimensional cover $\scr U$ \cite[Proposition A.3]{mattis2024unstablearithmeticfracturesquares}.
Note that by construction there is an adjunction $L_{\A^1} \colon \scr X \adj \Spc(k) \noloc \iota_{\A^1}$.
Hence, we will work with the composed adjunction $L \colon \scr X \adj \Spc(k) \adj \scr D \noloc R$.
We get the associated completion functor $\Tot T^\bullet_A \colon \scr X_* \to \scr X_*$ (which in fact takes values in $\Spc(k)_* \subset \scr X$).

\begin{remark}\label{rmk:L-pres-eff}
Recall the notion of $\nu$-effective motivic spaces $\Spc(k)_{\ge (\nu,\nu)}$ from \cite[\S 3]{asok2024p1stabilizationunstablemotivichomotopy},
where the notation $O(S^{\nu,\nu})$ was used.
It follows from the definition of $O(S^{d+\nu,\nu})$ in the stable 
situation together with \cite[Proposition 3.2.4]{asok2024p1stabilizationunstablemotivichomotopy} that if $X \in \Spc(k)_*$ is $\nu$-effective and $d$-connective, then $LX \in \scr D^{\nu\sld}_{\ge d}$.
\end{remark}

\begin{remark}[t-structures on $\SH(k)$] \label{rmk:t-structures-SH}
  The stable $\infty$-category $\SH(k)$ admits multiple useful t-structures.
  The \emph{homotopy t-structure} was defined by Morel \cite[Theorem 5.2.3]{morel2003introduction}, its connective and coconnective parts 
  are given by vanishing conditions on the bigraded homotopy sheaves.
  We will write $\SH(k)^\heartsuit$ for the heart of the homotopy t-structure,
  which Morel identified as the abelian category of homotopy modules \cite[Theorem 5.2.6]{morel2003introduction}.
  As usual, we will write $\ul\pi_n(-)_*$ for the homotopy objects in this t-structure. This t-structure is left-complete,
  cf.\ \cite[Corollary 2.4]{Hoyois:algebraicCobordismtoMotivicCohomology}.

  One other choice of a t-structure is the \emph{effective t-structure},
  whose connective part is given by $\SH(k)^\veff$, i.e., by those motivic spectra 
  that are both connective in the homotopy t-structure and effective \cite[§3]{bachmann2017generalized}.
  This t-structure restricts to a t-structure on $\SH(k)^\eff$, the subcategory of effective motivic spectra.
  We will write $\SH(k)^{\eff\heartsuit}$ for the heart,
  and $\ul\pi_n^\eff(-)$ for the homotopy objects in the effective t-structure.
\end{remark}

\begin{remark}[Slice filtration] \label{rmk:slices}
  For every $n \in \Z$ we write $f_n \colon \SH(k) \to \SH(k)$ for the $n$-th effective cover functor,
  and $s_n \coloneqq \cof(f_{n+1} \to f_n)$ for the $n$-th slice functor, cf.\ \cite[§2]{voevodsky3open}.
\end{remark}

We will first verify the axioms.
\begin{lemma} \label{lem:motivic:omega-infty}
  The functors 
  \begin{equation*}
    \SH(k)^\eff \xrightarrow{\omega^\infty} \SH^{S^1}(k) \hookrightarrow \Sp(\Shv_\Nis(\Sm_k))
  \end{equation*}
  preserve limits and colimits and are conservative.
\end{lemma}
\begin{proof}
  Both functors are right adjoints by construction, and hence preserve limits.
  That $\omega^\infty$ is conservative and preserves colimits was shown in \cite[Lemma 6.1 (1)]{Bachmann2020TowardsConservativity}.
  The second functor is fully faithful and hence conservative.
  It is exact and thus preserves finite colimits.
  That it also preserves filtered colimits 
  follows immediately from the fact that the 
  Nisnevich sheaves of spectra $\Sigma^{\infty}_+ U$ 
  with $U \in \Sm_k$ are compact (see, e.g., \cite[Proposition A.3(2)]{bachmann2021norms}).
\end{proof}

\begin{lemma}
  The axioms \axiomref{C}, \axiomref{M}, \axiomref{S1} and \axiomref{S3} hold for the adjunction $L \dashv R$.
  If $\nu \ge 1$ and $\ul\pi_0(A)_*$ is idempotent in $\SH(k)^{\heartsuit}$,
  then also \axiomref{S2} holds.
\end{lemma}
\begin{proof}
  For axioms \axiomref{C} and \axiomref{S3} we consider the following commutative diagram:
  \begin{center}
    \begin{tikzcd}
      \Sp(\Shv_\Nis(\Sm_k))_{\ge 0} \ar[d, "\Omega^\infty"] &\SH^{S^1}(k)_{\ge 0} \ar[d] \ar[l, hook] &\SH(k)^\veff \ar[l, "\omega^\infty"'] &\Mod_A(\SH(k)^\veff) \ar[l] \\
      \Shv_\Nis(\Sm_k) &\Spc(k) \ar[l, hook] \rlap{.}
    \end{tikzcd}
  \end{center}
  The functor $R$ is given by the composition from the top right to the bottom left.
  Hence, it suffices to show that any functor in this composition preserves 
  sifted colimits and is conservative. 
  For the forgetful functor $\Mod_A(\SH(k)^\veff) \to \SH(k)^\veff$,
  this follows from e.g.\ \cite[Proposition 4.8.5.8 (4) and Corollary 4.2.3.7 (2)]{higheralgebra}.
  See \cref{lem:motivic:omega-infty} for $\omega^\infty$ and the inclusion.
  In the proof of \cref{lem:sheaves:axioms} we showed the statement for $\Omega^\infty$.

  For axiom \axiomref{M}, it suffices to note that for any $X \in \scr X_*$, we have $LX \in \scr D^{0\sld}_{\ge 0}$ (see \cref{rmk:L-pres-eff}), and that \[ D^{a\sld}_{\ge b} \otimes D^{a'\sld}_{\ge b'} \subset D^{(a+a')\sld}_{\ge b+b'}. \]
  The same reasoning also implies axiom \axiomref{S1}, using that $\Omega^\infty$ preserves effectivity \cite[Proposition 3.2.12]{asok2024p1stabilizationunstablemotivichomotopy}.

  For \axiomref{S2}, let us first show that if $E \in \Sigma^{2,1}\SH(k)^\veff$ then $\Sigma^\infty \Omega^\infty E \to E$ induces an isomorphism on $\ul\pi_1(\ph)_*$. 
  Equivalently, using \cite[Proposition 5]{bachmann2017generalized}, the composite \[ \Sigma^{2,1}\SH(k)^{\eff\heartsuit} \to \Sigma \SH^{S^1}(k)(1)^\heartsuit \to \scr X_* \] should be fully faithful.
  The first functor is fully faithful by \cite[Theorem 6.9]{Bachmann2020TowardsConservativity}, and the second by the classical Hurewicz theorem together with \cite[Proposition 3.2.12]{asok2024p1stabilizationunstablemotivichomotopy}.
  Now let $E \in \scr D^{\nu\sld}_{\ge 1}$.
  The functor $\scr D^{0\sld} \to \SH(k)^\eff$ being $t$-exact (see e.g. \cite[Lemma 29]{bachmann-tambara}), in order to show that $LRE \to E$ induces an isomorphism on $\pi_1$ (with respect to the $t$-structure specified by $\scr D^{\nu\sld}$), it is by \cite[Proposition 4]{bachmann2017generalized} sufficient (and necessary) to show that $\ul\pi_1(LRE)_{-\nu} \wequi \ul\pi_1(E)_{-\nu}$.
  Since $\nu \ge 1$ we have (by what we first showed) that $\ul\pi_1(LRE)_* \wequi \ul\pi_1(E)_* \otimes \ul\pi_0(A)_*$.
  This implies what we want since $\ul\pi_0(A)_*$ is idempotent by assumption.
\end{proof}
In the cases discussed below, $\ul\pi_0(A)_*$ will always be either a quotient or a localization of the unit (or a combination of the two), whence idempotent,
so that the axioms hold.

\begin{corollary} \label{cor:motivic-preserve-effectivity}
Let $X \in \Spc(k)_*$ be $n$-effective and $d$-connected.
Then $T^\bullet_A(X)$ consists of $n$-effective $d$-connected spaces.
\end{corollary}
\begin{proof}
This follows from \cref{cor:connected} by taking $\nu=n$.
\end{proof}

Before stating the main theorems of this section, we need some preliminaries.

\begin{definition} \label{def:motivic:resolvable}
We call $X \in \Spc(k)_*$ \emph{resolvable} if there exists a tower $X \to X_\bullet$ as well as fiber sequences $X_{i+1} \to X_i \to \Omega^\infty B_i$, such that:
\begin{enumerate}
\item The connectivity of $X \to X_i$ tends to infinity with $i$.
\item The connectivity of $B_i$ tends to infinity with $i$.
\item $B_i \in \Sigma^{3,1} \SH(k)^\veff$.
\item $X_0 = *$.
\end{enumerate}
\end{definition}
\begin{remark} \label{rmk:motivic:resolvable}
It is proved in \cite[Construction 4.1.7, Remark 4.1.13]{asok2024p1stabilizationunstablemotivichomotopy}\todolater{add ref for the $1$-effective case once freudenthal paper gets updated} that $X \in \Spc(k)_*$ is resolvable as soon as $X$ is nilpotent and either
\begin{itemize}
\item $X$ is $2$-effective, or
\item $X$ is $1$-effective and $e$-periodic, where $e$ is the exponential characteristic of $k$.
\end{itemize}
\end{remark}

\begin{remark} \label{rmk:generalized-tower}
It follows that $X_i \in \scr X^{1\sld}_{*,\ge 1}$ and $\Omega^\infty B_i \in \scr X^{1\sld}_{*,\ge d_i}$ for some sequence $d_i$ tending to infinity with $i$.
Indeed, the second claim is (essentially) \cref{lemm:first-consequences} (2), combined with the assumption 
that the connectivity of the $B_i$ tends to infinity with $i$.
The first claim is proven inductively, 
using the second claim, \cite[Theorem 4.2.3]{asok2024p1stabilizationunstablemotivichomotopy}
and the definition of $\scr D^{1\sld}$.
\end{remark}

\begin{remark}
  If $X$ is resolvable and $n$-effective,
  one may choose the $X_i$ and $B_i$ such that they are all $n$-effective,
  again by \cite[Construction 4.1.7 and Remark 4.1.13]{asok2024p1stabilizationunstablemotivichomotopy}.
\end{remark}

The next result concerns the interaction between the $\infty$-topos-theoretic $S$-localization in $\scr X$ 
and the stable $S$-localization in $\SH(k)$.
\begin{lemma} \label{lem:motivic:unstable-stable-S-localization}
  Let $S \subset \Z$, and let $E \in \SH(k)$ be $1$-connective.
  Denote by $L_S \colon \scr X \to \scr X$ the unstable $S$-localization functor.
  Then $L_S \Omega^{\infty} E \wequi \Omega^\infty S^{-1} E$.
\end{lemma}
\begin{proof}
  Consider the commutative diagram 
  \begin{center}
    \begin{tikzcd}
      \Sp(\Shv_\Nis(\Sm_k)) \ar[d, "\Omega^\infty"] &\SH^{S^1}(k) \ar[d] \ar[l, hook] &\SH(k)^\eff \ar[l, "\omega^\infty"']\\
      \Shv_\Nis(\Sm_k) &\Spc(k) \ar[l, hook]\rlap{.}
    \end{tikzcd}
  \end{center}
  It suffices to show that the composition on the top and left commutes with $S$-localization.
  Since both $\omega^\infty$ and the inclusion
  $\SH^{S^1}(k) \hookrightarrow \Sp(\Shv_\Nis(\Sm_k))$ preserve colimits by \cref{lem:motivic:omega-infty}, 
  it follows from e.g.\ \cite[Corollary 2.7]{mattis2024unstablearithmeticfracturesquares} that they commute with $S$-localization.
  Moreover, both of these functors preserve $1$-connective objects.
  Hence, the result follows from \cite[Lemma 3.18]{mattis2024unstablearithmeticfracturesquares}.
\end{proof}
We also need a version of the last lemma for the $n$-completion functors:
\begin{lemma} \label{lem:motivic:unstable-stable-n-completion}
  Let $n \in \Z$, and let $E \in \SH(k)$ be $1$-connective.
  Denote by $L_n \colon \scr X \to \scr X$ the unstable $n$-completion functor from \cref{def:unstable-n:def-n-equiv}.
  Then $L_n \Omega^{\infty} E \wequi \tau_{\ge 1} \Omega^\infty E^{\comp}_n$.
\end{lemma}
\begin{proof}
  Since $\Omega^\infty E$ is nilpotent (as $E$ is $1$-connective),
  we have that $L_n \Omega^\infty E \wequi \prod_{\ell | n} L_\ell \Omega^\infty E$,
  cf.\ \cref{lem:unstable-n:main-thm}.
  Similarly, from \cref{lem:stable-bousfield:n-completion} we have $E^\comp_n \wequi \prod_{\ell|n} E^\comp_\ell$.
  Since $\Omega^\infty$ and $\tau_{\ge 1}$ preserve finite products, we may assume that $n = \ell$.
  Consider again the commutative diagram
  \begin{center}
    \begin{tikzcd}
      \Sp(\Shv_\Nis(\Sm_k)) \ar[d, "\Omega^\infty"] &\SH^{S^1}(k) \ar[d] \ar[l, hook] &\SH(k)^\eff \ar[l, "\omega^\infty"']\\
      \Shv_\Nis(\Sm_k) &\Spc(k) \ar[l, hook]\rlap{.}
    \end{tikzcd}
  \end{center}
  It suffices to show that the composition on the top and left commutes with $\ell$-completion,
  up to a connected cover.
  The functors on the top both preserve limits by \cref{lem:motivic:omega-infty},
  and hence commute with $\ell$-completion, see e.g.\ \cite[Lemma 2.32]{mattis2024unstablepcompletionmotivichomotopy}.
  Moreover, they preserve $1$-connective objects.
  Hence, the result follows from \cite[Lemma 3.17]{mattis2024unstablepcompletionmotivichomotopy}.
\end{proof}

We now discuss Bousfield--Kan completions.
We now have four different situations: First, we can look at those $A$ 
such that $\ul\pi_0(A)_* = S^{-1}\ul\pi_0(\1)_*$ or $\ul\pi_0(A)_* = \ul\pi_0(\1)_* / n$;
this is similar to the case of sheaves discussed above.
But these two cases do not cover two important situations:
One wants to be able to complete at $A = H\Z$, the motivic cohomology spectrum,
or $A = \mathrm{MGL}$, the algebraic cobordism spectrum.
Here, we do not have that $\ul\pi_0(A)_* = \ul\pi_0(\1)_*$,
but instead $\ul\pi_0(A)_* = \ul\pi_0(\1)_*/\eta$.
Thus, in the second half of this section, we will look at completions 
at those $A$ such that either $\ul\pi_0(A)_* = S^{-1}\ul\pi_0(\1)_*/\eta$,
or $\ul\pi_0(A)_* = \ul\pi_0(\1)_*/(\eta,n)$. This will complicate things slightly,
as it will also include an additional $\eta$-completion.
To circumvent this problem, we will restrict ourselves to fields of finite $2$-étale 
cohomological dimension, so that all very effective motivic spectra
are already $\eta$-complete.

\subsubsection{Milnor--Witt Resolutions: Localization}
In this subsection, we assume that $\ul\pi_0(A)_* \wequi S^{-1} \ul\pi_0(\1)_*$
with $S \subset \Z \setminus 0$.
Thus, $\ul\pi_0(A)_* \otimes_{\ul\pi_0(\1)_*} \ul\pi_0(A)_* \wequi \ul\pi_0(A)_*$,
and hence all axioms are satisfied.
In this situation, the Bousfield--Kan completion is related to unstable $S$-localization.
We write $L_S$ for the unstable $S$-localization functor in $\scr X$, cf.\ \cite[\S 3]{mattis2024unstablearithmeticfracturesquares}.

\begin{example}
  Possible examples for $A$ are the rings $S^{-1} \1$ 
  and $f_0 S^{-1} \ul{K}^{MW}$.
\end{example}

\begin{theorem} \label{thm:motivic:milnor-witt-localization}
  \begin{enumerate}
    \item Let $X \in \Spc(k)_*$ be nilpotent. Then $L_S X \in \Spc(k)_*$ (i.e., $L_S X$ is again $\A^1$-invariant).
    \item Let $X \in \Spc(k)_*$ be resolvable. The canonical maps $\tau_{<\infty} L_S X \to \tau_{<\infty} \hat L_A X \to \tau_{< \infty} T^\bullet X$ are equivalences.
      In particular, $L_S X \to \hat L_A X \to \Tot T^\bullet X$ are equivalences.
    \item For $E \in \SH(k)_{\ge 1}$ the canonical map $L_S \Omega^\infty E \to \Omega^\infty L^{st}_A E$
      is an equivalence.
    \item Let $F \to E \to B$ be a fiber sequence of pointed nilpotent Nisnevich sheaves.
      Then $L_S F \to L_S E \to L_S B$ is a fiber sequence.
    \item Let $X_\bullet$ be a tower of pointed nilpotent Nisnevich sheaves with $X = \lim_i X_i$.
      Assume that the tower is locally highly connected subordinate to $\scr U$, e.g.\ it is highly connected.
      Then $L_S X \wequi \lim_i L_S X_i$.
  \end{enumerate}
\end{theorem}

\begin{proof}
  Statement (1) is proven in \cite[Proposition 4.3.8 and Theorem 4.3.9]{asok2022localizationnilpotentspacesa1homotopy},
  statement (4) is \cite[Lemma 3.13]{mattis2024unstablearithmeticfracturesquares},
  and statement (5) is \cite[Proposition 6.9]{mattis2024unstablearithmeticfracturesquares}.

  For (3), since $\Omega^\infty E$ is nilpotent, we have 
  \begin{equation*}
    L_S \Omega^\infty E \wequi \Omega^\infty S^{-1}E \wequi \Omega^\infty L_A^{st} E,
  \end{equation*}
  where we used \cref{lem:motivic:unstable-stable-S-localization} for the first,
  and \cref{lem:stable-bousfield} for the second equivalence.

  We now prove (2).
  Note that $L$ inverts $S$-equivalences, and so in order to show that $L_S X \wequi \hat L_A X$, it suffices to show that $L_S X$ is $L$-local.
  This is clearly true for $\Tot T^\bullet X$, and so, it is enough to show that $\tau_{<\infty} L_S X \wequi \tau_{< \infty} T^\bullet X$.
  Write $F \coloneqq T^\bullet$.
  Let $X$ be resolvable, and consider the generalized Postnikov tower.
  The functor $L_S$ preserves all fiber sequences and the limit appearing in the tower by (4) and (5),
  and similarly $F$ preserves all those fiber sequences by \cref{thm:fiber-sequences}.
  To see that $F$ preserves the limit, 
  note that $(X_i)_i$ is a highly connected tower (under $X$),
  and we conclude by \cref{lem:sheaves:highly-connected} that
  the connectivity of $X \to X_i$ tends to $\infty$.
  Thus, by \cref{prop:Tn-conn} also the connectivity of $FX \to FX_i$ tends to $\infty$,
  even uniformly in the levels of the pro-object. Hence, by definition of $\tau_{<\infty}$,
  it follows that $\tau_{<\infty} FX \wequi \tau_{<\infty} \lim_i FX_i$.
  Hence, we reduced to the case that $X = \Omega^\infty E$ for some motivic spectrum $E \in \Sigma^{3,1} \SH(k)^{\veff}$.
  Consider now the Postnikov tower of $E$ in the $1$-effective t-structure on $\SH(k)(1)$. With exactly the same arguments as above,
  we can thus further reduce to the case that $E \in \Sigma^{m+1,1} \SH(k)^{\eff\heartsuit}$ for $m \ge 2$.  
  Since both sides invert $L_S$-equivalences (for $F$ note that already $L$ inverts them, since any $m \in S$ is invertible on 
  $A$, for this note that it is enough to check that $m$ is invertible on $\ul\pi_*(A)_*$,
  but $\ul\pi_*(A)_*$ is a $\ul\pi_0(A)_*$-algebra, and $m$ is invertible on $\ul\pi_0(A)_* = S^{-1}\ul\pi_0(\1)_*$), we can replace $X$ by $L_S X$,
  and we have to see that $\tau_{<\infty}X \to \tau_{<\infty}FX$ is an equivalence.
  By (3) we know $X = L_S \Omega^\infty E \wequi \Omega^\infty S^{-1} E$.
  Note that $S^{-1} E$ has the structure of a $S^{-1}\ul\pi_0^\eff(\1)$-module.
  We have
  \begin{equation*}
    S^{-1}\ul\pi_0^\eff(\1) \wequi S^{-1} f_0 \ul\pi_0(\1)_* \wequi f_0 S^{-1} \ul\pi_0(\1)_* \wequi f_0 \ul\pi_0(A)_* \wequi \ul\pi_0^\eff(A),
  \end{equation*}
  where we used that $\1$ and $A$ are effective, that $f_0$ preserves colimits, and that $f_0$ is t-exact from the homotopy t-structure 
  to the effective t-structure, see \cite[Lemma 6.2 (2)]{Bachmann2020TowardsConservativity}.
  In particular, $S^{-1} E$ acquires the structure of a $\ul\pi_0^\eff(A)$-module,
  and hence also has the structure of an $A$-module (via the canonical ring map $A \to \ul\pi_0^\eff(A)$).
  Hence, $X$ is in the essential image of $R$,
  and hence $\tau_{<\infty}X \to \tau_{<\infty}FX$ is an equivalence by \cref{lem:adjunction:split} (even before applying $\tau_{<\infty}$).
\end{proof}

\subsubsection{Milnor--Witt Resolutions: Completion}
We assume now that $\ul\pi_0(A)_* \wequi \ul\pi_0(\1)_* / n$
with $n \in \N$.
Thus, $\ul\pi_0(A)_* \otimes_{\ul\pi_0(\1)_*} \ul\pi_0(A)_* \wequi \ul\pi_0(A)_*$,
and hence all axioms are satisfied.
In this situation, the Bousfield--Kan completion is related to unstable $n$-completion.
We write $L_n$ for the unstable $n$-completion functors in $\scr X$, cf.\ \cref{def:unstable-n:def-n-equiv}.

\begin{example}
  A possible example for $A$ is the ring $f_0 (\ul{K}^{MW} / n)$.
\end{example}

\begin{theorem} \label{thm:motivic:milnor-witt-completion}
  \begin{enumerate}
    \item Let $X \in \Spc(k)_*$ be nilpotent. Then $L_n X \in \Spc(k)_*$ (i.e., $L_n X$ is again $\A^1$-invariant).
    \item Let $X \in \Spc(k)_*$ be resolvable. There is a canonical map $L_n X \to \tau_{\ge 1} \Tot T^\bullet X$, which is an equivalence.
    \item For $E \in \SH(k)_{\ge 1}$ the canonical map $L_n \Omega^\infty E \to \tau_{\ge 1} \Omega^\infty L^{st}_A E$
      is an equivalence.
    \item Let $F \to E \to B$ be a fiber sequence of pointed nilpotent Nisnevich sheaves.
      Then the canonical map $L_n F \to \tau_{\ge 1} \fib(L_n E \to L_n B)$ is an equivalence.
    \item Let $X_\bullet$ be a highly connected tower of pointed nilpotent Nisnevich sheaves with $X = \lim_i X_i$.
      Then the canonical map $L_n X \to \lim_i L_n X_i$ is an equivalence.
  \end{enumerate}
\end{theorem}
\begin{proof}
  For (1) note that $L_n X \wequi \prod_{\ell | n} L_\ell X$ by \cref{lem:unstable-n:main-thm}.
  Thus, since $\A^1$-invariant sheaves are stable under limits,
  the claim follows from \cite[Proposition 5.22]{mattis2024unstablepcompletionmotivichomotopy}.
  Moreover, (3) is \cref{lem:motivic:unstable-stable-n-completion} combined with \cref{lem:stable-bousfield},
  (4) is \cref{lem:unstable-n:fiber-sequence} 
  and (5) is \cref{lem:unstable-n:tower}.

  We now show (2).
  Since $\tau_{\ge 1}\colon \scr X_* \to \scr X_{*,\ge 1}$ is a right adjoint, it preserves limits.
  Similarly, the functor $L_n$ restricts to $\scr X_{*,\ge 1}$ (this follows from \cite[Lemma 3.12]{mattis2024unstablepcompletionmotivichomotopy}), 
  and the restricted functor preserves appropriate fiber sequences (by \cref{lem:unstable-n:fiber-sequence}).
  As in the proof for $\ul\pi_0(A)_* = S^{-1} \ul\pi_0(\1)_*$, we reduce to the case 
  that $X = \Omega^\infty \Sigma^k M$ for $M \in \Sigma^{1,1} \SH(k)^{\eff\heartsuit}$ with $k \ge 2$.
  Write $F \coloneqq \tau_{\ge 1} \Tot T_A^\bullet = \tau_{\ge 1} \mat \tau_{<\infty} T_A^\bullet$.
  Then
  \begin{equation*}
    G(M) \coloneqq (F \Omega^\infty \Sigma^k M, F \Omega^\infty \Sigma^{k+1} M, \dots) \in \Sp(\Spc(k)_{*,\ge1}) \wequi \SH^{S^1}(k),
  \end{equation*}
  since $F$ preserves loops by \cref{thm:fiber-sequences}.
  The functor $G$ also preserves fiber sequences of $2$-connective spectra.
  There is a canonical map $\omega^\infty \Sigma^k M \to G(M)$. 
  Assume for now that this is an equivalence whenever $M$ is a finite extension of $\ul\pi_0^\eff(\1)/n$-modules.
  It follows that $\Sigma^k M \to G(M)$ is an $n$-equivalence,
  i.e., induces an equivalence on fibers of multiplication by $n$.
  Since $\omega^\infty \Sigma^k M$ is connected, this implies that also $\omega^\infty \Sigma^k M \to \tau_{\ge 1}G(M)$ is an $n$-equivalence
  (this follows easily from e.g.\ \cite[Lemma 2.9 and Corollary 2.11]{mattis2024unstablepcompletionmotivichomotopy};
  note that in the proof there it is never used that $p$ is a prime).
  Hence, using \cref{lem:unstable-n:infinite-loop-space}, we see that \[ \Omega^\infty \Sigma^k M \to \Omega_{S^1}^\infty \tau_{\ge 1}G(M) \wequi F(\Omega^\infty \Sigma^k M) \] is also an $n$-equivalence.
  Since the target is $n$-complete by construction, we conclude.

  We end the proof by showing that $\Sigma^k M \to G(M)$ is an equivalence if $M$ 
  is a finite extension of $\ul\pi_0^\eff(\1)/n$-modules.
  By preservation of fiber sequences, we immediately reduce to the case where $M$ itself is such a module.
  For this, we will show that in fact $M$ is even an $A$-module, whence the claim follows from \cref{lem:adjunction:split}.
  We have
  \begin{equation*}
    \ul\pi_0^\eff(\1)/n \wequi f_0 (\ul\pi_0(\1)_* /n) \wequi f_0 \ul\pi_0(A)_* \wequi \ul\pi_0^\eff(A),
  \end{equation*}
  where we used that $\1$ and $A$ are effective, and that $f_0$ is t-exact from the homotopy t-structure 
  to the effective t-structure, see \cite[Lemma 6.2 (2)]{Bachmann2020TowardsConservativity}.
  Thus, $M$ acquires the structure of a $\ul\pi_0^\eff(A)$-module,
  and hence also has the structure of an $A$-module (via the canonical ring map $A \to \ul\pi_0^\eff(A)$).
  This finishes the proof.
\end{proof}

\subsubsection{Milnor Resolutions: Generalities}
From now on, we will assume that $\mathrm{cd}_2(k) < \infty$.
We will now study the situation that $\ul\pi_0(A)_* \wequi S^{-1} (\ul\pi_0(\1)_*/\eta)$,
or that $\ul\pi_0(A)_* \wequi \ul\pi_0(\1)_* /(\eta,n)$.

\begin{remark}
  We quickly explain why we need the assumption that $k$ has finite $2$-étale cohomological dimension.
  In topology, going from the sphere $\mathbb{S}_{\mathrm{top}}$ to its $0$-truncation $\pi_0(\mathbb{S}_{\mathrm{top}})$
  kills off only nilpotent elements, hence the associated completion functors agree (at least on bounded below objects).
  In contrast, going from the motivic sphere $\1$ to $\ul\pi_0(\1)_*/ \eta$
  kills the motivic Hopf map $\eta \colon \Gm \to \1$, which is known to be non-nilpotent \cite[Corollary 6.4.5]{morel2004pi0}.
  Hence, homological localization with respect to $H\Z$ introduces a completion operation with respect to $\eta$.
  Unstably, the associated Bousfield localization is currently poorly understood.
  The situation gets much better if one requires the $2$-étale cohomological dimension to be finite,
  as then every very effective motivic spectrum is automatically $\eta$-complete,
  and hence the associated unstable Bousfield localization is just the identity.
\end{remark}

\begin{lemma} \label{lem:motivic:eta-complete}
  Assume that $\mathrm{cd}_2(k) < \infty$,
  and let $E \in \SH(k)^\veff$ be a very effective motivic spectrum.
  Then $E$ is $\eta$-complete.
  If $E$ is $2$-power-torsion, then $E$ is slice-complete.
\end{lemma}
\begin{proof}
Since $\mathrm{cd}_2(k) < \infty$ we know that $E[1/2]$ is $\eta$-complete.
Indeed, there is a canonical splitting $E[1/2] \wequi E[1/2]^+ \oplus E[1/2]^-$ 
with $E[1/2]^+$ an $\eta$-complete motivic spectrum by 
e.g.\ \cite[§2.7.3]{bachmann2021etaperiodicmotivicstablehomotopy},
and under our assumption on the cohomological dimension, we have 
$E[1/2]^- = 0$: Indeed, the splitting of $E[1/2]$ is induced by a splitting of 
the whole $\infty$-category $\SH(k)[1/2]$, since there is a splitting of the 
endomorphisms of the unit $GW(k)[1/2] \wequi \Z[1/2] \times W(k)[1/2]$.
Hence, to show that $E[1/2]^- = 0$, it suffices to show that the Witt ring $W(k)$
has $2^N$ torsion for some $N \gg 0$. Write $I \subset W(k)$ for the fundamental ideal.
Then, by \cite[the discussion after Theorem 5.1]{milnor1973symmetric}, all the quotients 
$I^n / I^{n+1}$ are $\F_2$-vectorspaces, hence $2$-torsion.
It thus is enough to show that $I^N = 0$ for $N \gg 0$,
for this one can argue as in the proof of \cite[Theorem 16]{BachmannConservativityOfTheFunctorToDM}, using that $\mathrm{cd}_2(k) < \infty$.

Using that $\eta$-complete spectra are stable under fiber sequences, replacing $E$ by the cofiber of $E \to E[1/2]$ we may assume that $E$ is $2$-power-torsion.
To prove that $E$ is $\eta$-complete it now suffices to show that it is slice-complete (indeed, since 
$E$ is effective, the slice completion is a limit of 
finite extensions of slices, and slices are $\eta$-complete as $s_0(\1) = H\Z$).
This follows from Levine's convergence theorem.
If $k$ has characteristic $\ne 2$, see \cite[Theorem 7.3]{levine2013convergence} (also recalled in \cite[Theorem 5.3]{bachmann-bott}).
If $k$ has characteristic $2$, see our variant in \cref{cor:levine-convergence}.
\end{proof}

We need the following ad-hoc definition:
\begin{definition}
  Let $X \in \Spc(k)_*$ be resolvable.
  We say that $X$ is $2$\emph{-power-torsion} if one can choose the $X_i$ and $B_i$ 
  appearing in \cref{def:motivic:resolvable} so that $B_i$ is $2$-power-torsion (i.e.\ $B_i[1/2] = 0$) for all $i$.
\end{definition}

\begin{lemma} \label{lem:motivic:infty-effective-vanishing}
  Suppose that $\mathrm{cd}_2(k) < \infty$, and let $X \in \Pro(\Spc(k)_*)$ be nilpotent, $2$-power-torsion and $\infty$-effective, i.e., $X \wequi \lim_i cX_i$, where $X_i$ is nilpotent, $2$-power-torsion and $n_i$-effective, with $\lim_i n_i = \infty$.
  Then $\mat X = *$.
\end{lemma}
\begin{proof}
It suffices to show that if $X \in \Spc(k)_*$ is nilpotent, $2$-power-torsion and $n$-effective (without loss of generality we may assume that $n \ge 2$, so that $X$ is resolvable by \cref{rmk:motivic:resolvable}) and $U \in \Sm_k$ has dimension $\le d$, then $X|_{U_\Nis}$ is $N(d,n)$-connective, where $N(d,n)$ is a function (independent of $X$ and $U$) with $\lim_n N(d,n) = \infty$.
(Indeed, the restrictions to the small sites $U_{\Nis}$ preserve all limits and are jointly conservative,
and moreover $\lim_i (X_i|_{U_{\Nis}}) = *$ by e.g.\ \cite[Lemma 2.4]{mattis2025etale} and the increasing connectivity of the $X_i|_{U_{\Nis}}$.)
By an induction on the generalized Postnikov tower of \cref{def:motivic:resolvable}, we may assume that $X = \Omega^\infty E$ for some $E \in \Sigma^{n,n}\SH(k)^\veff$ with $E[1/2]=0$.
The claim now follows from Levine's convergence theorem; see  \cite[Theorem 7.3]{levine2013convergence} (also recalled in \cite[Theorem 5.3]{bachmann-bott}) if $k$ has characteristic $\ne 2$, and \cref{cor:levine-convergence} else.
\end{proof}

\subsubsection{Milnor Resolutions: Localization}
In this subsection, we assume that $\ul\pi_0(A)_* \wequi S^{-1} (\ul\pi_0(\1)_*/\eta)$
with $S \subset \Z \setminus 0$.
In this situation, the Bousfield--Kan completion is related to unstable $S$-localization.
We write $L_S$ for the unstable $S$-localization functor in $\scr X$, cf.\ \cite[\S 3]{mattis2024unstablearithmeticfracturesquares}.
\begin{example}
  Possible examples for $A$ are the rings $S^{-1} H\Z$ 
  and $S^{-1} \mathrm{MGL}$.
  Of course, we may take $S=\{1\}$.
\end{example}
Before stating the main theorem in this situation, let us identify the stable localization functor.
\begin{lemma} \label{lem:motivic:milnor-localization-stable}
  Assume that $\mathrm{cd}_2(k) < \infty$.
  Let $E \in \SH(k)^\veff$. Then $L^{st}_A(E) \wequi S^{-1} E$.
\end{lemma}
\begin{proof}
  Let $m \in S$. We first show that $A \otimes (E \xrightarrow{m} E)$ is invertible,
  for which it suffices to show that $m$ is invertible on $A$,
  which holds since $m$ is invertible on $\ul\pi_0(A)_* = S^{-1} \ul\pi_0(\1)_*/\eta$.
  Thus, we see that the canonical map $E \to L^{st}_A(E)$ factors 
  through $E \to S^{-1}E$, and that the resulting map $S^{-1} E \to E$ is inverted by $L^{st}_A(\ph)$.
  Hence, it suffices to show that $S^{-1} E$ is $L^{st}_A$-local.
  We have morphisms $S^{-1}\1 \to A \to \ul\pi_0(A)_* = S^{-1} (\ul\pi_0(\1)_* / \eta)$, and hence, get a diagram
  \begin{equation*}
    L^{st}_{S^{-1}\1} E \to L^{st}_{A} E \to L^{st}_{S^{-1} (\ul\pi_0(\1)_* / \eta)} E.
  \end{equation*}
  By \cref{lem:stable-bousfield} we see that
  \begin{equation*}
    L^{st}_{S^{-1}\1} E \wequi S^{-1} E
  \end{equation*}
  and
  \begin{equation*}
    L^{st}_{S^{-1} (\ul\pi_0(\1)_* / \eta)} E \wequi L^{st}_{S^{-1} \1 / \eta} E \wequi (S^{-1} E)^{\comp}_\eta \wequi S^{-1} E,
  \end{equation*}
  where the last equivalence is \cref{lem:motivic:eta-complete}.
  Thus, the above sequence of morphisms is equivalent to a retract diagram $S^{-1} E \to L_{A}^{st} E \to S^{-1} E$.
  Since $L_A^{st} E$ is $L_A^{st}$-local, the same is true for the retract $S^{-1} E$,
  which is what we wanted to show.
\end{proof}

\begin{theorem}\label{thm:motivic:milnor-localization}
  Assume that $\mathrm{cd}_2(k) < \infty$.
  \begin{enumerate}
    \item Let $X \in \Spc(k)_*$ be nilpotent. Then $L_S X \in \Spc(k)_*$ (i.e., $L_S X$ is again $\A^1$-invariant).
    \item Let $X \in \Spc(k)_*$ be resolvable. The canonical maps $L_S X \to \hat L_A X \to \Tot T_A^\bullet X$ are equivalences.
    \item For $E \in \SH(k)^\eff_{\ge 1}$ the canonical map $L_S \Omega^\infty E \to \Omega^\infty L^{st}_A E$
      is an equivalence.
    \item Let $F \to E \to B$ be a fiber sequence of pointed nilpotent Nisnevich sheaves.
      Then $L_S F \to L_S E \to L_S B$ is a fiber sequence.
    \item Let $X_\bullet$ be a tower of pointed nilpotent Nisnevich sheaves with $X = \lim_i X_i$.
      Assume that the tower is locally highly connected subordinate to $\scr U$, e.g.\ it is highly connected.
      Then $L_S X \wequi \lim_i L_S X_i$.
  \end{enumerate}
\end{theorem}
\begin{proof}
  (1), (4) and (5) were already shown in \cref{thm:motivic:milnor-witt-localization}.

  For (3) we combine \cref{lem:motivic:unstable-stable-S-localization} and \cref{lem:motivic:milnor-localization-stable}.
  
  Hence, the only thing left is (2).
  As before, since $L$ inverts $S$-equivalences, in order to prove that $L_S X \wequi \hat L_A X$, it suffices to show that $L_S X \wequi \Tot T_A^\bullet X$ (the latter being $L$-local).
  Write $F \coloneqq \Tot T_A^\bullet$.
  It is clear that $F$ inverts $S$-equivalences and hence the canonical map $X \to FX$ factors through $L_S X$.
  Let $X$ be resolvable, and consider its generalized Postnikov tower (\cref{def:motivic:resolvable}).
  The functor $L_S$ preserves all fiber sequences and the limit appearing in the tower by (4) and (5),
  and similarly $F$ preserves all those fiber sequences by \cref{thm:fiber-sequences} and \cref{rmk:generalized-tower}.
  To see that $F$ preserves the limit, 
  note that $(X_i)_i$ is a highly connected tower (under $X$),
  and we conclude by \cref{lem:sheaves:highly-connected} that
  the connectivity of $X \to X_i$ tends to $\infty$.
  Thus, by \cref{prop:Tn-conn} also the connectivity of $FX \to FX_i$ tends to $\infty$.
  Moreover, the same proposition shows that $FX_i$ is a highly connected tower.
  In particular, it follows again from \cref{lem:sheaves:highly-connected}
  that $FX \wequi \lim_i FX_i$.
  We have thus reduced to the case $X = \Omega^\infty E$ for some motivic spectrum $E \in \Sigma^{3,1} \SH(k)^\veff$.
  In fact, since both sides invert $S$-equivalences, we may assume that $E \wequi S^{-1}E$, and we must show that $X \to FX$ is an equivalence.
  Consider the fiber sequence \[ E \to E[1/2] \to C. \]
  Applying $\Omega^\infty$, we obtain a fiber sequence preserved by $F$ (by \cref{thm:fiber-sequences}, 
  as all involved motivic spaces are $1$-effective and $2$-connective) and $L_S$; it thus suffices to treat $E[1/2]$ and $C$, i.e., $2$-periodic spectra and $2$-power-torsion spectra.

  Let $E$ be $2$-periodic.
  Then since $\mathrm{cd}_2(k) < \infty$, $\eta$ acts by zero on $E$: 
  Indeed, as in the proof of \cref{lem:motivic:eta-complete}, we see that $E \wequi E^+$ (i.e.\ the minus-part vanishes),
  and we know from \cite[Lemma 39]{bachmann2018realetale} that $\eta$ acts by zero on $E^+$.
  Consider now the Postnikov tower of $E$ in the $1$-effective t-structure on $\SH(k)(1)$.
  By an analogous argument as in the beginning of the proof we are thus allowed to reduce to the 
  case where $E \in \Sigma^{k+1,1} \SH(k)^{\eff\heartsuit}$ for some $k \ge 2$,
  which is still $S$- and $2$-periodic.
  Consider $\Sigma^k\ul\pi_k(E)_*$, which again is $S$- and $2$-periodic, and hence $\eta$ acts by zero on $E$.
  In particular, $\Sigma^k\ul\pi_k(E)_*$ is a $(S^{-1} \ul\pi_0(\1)_*/\eta = \ul\pi_0(A)_*)$-module.
  Applying $f_1$, which is t-exact from the homotopy t-structure to the $1$-effective t-structure 
  by \cite[Lemma 6.2 (2)]{Bachmann2020TowardsConservativity}, we get 
  that $E = \Sigma^k \ul\pi_k^{1\eff}(E) \wequi \Sigma^k f_1 \ul\pi_k(E)_* \wequi \Sigma^{k,1} f_0 \Sigma^{0,-1} \ul\pi_k(E)_*$ 
  is an $(f_0 \ul\pi_0(A)_* = \ul\pi_0^\eff(A))$-module.
  In particular, it is also an $A$-module, 
  and thus the canonical map $\Omega^\infty E \to F \Omega^\infty E$ 
  is an equivalence by \cref{lem:adjunction:split}.

  Now let $E$ be $2$-power-torsion.
  Consider the fiber sequences $f_n E \to E \to E/f_n E$, which are preserved by $F\Omega^\infty$.
  We have $F\Omega^\infty(E/f_nE) \wequi \Omega^\infty(E/f_nE)$: Indeed, $E / f_nE$ is a finite extension of slices 
  (since $E$ is effective). Thus, as above we can reduce to the case of $s_n(E)$.
  Slices are $H\Z$-modules, hence $s_n(E)$ is an $S^{-1} H\Z$-module. We have
  \begin{equation*}
    S^{-1} H\Z \wequi S^{-1} f_0 (\ul\pi_0(\1)_*/\eta) \wequi f_0 \ul\pi_0(A)_* \wequi \pi_0^\eff(A),  
  \end{equation*}
  using \cite[Lemma 6.2 (2)]{Bachmann2020TowardsConservativity} and \cite[Lemma 12]{bachmann2017generalized}. In particular,
  we see that $s_n(E)$ is an $\ul\pi_0^\eff(A)$-module and hence an $A$-module, and thus 
  $F\Omega^\infty(s_n E) \wequi \Omega^\infty(s_nE)$ by \cref{lem:adjunction:split}.
  Taking the limit over $n$ we obtain a fiber sequence \[ \mat(\lim_{n,m} c\Tot_m T_A^\bullet \Omega^\infty(f_nE)) \to F\Omega^\infty E \to \lim_n \Omega^\infty E/f_n E \wequi \Omega^\infty E, \] where the equivalence is by \cref{lem:motivic:eta-complete}.
  We finish the proof by showing that the second map is an equivalence.
  Since the base is connected, for this it will be enough to show that the fiber is contractible.
  This follows from \cref{cor:motivic-preserve-effectivity} and \cref{lem:motivic:infty-effective-vanishing}.
\end{proof}

\subsubsection{Milnor Resolutions: Completion}
In this subsection, we assume that $\ul\pi_0(A)_* \wequi \ul\pi_0(\1)_* / (\eta, n)$
with $n \in \N$.
In this situation, the Bousfield--Kan completion is related to unstable $n$-completion.
We write $L_n$ for the unstable $n$-completion functor in $\scr X$, cf.\ \cref{def:unstable-n:def-n-equiv}.
\begin{example}
  A possible example for $A$ is the ring $H\Z/n$.
\end{example}
Before stating the main theorem in this situation, let us identify the stable localization functor.
\begin{lemma} \label{lem:motivic:milnor-completion-stable}
  Assume that $\mathrm{cd}_2(k) < \infty$.
  Let $E \in \SH(k)^\veff$. Then $L^{st}_A(E) \wequi E^{\comp}_n$.
\end{lemma}
\begin{proof}
  By assumption, we get morphisms 
  $\1/n \to A \to \ul\pi^\eff_0(A) = H\Z/n$, and thus a diagram
  \begin{equation*}
    L^{st}_{\1/n} E \to L^{st}_{A} E \to L^{st}_{H\Z/n} E,
  \end{equation*}
  where the first map is inverted by $L^{st}_A(-)$.
  By definition $L^{st}_{\1/n} E = E^\comp_n$.
  Hence, it suffices to show that $L^{st}_{\1/n} E$ is $L^{st}_A$-local.
  By \cite[\S 5.2]{mantovani2021localizationscompletionsstableinftycategories} we get
  \begin{equation*}
    L^{st}_{H\Z/n} E \wequi L^{st}_{\1 / (\eta, n)} E = E^{\comp}_{\eta,n} \wequi E^\comp_n,
  \end{equation*}
  where the last equivalence is \cref{lem:motivic:eta-complete}.
  Thus, the above sequence of morphisms is equivalent to a retract diagram $E^\comp_n \to L_{A}^{st} E \to E^\comp_n$.
  Since $L_A^{st} E$ is $L_A^{st}$-local, the same is true for the retract $E^\comp_n$,
  which is what we wanted to show.
\end{proof}

\begin{theorem} \label{thm:motivic:milnor-completion}
  For this theorem, assume that $\mathrm{cd}_2(k) < \infty$.
  \begin{enumerate}
    \item Let $X \in \Spc(k)_*$ be nilpotent. Then $L_n X \in \Spc(k)_*$ (i.e., $L_n X$ is again $\A^1$-invariant).
    \item Let $X \in \Spc(k)_*$ be resolvable. There is a canonical map $L_n X \to \tau_{\ge 1} \Tot T^\bullet X$, which is an equivalence.
    \item For $E \in \SH(k)_{\ge 1}$ the canonical map $L_n \Omega^\infty E \to \tau_{\ge 1} \Omega^\infty L^{st}_A E$
      is an equivalence.
    \item Let $F \to E \to B$ be a fiber sequence of pointed nilpotent Nisnevich sheaves.
      Then $L_S F \wequi \tau_{\ge 1} \fib(L_S E \to L_S B)$.
    \item Let $X_\bullet$ be a tower of pointed nilpotent Nisnevich sheaves with $X = \lim_i X_i$.
      Assume that the tower is locally highly connected subordinate to $\scr U$, e.g.\ it is highly connected.
      Then $L_n X \wequi \lim_i L_n X_i$.
  \end{enumerate}
\end{theorem}
\begin{proof}
  (1), (4) and (5) were already shown in \cref{thm:motivic:milnor-witt-completion}.

  For (3) we combine \cref{lem:motivic:unstable-stable-n-completion} and \cref{lem:motivic:milnor-completion-stable}.

  Hence, the only thing left is (2). Exactly as in the proof of \cref{thm:motivic:milnor-witt-completion}
  we reduce to the case that $X = \Omega^\infty \Sigma^k M$ for $M \in \Sigma^{1,1} \SH(k)^{\eff\heartsuit}$ 
  such that $n$ acts by zero on $M$.

  Considering the fiber sequence $M \to M[1/2] \to C$ (and potentially further splitting up $C$ into two objects in the heart),
  we may furthermore assume that either $M$ is $2$-periodic or $2$-power-torsion.
  One now argues exactly as in \cref{thm:motivic:milnor-localization} to show that in both cases 
  the canonical map $\Omega^\infty \Sigma^k M \to F\Omega^\infty \Sigma^k M$ is an equivalence.
\end{proof}

\appendix
\section{Coherence data from adjunctions} \label{sec:appendix:coherence-data}
In this section, let $L \colon \scr C \adj \scr D \noloc R$ 
be an adjunction between $\infty$-categories.
Alternatively, by work of Riehl--Verity \cite{riehl20152ccategorytheoryofquasicats} and Haugseng \cite[Theorem 1.1]{haugseng2021laxtransformationsadjunctionsmonads}, 
this adjunction is classified 
by a functor $\scr M \colon \ADJ \to \mathrm{Cat}_\infty$ of $(\infty, 2)$-categories,
where $\mathrm{Cat}_\infty$ is the $(\infty,2)$-category of $\infty$-categories,
and $\ADJ$ is the free living adjunction \cite{schanuel1986free,Auderset:Adjunctions-and-monads},
which is a $(2,2)$-category.
In particular, recall that $\ADJ$ has two objects $0$ and $1$,
and there are morphisms $l \colon 0 \to 1$ and $r \colon 1 \to 0$,
such that $\scr M(0) = \scr C$, $\scr M(1) = \scr D$, $\scr M(l) = L$ 
and $\scr M(r) = R$.
Moreover, there are $2$-morphisms $u \colon \id_0 \to rl$ 
and $c \colon lr \to \id_1$ that map under $\scr M$ to the unit and counit 
of the adjunction $L \dashv R$.

In this section, we explain how one can use this point of view to 
extract simplicial data out of this adjunction,
such as, for example, the cobar resolution or the codegeneracy cube.
As far as we are aware, some of the results presented in this appendix 
have not appeared in the literature in the setting of $\infty$-categories, even though they are well-known to experts.

In the following, $\MAP_{\mathbb E}(-, -)$ denotes the mapping $\infty$-category 
in an $(\infty, 2)$-category $\mathbb E$. This appendix is the only part 
of the article where we make (light) use of the theory of $(\infty, 2)$-categories.

\begin{construction}[Cobar resolution] \label{lem:adjunction:cobar-construction}
  We have $\MAP_{\ADJ}(0, 0) \wequi \Delta_+$ \cite[Corollary 2.8]{Auderset:Adjunctions-and-monads}.
  This lets us define an augmented cosimplical object $(\CB)^\bullet_+ \colon \Delta_+ \to \Fun(\scr C, \scr C)$ 
  as the composition
  \begin{equation*}
    \Delta_+ \wequi \MAP_{\ADJ}(0, 0) \xrightarrow{\scr M} \MAP_{\mathrm{Cat}_\infty}(\scr C, \scr C) = \Fun(\scr C, \scr C).
  \end{equation*}
\end{construction}

\begin{remark}
  By \cite[Corollary 2.8]{Auderset:Adjunctions-and-monads}, the object $[n] \in \Delta_+$
  corresponds to the morphism $(rl)^n \in \MAP_{\ADJ}(0, 0)$.
  Therefore, $(\CB)^n_+ = \scr M((rl)^n) = (RL)^n$.
  Similarly, one sees that the transition morphisms in the cosimplical diagram are given by (compositions of) 
  the co/unit of the adjunction.
\end{remark}

Recall the categories $\Delta_{+\infty}$ and $\Delta_{-\infty}$ from \cite[Lemma 6.1.3.16]{highertopoi} and \cite[Definition 4.7.2.1]{higheralgebra}.
An augmented cosimplicial object $X^\bullet \colon \Delta_+ \to \scr C$
is called \emph{split} if it extends along the inclusion $\Delta_+ \subset \Delta_{+\infty}$ \cite[Definition 4.7.2.2]{higheralgebra}, 
or equivalently along the inclusion $\Delta_+ \subset \Delta_{-\infty}$
\cite[Remark 4.7.2.3]{higheralgebra}. 
\begin{remark} \label{rmk:adjunction:split-limit}
  Split objects have the useful property 
  that they are automatically limit diagrams \cite[dual of Lemma 6.1.3.16]{highertopoi}, and the limit is preserved by any functor (since any 
  functor preserves the extension and thus split objects).
\end{remark}

\begin{proposition} \label{lem:adjunction:split}
  Both $L (\CB)^\bullet_+$ and $(\CB)^\bullet_+ R$ are split cosimplicial objects.
  In particular, they are limit diagrams which are preserved by any functor.
\end{proposition}
\begin{proof}
  We start with the statement about $L (\CB)^\bullet_+$.
  Consider the following diagram:
  \begin{center}
      \begin{tikzcd}
          \Delta_+ \ar[d, hook] \ar[r, equal] \ar[rr, dashed, bend left, "{(\CB)^\bullet_+}"] &\Map_{\ADJ}(0,0) \ar[d, "l \circ (-)"] \ar[r, "\scr M"] &\Fun(\scr C, \scr C) \ar[d, "L \circ (-)"] \\
          \Delta_{-\infty} \ar[r, equal] &\Map_\ADJ(0,1) \ar[r, "\scr M"] &\Fun(\scr C, \scr D)\rlap{.}
      \end{tikzcd}
  \end{center}
  The squares commute essentially by \cite[Corollary 2.8]{Auderset:Adjunctions-and-monads} and since $\scr M$ is a $2$-functor.
  The composition on the top and then the right is $L (\CB)^\bullet_+$. Since the diagram commutes,
  we see that this has an extension to $\Delta_{-\infty}$, thus is split.
  For the other statement, consider instead the following diagram:
  \begin{center}
      \begin{tikzcd}
          \Delta_+ \ar[d, hook] \ar[r, equal] \ar[rr, dashed, bend left, "{(\CB)^\bullet_+}"] &\Map_{\ADJ}(0,0) \ar[d, "(-) \circ r"] \ar[r, "\scr M"] &\Fun(\scr C,\scr C) \ar[d, "(-) \circ R"] \\
          \Delta_{+\infty} \ar[r, equal] &\Map_{\ADJ}(1,0) \ar[r, "\scr M"] &\Fun(\scr D,\scr C)\rlap{.}
      \end{tikzcd}
  \end{center}
  The squares commute essentially by (the dual of) \cite[Corollary 2.8]{Auderset:Adjunctions-and-monads} and since $\scr M$ is a $2$-functor.
  The composition on the top and then the right is $(\CB)^\bullet_+ R$. Since the diagram commutes,
  we see that this has an extension to $\Delta_{+\infty}$, thus is split.
\end{proof}

\begin{construction}
  Write $\Delta^{\pm} \subset \Delta$ for the wide subcategory 
  spanned by those maps that preserve both the minimal and maximal element.
  Moreover, write $\Delta_+^\inj \subset \Delta_+$ and $\Delta^\surj \subset \Delta$ 
  for the wide subcategories of injective, respectively surjective maps.
  In particular, $\Delta^\surj \subset \Delta^\pm$.
  Observe that there is a canonical equivalence $\Phi \colon \Delta^\pm \xrightarrow{\simeq} (\Delta_+)^\op$
  \cite[\S 3.11]{Street1980}.
  This equivalence restricts to an equivalence $\Phi \colon \Delta^\surj \xrightarrow{\simeq} (\Delta_+^\inj)^\op$.
  For $n \ge 0$, slicing over $[n]$ induces an equivalence 
  $\Phi_n \colon \Delta^\surj_{[n]/} \xrightarrow{\simeq} (\Delta_{+,/[n-1]}^{\inj})^{\mathrm{op}}$.
\end{construction}

\begin{definition}
  For $n \ge 0$ and $\scr E$ an $\infty$-category,
  write $\mathrm{Cube}^n(\scr E) \coloneqq \Fun(\Delta^\surj_{[n]/}, \scr E)$
  for the $\infty$-category of \emph{cubes} in $\scr E$.
\end{definition}

\begin{definition} \label{def:adjunction:codegen-cube}
  Let $n \ge 0$. We define two cubes as the following compositions:
  \begin{equation*}
    (\CB)^\bullet_{\surj, n} \colon \Delta^{\surj}_{[n]/} \to \Delta^{\surj} \subset \Delta_+ = \MAP_{\ADJ}(0,0) \xrightarrow{\scr M} \Fun(\scr C, \scr C),
  \end{equation*}
  so that $(\CB)^\bullet_{\surj, n} \in \Cube^n(\Fun(\scr C, \scr C))$, and 
  \begin{equation*}
    (\CB)^\bullet_{+,\inj, n} \colon \Delta^{\surj}_{[n]/} \xrightarrow[\simeq]{\Phi_n} (\Delta_{+,/[n-1]}^{\inj})^{\op} \to (\Delta_{+}^{\inj})^{\op} \subset (\Delta_+)^\op = \MAP_\ADJ(1,1) \xrightarrow{\scr M} \Fun(\scr D, \scr D),
  \end{equation*}
  so that $(\CB)^\bullet_{+,\inj, n} \in \Cube^n(\Fun(\scr D, \scr D))$.
  We call $(\CB)^\bullet_{\surj, n}$ the $n$-th \emph{codegeneracy cube}
  and $(\CB)^\bullet_{+,\inj, n}$ the $n$-th \emph{face cube} of the adjunction $L \dashv R$.
\end{definition}

\begin{definition}
  Let $\scr E$ be an $\infty$-category and $n \ge 0$. We define a functor 
  $\iota_\mathrm{top} \colon \Delta^{\surj}_{[n]/} \to \Delta^{\surj}_{[n+1]/}$
  via 
  \begin{equation*}
    (f \colon [n] \twoheadrightarrow [k]) \mapsto ([n+1] = [n] \star [0] \xrightarrow{f \star \id} [k] \star [0] = [k+1]),
  \end{equation*}
  where $\star$ is the join operation.
  Precomposition with $\iota_\mathrm{top}$ gives a functor
  \begin{equation*}
    (\ph)^\mathrm{top} \colon \Cube^{n+1}(\scr E) \to \Cube^n(\scr E),
  \end{equation*}
  the \emph{top face} of a cube.

  Similarly, we define $\iota_\mathrm{bot} \colon \Delta^{\surj}_{[n]/} \to \Delta^{\surj}_{[n+1]/}$
  via precomposition with the surjective function $[n+1] \to [n]$ that is the identity on $0 \le k \le n$ 
  and sends $n+1$ to $n$.
  Precomposition with $\iota_\mathrm{bot}$ gives a functor
  \begin{equation*}
    (\ph)^\mathrm{bot} \colon \Cube^{n+1}(\scr E) \to \Cube^n(\scr E),
  \end{equation*}
  the \emph{bottom face} of a cube.

  There is a natural transformation $\iota_\mathrm{top} \to \iota_\mathrm{bot}$
  given on $f \colon [n] \twoheadrightarrow [k]$ by the square 
  \begin{center}
    \begin{tikzcd}
      {[n + 1]} \ar[d, "\id"] \ar[r, "{\iota_{top}(f)}"] &{[k+1]} \ar[d, "{i \mapsto i, k+1 \mapsto k}"] \\
      {[n + 1]} \ar[r, "{\iota_{bot}(f)}"] &{[k]}\rlap{,}
    \end{tikzcd}
  \end{center}
  and hence an induced natural transformation 
  \begin{equation*}
    \psi \colon (\ph)^\mathrm{top} \to (\ph)^\mathrm{bot}.
  \end{equation*}
\end{definition}

This lets us inductively define the total fiber of a cube:
\begin{definition}
  Let $\scr E$ be an $\infty$-category with finite limits, and $n \ge 0$.
  We inductively define a functor $\mathrm{ToFib} \colon \Cube^n(\scr E) \to \scr E$ as follows:
  If $n = 0$, we use the equivalence $\mathrm{ToFib} \colon \Cube^n(\scr E) \wequi \scr E$ 
  given by evaluating at $[0]$.
  If $n \ge 1$, then we define 
  \begin{equation*}
    \mathrm{ToFib}(\ph) \coloneqq \fib(\mathrm{ToFib}((\ph)^\mathrm{top}) \xrightarrow{\psi} \mathrm{ToFib}((\ph)^\mathrm{bot})).
  \end{equation*}
\end{definition}

We write $C^n \coloneqq \mathrm{ToFib}((\CB)^\bullet_{\surj, n}) \colon \scr C \to \scr C$  
and $S^n \coloneqq \mathrm{ToFib}((\CB)^\bullet_{+,\inj,n}) \colon \scr D \to \scr D$ 
for the total fibers of the codegeneracy and face cubes.
  
\begin{proposition} \label{lem:adjunction:cube-relation}
  There is a canonical equivalence of cubes $(\CB)^\bullet_{\surj, n} \wequi R (\CB)^\bullet_{+,\inj,n} L$,
  and hence also of their total fibers $C^n \wequi RS^nL$.
\end{proposition}
\begin{proof}
  We have the following diagram:
  \begin{center}
      \begin{tikzcd}
          \Delta^{\surj}_{[n]/} \ar[d, "\wequi", "\Phi_{n}"'] \ar[rrrrr, bend left=10, dashed, "(\CB)^\bullet_{\surj, n}"] \ar[r] &\Delta^{\surj} \ar[d, "\wequi", "\Phi"'] \ar[r, hook] &\Delta^{\pm} \ar[d, "\wequi", "\Phi"']  \ar[r, hook] &\Delta_+ \ar[r, equal] &\Map_{\ADJ}(0,0) \ar[r, "\scr M"'] &\Fun(\scr C,\scr C) \\
          (\Delta_{+,/[n-1]}^{\inj})^{\mathrm{op}} \ar[r] \ar[rrrrr, bend right=10, dashed, "(\CB)^\bullet_{+,\inj,n} \circ \Phi_{n}^{-1}"'] &(\Delta_{+}^{\inj})^\op \ar[r, hook] &(\Delta_+)^\op \ar[rr, equal] && \Map_{\ADJ}(1,1) \ar[r, "\scr M"] \ar[u, "r \circ (-) \circ l"] &\Fun(\scr D,\scr D) \ar[u, "R \circ (-) \circ L"] \rlap{.}
      \end{tikzcd}
  \end{center}
  Here, the left two squares commute by the definitions of $\Phi_n$ and $\Phi$.
  That the rectangle commutes can be seen by comparing the definition of $\Phi$ 
  with the explicit description of the morphism $r \circ (\ph) \circ l$, cf.\ \cite[Corollary 2.8]{Auderset:Adjunctions-and-monads}.
  The right square commutes since $\scr M$ is a $2$-functor.
  Thus, the composition on the top agrees with the composition on the left, bottom and right.
  This is exactly what we want to prove.
\end{proof}

\begin{proposition} \label{lem:adjunction:cube-tofib-relation}
  There is a fiber sequence 
  \begin{equation*}
    S^{n+1} \to S^n LR \to S^n,
  \end{equation*}
  where the second map is induced by the counit of the adjunction $L \dashv R$.
\end{proposition}
\begin{proof}
  By definition there is a fiber sequence $S^{n+1} \to \mathrm{ToFib} (\CB)^{\bullet, \mathrm{top}}_{+,\inj,n} \to \mathrm{ToFib} (\CB)^{\bullet, \mathrm{bot}}_{+,\inj,n}$.
  Therefore, it suffices to identify $(\CB)^{\bullet, \mathrm{top}}_{+,\inj,n} \wequi (\CB)^{\bullet}_{+,\inj,n-1} LR$
  and $(\CB)^{\bullet, \mathrm{bot}}_{+,\inj,n} \wequi (\CB)^{\bullet}_{+,\inj,n-1}$.
  We start with the first claim, for which we have to show that the following diagram commutes:
  \begin{center}
    \begin{tikzcd}
        \Delta^{\surj}_{[n-1]/} \ar[d, "\iota_{\mathrm{top}}"] \ar[r, "\Phi_{n-1}"]& \Delta^{\inj,\op}_{+,/[n-2]} \ar[r]& \Delta_{+}^\op \ar[r, equal] & \MAP_\ADJ(1,1) \ar[d, "lr \circ -"] \ar[r, "\scr M"] &\Fun(\scr{D},\scr{D})  \ar[d, "LR \circ -"] \\
        \Delta^{\surj}_{[n]/} \ar[r, "\Phi_{n}"]  &\Delta^{\inj, \op}_{+,/[n-1]} \ar[r] &\Delta_{+}^\op \ar[r, equal] & \MAP_\ADJ(1,1) \ar[r, "\scr M"] &\Fun(\scr{D}, \scr{D})\rlap{.}
    \end{tikzcd}
  \end{center}
  The left rectangle commutes by the explicit description of the morphism $lr \circ -$,
  cf.\ \cite[Corollary 2.8]{Auderset:Adjunctions-and-monads}.
  The right square commutes since $\scr M$ is a $2$-functor.
  The second claim follows from the commutativity of the diagram 
  \begin{center}
    \begin{tikzcd}
        \Delta^{\surj}_{[n-1]/}\ar[d, "\iota_{\mathrm{bot}}"] \ar[r, "\Phi_{n-1}"]& \Delta^{\inj,\op}_{+,/[n-2]} \ar[dr]&\\
        \Delta^{\surj}_{[n]/} \ar[r, "\Phi_{n}"]  &\Delta^{\inj,\op}_{+,/[n-1]} \ar[r] &\Delta^{\inj,\op}_{+}\rlap{,}
    \end{tikzcd}
  \end{center}
  which is immediate from the definition.
\end{proof}

\section{Bousfield localizations in stable categories at Moore objects} \label{sec:stable-localization}
In this section, we prove the following proposition:
\begin{proposition} \label{lem:stable-bousfield}
  Let $\scr D$ be a stable presentably symmetric monoidal $\infty$-category with unit $\1$, and $\scr D_{\ge 0} \subset \scr D$ 
  the connective part of a t-structure that is left-complete (cf.\ the discussion before \cite[Remark 1.2.1.18]{higheralgebra})
  and compatible with the symmetric monoidal structure and filtered colimits.
  Suppose that $A \in \CAlg(\scr D_{\ge 0})$. Write $L$ for the Bousfield
  localization at those morphisms $f$ such that $f \otimes A$ is an equivalence.
  Suppose moreover that there are maps $x_i \colon L_i \to \1$ such that $L_i \in \scr D_{\ge 0}$ 
  is dualizable, with dual $DL_i \in \scr D_{\ge 0}$. Let $S \subset \Z \setminus 0$ 
  be a set of nonzero integers.
  If $\pi_0(A) \wequi S^{-1} \pi_0(\1 \sslash (x_1, \dots, x_n))$,
  then $LE \wequi (S^{-1} E)^\comp_{x_1, \dots, x_n}$ for all $E \in \scr D_{\ge 0}$.
\end{proposition}
\begin{proof}
  As the t-structure on $\scr D$ is compatible with filtered colimits,
  it restricts to a t-structure on $\scr D[S^{-1}]$ such that the localization functor 
  $S^{-1}(-)$ is t-exact. It is moreover symmetric monoidal,
  so that $S^{-1} L_i \in \scr D[S^{-1}]_{\ge 0}$ 
  is dualizable, with dual $S^{-1} DL_i \in \scr D[S^{-1}]_{\ge 0}$.
  Note that $A \in \scr D[S^{-1}]$ since $S$ is invertible on $\pi_0(A)$ by assumption.
  Let $E \in \scr D_{\ge 0}$. 
  Since $\pi_0(A) = S^{-1} \pi_0(\1 \sslash (x_1, \dots, x_n)) = \pi_0((S^{-1} \1) \sslash (x_1, \dots, x_n))$,
  we conclude from \cite[Theorems 2.1 and 2.2]{bachmann2022topologicalmodels} that 
  $(S^{-1} E)^\comp_{x_1, \dots, x_n} \wequi \lim (A \otimes S^{-1} E \rightrightarrows A \otimes A \otimes S^{-1} E \cdots)$
  is given by the $A$-nilpotent completion.
  Consider the canonical map $E \to S^{-1} E \to (S^{-1} E)^\comp_{x_1, \dots, x_n}$.
  Since the nilpotent completion is $L$-local (as a limit of $A$-modules),
  the right-hand side is $L$-local.
  Thus, as the map is an $(- \otimes A)$-equivalence, 
  we get an equivalence $LE \wequi (S^{-1} E)^\comp_{x_1, \dots, x_n}$.
\end{proof}

\section{Unstable \texorpdfstring{$n$}{n}-completion} \label{sec:unstable-n-comp}
In this appendix we discuss unstable completion at a (possibly infinite) set of primes.
\begin{lemma} \label{lem:stable-bousfield:n-completion}
  Let $\scr D$ be a stable presentable $\infty$-category, and $n \in \Z \setminus 0$.
  Then $E^\comp_n \wequi \bigoplus_{\ell | n} E^\comp_\ell$,
  where the product is over all primes $\ell$ dividing $n$.
  In particular, $E^\comp_\ell$ is $n$-complete.
\end{lemma}
\begin{proof}
  First, $f$ is an $n$-equivalence if and only if $f$ is an $\ell$-equivalence 
  for all $\ell | n$.
  Indeed, consider the devissage fiber sequence 
  $f \sslash \ell \to f \sslash n \to f \sslash \frac{n}{\ell}$,
  and conclude by induction.
  In particular, $(-)^\comp_n$ inverts all $\ell$-equivalences,
  therefore there is a canonical functor $(-)^\comp_n \to \bigoplus_{\ell|n}(-)^\comp_\ell$.

  We show that this map is an $n$-equivalence.
  Indeed, since the map $E \to E^\comp_n$ is an $n$-equivalence,
  it suffices to show that $E \to \bigoplus_{\ell|n} E^\comp_\ell$ is an $n$-equivalence,
  or equivalently an $p$-equivalence for all $p | n$.
  This now follows since $(E^\comp_p)^\comp_p \wequi E^\comp_p$,
  and $(E^\comp_\ell)^\comp_p = 0$ for all $\ell \neq p$ 
  since $p$ is invertible on $E^\comp_\ell$.

  We finish the proof by showing that any object of the form $\bigoplus_{\ell|n}E^\comp_\ell$
  is $n$-complete. Indeed, it suffices to show that $E^\comp_\ell$ is $n$-complete 
  for all $\ell | n$, which is clear since it is local for all $\ell$-equivalences,
  and every $n$-equivalence is in particular an $\ell$-equivalence.

  For the last claim, note that $n$-complete objects are stable under retracts.
\end{proof}

Let $\scr P$ be a (not necessarily finite) set of primes and $\scr X$ be an $\infty$-topos.
Suppose that $\scr X$ admits a locally finite-dimensional cover, 
as in \cite[Definition 5.1]{mattis2024unstablearithmeticfracturesquares}.

\begin{definition} \label{def:unstable-n:def-n-equiv}
  Write $L_{\scr P} \colon \scr X \to \scr X$ 
  for the Bousfield-localization at $\scr P$-equivalences,
  i.e., at morphisms $f \colon X \to Y$ 
  such that $\Sigma^\infty_+ f \sslash \ell$ is an equivalence for all $\ell \in \scr P$.
  We call this functor the \emph{unstable $\scr P$-completion functor}.

  If $n \in \Z \setminus 0$ is a nonzero integer, 
  we write $\scr P_n$ for the set of prime divisors of $n$, and $L_n \coloneqq L_{\scr P_n}$ 
  for the \emph{unstable $n$-completion functor}.
\end{definition}

In particular, we have the following:
\begin{example}
  If $n = \ell^k$ is a prime power (for $k > 0$),
  then $L_n = L_\ell$ is just unstable $\ell$-completion.
  If $n = 1$, then $L_n = *$, as then any morphism is an $n$-equivalence.
  Similarly, $L_\emptyset = *$, since $\scr P_1 = \emptyset$.
\end{example}

The following is straightforward:
\begin{lemma} \label{lem:unstable-n:n-equivalence}
  Let $n \in \Z \setminus 0$ and $f \colon X \to Y \in \scr X$.
  Then $f$ is an $n$-equivalence (i.e.\ $\Sigma^\infty_+ f \sslash n$ is an equivalence) 
  if and only if $f$ is an $\ell$-equivalence for all $\ell|n$ (i.e.\ $L_{\scr P_n} f$ is an equivalence).
\end{lemma}
\begin{proof}
  This follows since it is true stably, cf.\ \cref{lem:stable-bousfield:n-completion}.
\end{proof}

Our main goal is to show that if $X$ is nilpotent,
then $L_{\scr P} X$ splits into the product of $L_\ell X$.
First, we get a natural transformation between the two functors in question:
\begin{lemma}
  There is a canonical natural transformation 
  $\varphi \colon L_{\scr P} \to \prod_{\ell \in \scr P} L_\ell$.

  If $X \in \scr X_*$ is connected,
  then so is $L_{\scr P} X$.
  In particular, we obtain by adjunction a natural morphism 
  $\widetilde{\varphi}_X \colon L_{\scr P} X \to \tau_{\ge 1} \prod_{\ell \in \scr P} L_\ell X$.
\end{lemma}
\begin{proof}
  For the first statement, it suffices to give natural transformations $L_{\scr P} \to L_\ell$ 
  if $\ell \in \scr P$.
  Since by definition any $\scr P$-equivalence 
  is also an $\ell$-equivalence, we see that $L_\ell$ inverts 
  all $\scr P$-equivalences,
  and hence the canonical map $\id{} \to L_\ell$ factors over $L_{\scr P}$.

  Now to prove the second statement, let $X \in \scr X_*$ be connected.
  We must show the same is true for $L_{\scr P} X$.
  If $\scr P = \emptyset$ then $L_{\scr P} X = *$, which is certainly connected.
  Otherwise, there exists $\ell \in \scr P$.
  By definition, $X \to L_{\scr P} X$ is an $\ell$-equivalence,
  and thus an equivalence on $\pi_0$ by \cite[Lemma 3.12]{mattis2024unstablepcompletionmotivichomotopy}.
  Hence, again $L_{\scr P} X$ is connected.
\end{proof}

We can now state the main result:
\begin{theorem} \label{lem:unstable-n:main-thm}
  Let $X \in \scr X_*$ be nilpotent.
  Then $\widetilde{\varphi}_X$ is an equivalence.

  If $\scr P$ is finite, then $\prod_{\ell \in \scr P} L_\ell X$ is connected 
  and in particular $\varphi_X$ is an equivalence.

  As a special case, if $n \in \Z \setminus 0$,
  then $L_n X \wequi \prod_{\ell | n} L_\ell X$.
\end{theorem}

For the proof we need some preparations. The following definition will be useful.
\begin{definition}
  Let $X \in \scr X_*$ be nilpotent. A \emph{Postnikov refinement} $(X_n, E_n)_n$
  is a tower of nilpotent sheaves $(X_n)_n$ under $X$,
  that fit into fiber sequences $X_{n+1} \to X_n \to \Omega^\infty_* E_n$ 
  for some $E_n \in \Sigma^{k_i} \Sp(\scr X)^{\heartsuit}$ with $k_i \ge 2$,
  such that moreover $X_0 = *$,
  $X \wequi \lim_n X_n$, and $k_i \to \infty$ as $i \to \infty$.
\end{definition}

\begin{remark}
  If $X \in \scr X_*$ is nilpotent, then there always exists a Postnikov refinement,
  cf.\ \cite[Lemma A.15]{mattis2024unstablepcompletionmotivichomotopy}.
  Moreover, if $(X_n, E_n)_n$ is a Postnikov refinement of $X$,
  then $(X_n)_n$ is a highly connected tower.
\end{remark}

\begin{lemma} \label{lem:unstable-n:nilpotent-refinement-completion}
  Let $X \in \scr X_*$ be nilpotent, and let $(X_n, E_n)_n$ be a Postnikov refinement of $X$.
  Then the following holds:
  \begin{enumerate}
    \item $\tau_{\ge 1} \prod_{\ell \in \scr P} L_\ell \Omega^\infty E_n \wequi \tau_{\ge 1} \Omega^\infty \prod_{\ell \in \scr P} (E_n)^\comp_\ell$.
    \item $\tau_{\ge 1} \prod_{\ell \in \scr P} L_\ell X_{n+1} \wequi \tau_{\ge 1} \fib(\tau_{\ge 1} \prod_{\ell \in \scr P} L_\ell X_{n} \to \tau_{\ge 1} \prod_{\ell \in \scr P} L_\ell \Omega^\infty E_n)$ for every $n$.
  \end{enumerate}
\end{lemma}
\begin{proof}
  For (1) we have
  \begin{equation*}
    \tau_{\ge 1} \prod_{\ell \in \scr P} L_\ell \Omega^\infty E_n \wequi \tau_{\ge 1} \prod_{\ell \in \scr P} \tau_{\ge 1} \Omega^\infty (E_n)^\comp_{\ell} \wequi \tau_{\ge 1} \Omega^\infty \prod_{\ell \in \scr P} (E_n)^\comp_{\ell},
  \end{equation*}
  using \cite[Lemma 3.17]{mattis2024unstablepcompletionmotivichomotopy} and \cite[Lemma 4.2]{mattis2024unstablearithmeticfracturesquares}.
  
  For (2) we get 
  \begin{equation*}
    \tau_{\ge 1} \prod_{\ell \in \scr P} L_\ell X_{n+1} \wequi \tau_{\ge 1} \prod_{\ell \in \scr P} \tau_{\ge 1} \fib(L_\ell X_n \to L_\ell \Omega^\infty E_n) \wequi \tau_{\ge 1} \fib(\tau_{\ge 1} \prod_{\ell \in \scr P} L_\ell X_n \to \tau_{\ge 1} \prod_{\ell \in \scr P} L_\ell \Omega^\infty E_n),
  \end{equation*}
  using \cite[Proposition 3.20]{mattis2024unstablepcompletionmotivichomotopy} and \cite[Lemma 4.2]{mattis2024unstablearithmeticfracturesquares}.
\end{proof}

\begin{lemma} \label{lem:unstable-n:nilpotence-of-product}
  Let $X \in \scr X_*$ be nilpotent, and let $(X_n, E_n)_n$ be a Postnikov refinement of $X$.
  Then for every $n \ge 0$ the sheaf $\tau_{\ge 1} \prod_{\ell \in \scr P} L_\ell X_n$ is nilpotent.
\end{lemma}
\begin{proof}
  We show the claim by induction on $n$, the case $n = 0$ holds trivially.
  For any $n \ge 0$ we get 
  $\tau_{\ge 1} \prod_{\ell \in \scr P} L_\ell X_{n+1} \wequi \tau_{\ge 1} \fib(\tau_{\ge 1} \prod_{\ell \in \scr P} L_\ell X_n \to \tau_{\ge 1} \prod_{\ell \in \scr P} L_\ell \Omega^\infty E_n)$
  from \cref{lem:unstable-n:nilpotent-refinement-completion} (2).
  The first term on the right is nilpotent by induction, whereas the second is by combining
  \cref{lem:unstable-n:nilpotent-refinement-completion} (1) with \cite[Lemma A.11]{mattis2024unstablepcompletionmotivichomotopy}.
  Hence, we conclude using \cite[Lemma A.12]{mattis2024unstablepcompletionmotivichomotopy}. 
\end{proof}

\begin{lemma} \label{lem:unstable-n:colimit-generation-n-connective}
  Let $\scr U$ be an $\infty$-topos generated under colimits by a set $W \subset \scr U$.
  Let $X \in \scr U_*$ be an object and $n \in \N$ such that $X(w)$ is an $n$-connective space for every $w \in W$.
  Then $X$ is $n$-connective in $\scr U_*$.
\end{lemma}
\begin{proof}
  Consider the adjunction $L \colon \scr P(W) \adj \scr U \noloc i$,
  where $L$ is the left Kan extension of $W \hookrightarrow \scr U$,
  and $i$ is the restricted Yoneda.
  Since $W$ generates $\scr U$ under colimits, it follows that $L$ is essentially surjective,
  and $i$ is fully faithful, so that $Li(X) \wequi X$.
  Moreover, note that since $i$ preserves the terminal object (as a right adjoint),
  we also get $L(*) = Li(*) = *$.
  By assumption on $X$, $i(X)$ is an $n$-connective object of $\scr P(W)_*$ (note that in a presheaf 
  topos $n$-connective objects are exactly those presheaves such that they are $n$-connective on sections).
  
  It now suffices to show that any functor $F \colon \scr V \to \scr U$ between $\infty$-topoi 
  that preserves colimits and the final object also preserves $n$-connective objects.
  Indeed, for this note first that $F$ preserves suspensions, i.e.\ $\Sigma^n F V \wequi F \Sigma^n V$,
  since suspensions are iterated pushouts of the form $* \leftarrow V \rightarrow *$.
  Thus, the claim follows since any $n$-connective object in $\scr V_*$ can be written as a colimit 
  of objects of the form $\Sigma^n V$, cf.\ \cite[Proposition 2.38]{mattis2025etale}.
\end{proof}

\begin{lemma} \label{lem:unstable-n:infinite-product-connectivity}
  Let $\scr U$ be an $\infty$-topos locally of homotopy dimension $\le N$ for some $N \in \N$.
  Let $k > N$ and $X_\ell \in \scr U_*$ be a $k$-connective object for every $\ell \in \scr P$.
  Then $\prod_{\ell \in \scr P} X_\ell$ is $(k - N)$-connective.
\end{lemma}
\begin{proof}
  It suffices to show that $(\prod_{\ell \in \scr P} X_\ell)(U)$ is $(k-N)$-connective 
  for every $U \in \scr U$ of homotopy dimension $\le N$ (by \cref{lem:unstable-n:colimit-generation-n-connective}, 
  since they generate the topos under colimits).
  By \cite[Lemma 7.2.1.7]{highertopoi} (applied to the slice topos $\scr U_{/U}$), 
  we have that $X_\ell(U)$ is $(k-N)$-connective.
  This immediately implies the lemma since evaluation of sheaves commutes with limits,
  and arbitrary products of $(k-N)$-connective spaces are $(k-N)$-connective.
\end{proof}

\begin{remark}
  In the last lemma, if $\scr P$ is finite, then of course 
  $\prod_{\ell \in \scr P} X_\ell$ is even $k$-connective.
\end{remark}

\begin{lemma} \label{lem:unstable-n:product-highly-connected}
  Let $X \in \scr X_*$ be nilpotent, and let $(X_n, E_n)_n$ be a Postnikov refinement of $X$.
  The tower $(\tau_{\ge 1} \prod_{\ell \in \scr P} L_\ell X_n)_n$ is locally highly connected (subordinate to any highly connected cover of $\scr X$), 
  with limit $\tau_{\ge 1} \prod_{\ell \in \scr P} L_\ell X$.
\end{lemma}
\begin{proof}
  Choose a highly connected cover $\scr U = \{ p_i^* \colon \scr X \to \scr U_i \}$.
  For the first statement, it suffices to show that $(\prod_{\ell \in \scr P} L_\ell X_n)_n$ is locally highly connected.
  Since the $p_i^*$ commute with both $p$-completion and limits by \cite[Lemma 6.10 and Lemma 5.2]{mattis2024unstablearithmeticfracturesquares},
  we may assume that $\scr X$ itself is of homotopy dimension $\le N$ for some $N \in \N$,
  and that $(X_n)_n$ is a highly connected tower, 
  and our goal is to show that also $(\prod_{\ell \in \scr P} L_\ell X_n)_n$ is highly connected.
  So let $k \in \N$. Since $(X_n)_n$ is highly connected,
  there exists $L \ge 0$ such that for all $m \ge L$ the fiber $\fib(X_m \to X_L)$ 
  is $(k + 2N + 3)$-connective. We have to find $M \ge 0$ such that for all $m \ge M$ 
  the object $\prod_{\ell \in \scr P} \fib(L_\ell X_m \to L_\ell X_M)$ is $(k+1)$-connective
  (since we can commute the product with the fiber).
  Setting $M = L$, it follows from \cite[Lemma 2.11]{mattis2025etale} that $\fib(L_\ell X_m \to L_\ell X_M)$ is 
  $(k+1+N)$-connective for every $\ell \in \scr P$, and hence from \cref{lem:unstable-n:infinite-product-connectivity} that $\prod_{\ell \in \scr P} \fib(L_\ell X_m \to L_\ell X_M)$ 
  is $(k+1)$-connective, proving the claim.

  It is left to compute the limit of the tower, where we can again go to the highly connected cover.
  We have $\lim_n L_\ell X_n \wequi L_\ell X$ by \cite[Proposition 2.13]{mattis2025etale},
  since $(X_n)_n$ is highly connected.
  Hence, we get $\tau_{\ge 1} \lim_n \tau_{\ge 1} \prod_{\ell \in \scr P} L_\ell X_n \wequi \tau_{\ge 1} \lim_n \prod_{\ell \in \scr P} L_\ell X_n \wequi \tau_{\ge 1} \prod_{\ell \in \scr P} L_\ell X$,
  using \cite[Lemma 4.2]{mattis2024unstablearithmeticfracturesquares}.
  To conclude, we show that $\lim_n \tau_{\ge 1} \prod_{\ell \in \scr P} L_\ell X_n$ 
  is already connected.
  This holds since $(\tau_{\ge 1} \prod_{\ell \in \scr P} L_\ell X_n)_n$ is highly connected by the first part 
  of the proof,
  and thus $\pi_0(\lim_n \tau_{\ge 1} \prod_{\ell \in \scr P} L_\ell X_n) = *$ 
  by \cite[Lemma 2.4]{mattis2025etale}.
\end{proof}

\begin{lemma} \label{lem:unstable-n:canonical-map-n-equiv}
  Let $X \in \scr X_*$ be nilpotent.
  Then $\widetilde{\varphi}_X \colon L_{\scr P} X \to \tau_{\ge 1} \prod_{\ell \in \scr P} L_\ell X$
  is a $\scr P$-equivalence.
\end{lemma}
\begin{proof}
  We have to see that $\widetilde{\varphi}_X$ is an $p$-equivalence for every $p \in \scr P$.
  By definition, we know that $X \to L_{\scr P} X$ is a $p$-equivalence,
  hence, by $2$-out-of-$3$ it suffices to show that $X \to \tau_{\ge 1} \prod_{\ell \in \scr P} L_\ell X$ is a $p$-equivalence.
  Note that $\tau_{\ge 1} \prod_{\ell \in \scr P} L_\ell X \wequi L_p X \times \tau_{\ge 1} \prod_{p \neq \ell \in \scr P} L_\ell X$.
  Since $L_p$ preserves finite products \cite[Lemma 3.15]{mattis2024unstablepcompletionmotivichomotopy}
  and since $L_p L_p \wequi L_p$,
  it suffices to show that $L_p \tau_{\ge 1} \prod_{p \neq \ell \in \scr P} L_\ell X = *$.

  Choose a Postnikov refinement $(X_n, E_n)_n$ of $X$.
  First, for every $n$ we have
  \begin{equation*}
    \tau_{\ge 1} \prod_{p \neq \ell \in \scr P} L_\ell \Omega^\infty E_n \wequi \Omega^\infty \tau_{\ge 1} \prod_{p \neq \ell \in \scr P} (E_n)^\comp_{\ell},
  \end{equation*}
  using \cref{lem:unstable-n:nilpotent-refinement-completion} (1).
  In order to see that the $p$-completion of this vanishes, 
  using \cite[Lemma 3.16]{mattis2024unstablepcompletionmotivichomotopy} it is enough to show that $(\tau_{\ge 1} \prod_{p \neq \ell \in \scr P} (E_n)^\comp_{\ell}) \sslash p = 0$.
  This holds since $p$ is invertible on $(E_n)^\comp_{\ell}$ for all $\ell \neq p$.
  
  We now show inductively that $L_p \tau_{\ge 1} \prod_{p \neq \ell \in \scr P} L_\ell X_n = *$, the case $n = 0$ holds trivially.
  This follows immediately from \cref{lem:unstable-n:nilpotent-refinement-completion} (2), the inductive hypothesis 
  and the case for $E_n$ discussed above.

  We finish the proof by showing that $L_p \tau_{\ge 1} \prod_{p \neq \ell \in \scr P} L_\ell X = *$.
  The tower
  $(\tau_{\ge 1} \prod_{p \neq \ell \in \scr P} L_\ell X_n)_n$ is locally highly connected,
  with limit $\tau_{\ge 1} \prod_{p \neq \ell \in \scr P} L_\ell X$, cf.\ \cref{lem:unstable-n:product-highly-connected}.
  Moreover, for every $n \ge 0$ the sheaf $\tau_{\ge 1} \prod_{p \neq \ell \in \scr P} L_\ell X_n$
  is nilpotent by \cref{lem:unstable-n:nilpotence-of-product}.
  Hence, using \cite[Proposition 2.13]{mattis2025etale} and the case for $X_n$ discussed above, we get 
  \begin{equation*}
    L_p \tau_{\ge 1} \prod_{p \neq \ell \in \scr P} L_\ell X \wequi \lim_n L_p \tau_{\ge 1} \prod_{p \neq \ell \in \scr P} L_\ell X_n = *. \qedhere
  \end{equation*}
\end{proof}

We are now ready to prove the main theorem of this section.
\begin{proof}[Proof of \cref{lem:unstable-n:main-thm}]
  In \cref{lem:unstable-n:canonical-map-n-equiv} we have seen that $\widetilde{\varphi}_X \colon L_{\scr P} X \to \tau_{\ge 1} \prod_{\ell \in \scr P} L_\ell X$ 
  is a $\scr P$-equivalence.
  Since the left-hand side is clearly $\scr P$-complete,
  it suffices to show that the right-hand side is $\scr P$-complete.
  Since $\scr P$-equivalences are $\ell$-equivalences, limits of $\ell$-complete objects (for varying $\ell$) are $\scr P$-complete.
  Writing $\tau_{\ge 1} = \fib(\id \to \pi_0)$, this reduces to showing that $L_\ell X$ is $\ell$-complete (which is obvious) and that discrete sheaves are $\ell$-complete (which holds by \cite[Lemma 3.13]{mattis2024unstablepcompletionmotivichomotopy}).

  For the second statement, i.e. the case where $\scr P$ is finite, we just note that $\prod_{\ell \in \scr P} L_\ell X$ 
  is connected, since each of the finitely many factors is (see \cite[Lemma 3.12]{mattis2024unstablepcompletionmotivichomotopy} 
  for a proof that $L_\ell X$ is connected).
\end{proof}

Using the main theorem, we can now extend the results about unstable 
$p$-completion from \cite{mattis2024unstablepcompletionmotivichomotopy,mattis2024unstablearithmeticfracturesquares,mattis2025etale} 
to the unstable $\scr P$-completion functor.
For $n \in \Z \setminus 0$, specializing to $\scr P = \scr P_n$ we also 
obtain analogous results about the unstable $n$-completion functor.

\begin{lemma} [Infinite loop spaces] \label{lem:unstable-n:infinite-loop-space}
  Let $E \in \Sp(\scr X)_{\ge 1}$.
  Then $L_{\scr P} \Omega^\infty E \wequi \tau_{\ge 1} \Omega^\infty \prod_{\ell \in \scr P} E^\comp_\ell$.
  As a special case, we have $L_n \Omega^\infty E \wequi \tau_{\ge 1} \Omega^\infty E^\comp_n$.
\end{lemma}
\begin{proof}
  We calculate 
  \begin{equation*}
    L_{\scr P} \Omega^\infty E \wequi \tau_{\ge 1} \prod_{\ell \in \scr P} L_\ell \Omega^\infty E \wequi \tau_{\ge 1} \prod_{\ell \in \scr P} \tau_{\ge 1} \Omega^\infty E^\comp_\ell \wequi \tau_{\ge 1} \Omega^\infty \prod_{\ell \in \scr P} E^\comp_\ell,
  \end{equation*}
  where we used \cref{lem:unstable-n:main-thm} in the first equivalence,
  \cite[Lemma 3.17]{mattis2024unstablepcompletionmotivichomotopy} in the second equivalence,
  and that products commute with $\Omega^\infty$ in the third equivalence.

  The last claim holds since $\prod_{\ell \in \scr P_n} E^\comp_\ell = \bigoplus_{\ell \in \scr P} E^\comp_\ell \wequi E^\comp_n$
  by \cref{lem:stable-bousfield:n-completion}.
\end{proof}

\begin{lemma} [Fiber sequences] \label{lem:unstable-n:fiber-sequence}
  Let $F \to X \to Y$ be a fiber sequence of nilpotent sheaves in $\scr X_*$.
  Then $L_{\scr P} F \wequi \tau_{\ge 1} \fib(L_{\scr P} X \to L_{\scr P} Y)$.
  As a special case, we have $L_n F \wequi \tau_{\ge 1} \fib(L_n X \to L_n Y)$.
\end{lemma}
\begin{proof}
  We calculate 
  \begin{align*}
    L_{\scr P} F & \wequi \tau_{\ge 1} \prod_{\ell \in \scr P} L_\ell F  \wequi \tau_{\ge 1} \prod_{\ell \in \scr P} \tau_{\ge 1} \fib(L_\ell X \to L_\ell Y) \\
    & \wequi \tau_{\ge 1} \fib(\tau_{\ge 1} \prod_{\ell \in \scr P} L_\ell X \to \tau_{\ge 1} \prod_{\ell \in \scr P} L_\ell Y)   \wequi \tau_{\ge 1} \fib(L_{\scr P} X \to L_{\scr P} Y), 
  \end{align*}
  where we used \cref{lem:unstable-n:main-thm} in the first and last equivalence,
  \cite[Proposition 2.8]{mattis2025etale} in the second equivalence,
  and that products commute with fibers in the third equivalence.
\end{proof}

\begin{lemma} [Highly connected towers] \label{lem:unstable-n:tower}
  Let $(X_i)_i$ be a highly connected tower,
  where every $X_i \in \scr X_*$ is nilpotent.
  Then $L_{\scr P} \lim_i X_i \wequi \lim_i L_{\scr P} X_i$.
  As a special case, we have $L_n \lim_i X_i \wequi \lim_i L_n X_i$.
\end{lemma}
\begin{proof}
  First note that $\lim_i X_i$ is still nilpotent.
  This follows basically from \cref{lem:sheaves:highly-connected},
  as then $\ul\pi_1(X) \wequi \ul\pi_1(X_N)$ for some $N \gg 0$, and hence is nilpotent by assumption,
  and similarly, the action of $\ul\pi_1(X)$ on $\ul\pi_n(X)$ is equivalent 
  to the action of $\ul\pi_1(X_M)$ on $\ul\pi_n(X_M)$ for some $M \gg 0$, which is nilpotent by assumption.
  We calculate 
  \begin{equation*}
    L_{\scr P} \lim_i X_i \wequi \tau_{\ge 1} \prod_{\ell \in \scr P} L_\ell \lim_i X_i \wequi \tau_{\ge 1} \prod_{\ell \in \scr P} \lim_i L_\ell X_i \wequi \tau_{\ge 1} \lim_i \tau_{\ge 1} \prod_{\ell \in \scr P} L_\ell X_i \wequi \tau_{\ge 1} \lim_i L_{\scr P} X_i,
  \end{equation*}
  where we used \cref{lem:unstable-n:main-thm} in the first and last equivalence,
  \cite[Proposition 2.13]{mattis2025etale} in the second equivalence, 
  and that products commute with limits in the third equivalence.
  The lemma follows since $\lim_i \tau_{\ge 1} \prod_{\ell \in \scr P} L_\ell X_i \wequi \lim_i L_{\scr P} X_i$ is already connected,
  the proof of which is similar to the end of the proof of \cref{lem:unstable-n:product-highly-connected}.
\end{proof}

\section{Slice convergence} \label{sec:slice-convergence}
In this section we explain a minor extension of Levine's convergence theorem, allowing us to treat imperfect extensions of perfect base fields.
The argument mainly consists of observing that \cite[Proposition 6.7]{Bachmann2020TowardsConservativity} can be used as a slot-in replacement for a certain part of Levine's proof, and with this replacement, the original argument goes through in the more general setting.

\begin{proposition} \label{prop:eff-pi0-vanishing}
Let $k$ be a perfect field, $E \in \Sigma^{n,n}\SH(k)^\veff$, and $K/k$ a finitely generated field extension.
Then $\ul\pi_0(E)_*(K)$ for $* \ge -n$ is generated by transfers of elements in $\ul K_*^{MW}(L)\cdot \ul\pi_0(E)_{-n}(L)$, where $L/K$ runs through finite field extensions.
\end{proposition}
\begin{proof}
If $*=0$ and $n=1$ this is proved in \cite[Proposition 6.7]{Bachmann2020TowardsConservativity}.
By shifting, this also proves $n=1-*$ for any value of $*$.
The general case follows from this and the projection formula.
(Indeed we may as well assume that $n=0$, then apply the previous case to $\Sigma^{r+1,r+1} E$ to learn that $\ul\pi_0(\Sigma^{r+1,r+1}E)_0 \wequi \ul\pi_0(E)_{r+1}$ is generated under transfers and $\ul{K}_1^{MW}$ by $\ul\pi_0(\Sigma^{r+1,r+1}E)_{-1} \wequi \ul\pi_0(E)_{r}$; now iterate.)
\end{proof}

\begin{corollary} \label{cor:levine-convergence}
There exists a function $N\colon \N^2 \to \N$ with the following property:
For $k$ a perfect field of characteristic $p$, $m, n \ge 0$, $E \in \Sigma^{N(m,n),N(m,n)}\SH(k)^\veff$ $p$-torsion (i.e. $E[1/p]=0$), $K/k$ of transcendence degree at most $m$, $i \le n$ and $j \ge 0$, we have \[ \ul\pi_i(E)_j(K) = 0. \]
\end{corollary}
\begin{proof}
The proof is essentially the same as in \cite[Theorem 7.3]{levine2013convergence} (also recalled in \cite[Theorem 5.3]{bachmann-bott}), with two key differences.
Firstly, \cite[Theorem 5.3]{levine2013convergence} is replaced by \cref{prop:eff-pi0-vanishing}.
Secondly, in Levine's proof the Bloch--Kato conjecture is used to show that if $L/k$ has transcendence degree $\le m$, then $\ul K_i^{MW}(L)/p=0$ for $i>m$.
To conclude the proof, we supply a different argument to show that this still holds in our situation.
Let us first show that $\ul K_i^{M}(L) / p = 0$ for all $i > m$.
This can be seen as follows: There is an inclusion $\ul K_i^{M}(L) / p \hookrightarrow \Omega_L^i$ 
by the Bloch--Kato--Gabber theorem \cite[Theorem 2.1]{bloch1986padicetale}. But for $i > m$ 
we have $\Omega_L^i = \bigwedge_{l=1}^i \Omega_L^1 = 0$,
since $\Omega_L^1$ has dimension $m$ \cite[Theorem 59]{MatsumuraCA}.
We now note that modulo $p$ the powers of the fundamental ideal of $L$ 
vanish, i.e. $I^i(L) / p = 0$, for $i > m$.
Indeed, if $p \neq 2$, then this holds since $W(L)$ has $2$-power torsion (indeed, $W(L)$ is a module over $W(\F_p)$, which is of $2$-power torsion,
by the explicit calculation done in e.g.\ \cite[Example 2.1.10]{déglise2023notesmilnorwittktheory}),
and hence $p$ is invertible on $W(L)$ and therefore also on $I^i(L)$.
On the other hand, if $p = 2$, we have 
$I^{i+1}(L)/I^i(L) \wequi \ul{K}^M_i(L)/2 = 0$ for $i > m$ by the Milnor Conjecture (see e.g.\ \cite[Theorem 2.2.3]{déglise2023notesmilnorwittktheory}) and the above vanishing result for Milnor K-theory.
In particular, $I^{m+1}(L) = I^{m+2}(L) = \dots$, which implies that all of these powers of the fundamental ideal vanish 
(since their intersection is $0$ by the Arason--Pfister Hauptsatz \cite[Chapter 10, Corollary 3.2]{lam2005introduction}).
In view of the short exact sequence $0 \to I(L)^{i+1} \to \ul K_i^{MW}(L) \to \ul K_i^{M}(L) \to 0$ 
from e.g.\ \cite[Corollary 2.3.10 (1)]{déglise2023notesmilnorwittktheory}, the snake lemma, and 
the above calculations,
we thus get $\ul K_i^{MW}(L) / p \wequi \ul K_i^{M}(L) / p = 0$ for all $i > m$.
\end{proof}

\bibliographystyle{alpha}
\bibliography{bibliography}

\end{document}

%% file: preamble2.tex
\usepackage{amsmath}
\usepackage{amssymb}
\usepackage{amsthm}
\usepackage{thmtools}
\usepackage{amscd}
\usepackage{enumerate}
\usepackage[pdfusetitle,unicode,hidelinks,hypertexnames=false]{hyperref}
\usepackage{bbm}
\usepackage{etoolbox}

\usepackage[utf8]{inputenc}
\usepackage[T1]{fontenc}

\newcommand{\tomemail}{\href{mailto:tom.bachmann@zoho.com}{tom.bachmann@zoho.com}}

\newtheorem{proposition}{Proposition}[section]
\newtheorem{corollary}[proposition]{Corollary}
\newtheorem{lemma}[proposition]{Lemma}
\newtheorem{theorem}[proposition]{Theorem}

\newtheorem*{conjecture*}{Conjecture}
\newtheorem*{theorem*}{Theorem}
\newtheorem*{corollary*}{Corollary}
\newtheorem*{proposition*}{Proposition}
\newtheorem*{lemma*}{Lemma}
\theoremstyle{definition}
\newtheorem{definition}[proposition]{Definition}
\newtheorem{construction}[proposition]{Construction}

\newtheorem*{definition*}{Definition}
\newtheorem*{construction*}{Construction}
\theoremstyle{remark}
\newtheorem{remark}[proposition]{Remark}
\newtheorem*{remark*}{Remark}

\newtheorem{example}[proposition]{Example}
\newtheorem*{example*}{Example}

\newcommand{\id}{\operatorname{id}}
\newcommand{\Z}{\mathbb{Z}}

\newcommand{\N}{\mathbb{N}}
\newcommand{\Q}{\mathbb{Q}}
\newcommand{\F}{\mathbb{F}}

\let\scr=\mathcal
\let\bb=\mathbb
\newcommand{\Gm}{{\mathbb{G}_m}}

\def\A{\bb A}
\def\P{\bb P}
\def\R{\bb R}
\newcommand{\1}{\mathbbm{1}}

\newcommand{\eff}{{\text{eff}}}
\newcommand{\veff}{{\text{veff}}}

\newcommand{\SH}{\mathcal{SH}}

\DeclareMathOperator*{\colim}{colim}

\let\lim=\relax
\DeclareMathOperator*{\lim}{lim}
\def\Map{\mathrm{Map}}

\def\CAlg{\mathrm{CAlg}}

\def\PSh{\mathcal{P}}

\def\Span{\mathrm{Span}}
\def\Cat{\mathcal{C}\mathrm{at}{}}
\def\Spc{\mathcal{S}\mathrm{pc}{}}
\def\Fin{\cat F\mathrm{in}}
\def\Fun{\mathrm{Fun}}

\newcommand{\wequi}{\simeq}
\newcommand{\Mod}{\mathcal{M}\mathrm{od}}
\def\adj{\rightleftarrows}

\DeclareRobustCommand{\ul}{\underline}

\def\op{\mathrm{op}}

\let\cat=\mathrm
\def\Sm{{\cat{S}\mathrm{m}}}

\def\Nis{\mathrm{Nis}}

\def\ph{\mathord-}

\newcommand{\fib}{\operatorname{fib}}
\newcommand{\cof}{\operatorname{cof}}
\newcommand{\sslash}{/\mkern-6mu/}